\documentclass[10pt,a4paper]{amsart}
\usepackage[foot]{amsaddr}
\usepackage[lmargin=0.9in,rmargin=0.9in]{geometry}
\usepackage[T1]{fontenc}
\usepackage[utf8]{inputenc}
\usepackage[british]{babel}
\usepackage{mathtools}
\usepackage{amsthm}
\usepackage{lmodern}
\usepackage{calrsfs}
\usepackage{amssymb}
\usepackage{amsmath, calligra, mathrsfs}
\usepackage[mathscr]{euscript}
\usepackage{enumitem}
\usepackage{tikz-cd}
\usetikzlibrary{decorations.markings}
\usepackage{hyperref}
\usepackage[noabbrev,nameinlink]{cleveref}
\usepackage{rotating}
\theoremstyle{plain}
\newtheorem{thm}{Theorem}[section]
\newtheorem*{thm*}{Theorem}
\newtheorem{lem}[thm]{Lemma}
\newtheorem{prop}[thm]{Proposition}

\theoremstyle{definition}
\newtheorem{defn}[thm]{Definition}
\newtheorem{nota}[thm]{Notation}
\newtheorem{ex}[thm]{Example}

\newtheorem{hyp}[thm]{Hypothesis}

\theoremstyle{remark}
\newtheorem{rem}[thm]{Remark}
\Crefname{thm}{Theorem}{Theorems}
\Crefname{lm}{Lemma}{Lemmata}
\Crefname{prop}{Proposition}{Propositions}
\Crefname{cor}{Corollary}{Corollaries}
\Crefname{hyp}{Hypothesis}{Hypotheses}
\Crefname{q}{Question}{Questions}
\Crefname{defn}{Definition}{Definitions}
\Crefname{nota}{Notation}{Notations}
\Crefname{ex}{Example}{Examples}
\Crefname{xca}{Exercise}{Exercises}
\Crefname{rem}{Remark}{Remarks}
\Crefname{constr}{Construction}{Constructions}

\newcommand{\Q}{\mathbb{Q}}

\newcommand{\Ccal}{\mathcal{C}}

\newcommand{\Rcal}{\mathcal{R}}
\newcommand{\Scal}{\mathcal{S}}

\newcommand{\id}{\textup{id}}
\newcommand{\Hom}{\textup{Hom}}

\newcommand{\Hbb}{\mathbb{H}}

\newcommand{\conn}{\mathsf{conn}}
\newcommand{\Fact}{\mathsf{Fact}}
\newcommand*{\isoarrow}[1]{\arrow[#1,"\rotatebox{90}{\(\sim\)}"]}
\newcommand{\Var}{\mathsf{Var}}
\newcommand{\Sm}{\mathsf{Sm}}
\newcommand{\Perv}{\mathsf{Perv}}

\newcommand\blfootnote[1]{%
	\begingroup
	\renewcommand\thefootnote{}\footnote{#1}%
	\addtocounter{footnote}{-1}%
	\endgroup
}

\usepackage{changepage} 

\title{Constructing monoidal structures on fibered categories via factorizations}
\author[Luca Terenzi]{Luca Terenzi}
\address{Luca Terenzi \newline
	\indent UMPA, ENS de Lyon \newline
	\indent 46 Allée d'Italie, 69007 Lyon (France)}
\email{\normalfont\href{mailto:luca.terenzi@ens-lyon.fr}{luca.terenzi@ens-lyon.fr}}

\makeatletter
\hypersetup{
	pdfauthor={\author},
	pdftitle={\@title},
	colorlinks,
	linkcolor=[rgb]{0.2,0.2,0.6},
	citecolor=[rgb]{0.2,0.6,0.2},
	urlcolor=[rgb]{0.6,0.2,0.2}}
\makeatother

\calclayout
\begin{document}

\maketitle

\begin{abstract}
	Let $\Scal$ be a small category, and suppose that we are given two (non-full) subcategories $\Scal^{sm}$ and $\Scal^{cl}$ that generate all morphisms of $\Scal$ under composition in the same way as morphisms of quasi-projective algebraic varieties are generated by smooth morphisms and closed immersions. We show that a monoidal structure on a given $\Scal$-fibered category is completely determined by its restrictions to $\Scal^{sm}$ and $\Scal^{cl}$; in fact, any such pair of monoidal structures satisfying a natural coherence condition uniquely determines a monoidal structure over the whole of $\Scal$. The same principle applies to morphisms of $\Scal$-fibered categories and monoidality thereof.
	Under further assumptions on the two subcategories $\Scal^{sm}$ and $\Scal^{cl}$, and with suitable restrictions on the $\Scal$-fibered categories and morphisms involved, we provide a variant of the above method in which inverse images under closed immersion are partially replaced by the corresponding direct images; the latter variant is more adapted to the setting of perverse sheaves.
\end{abstract}

\tableofcontents

\blfootnote{\textit{\subjclassname}. 18D30, 18M05.}
\blfootnote{\textit{\keywordsname}. Monoidal fibered categories, factorizations, perverse sheaves.}

\blfootnote{The author acknowledges support by the GK 1821 "Cohomological Methods in Geometry" at the University of Freiburg and by the Labex Milyon at the ENS Lyon.}

\section*{Introduction}

\subsection*{Motivation and goal of the paper}

When studying a small category $\Scal$, it often proves useful to single out some distinguished classes of morphisms that one can control with relatively little effort; if these distinguished classes generate all morphisms in $\Scal$ under composition, one can hope to gain some control on the whole of $\Scal$ in this way. This principle is particularly effective in Algebraic Geometry, where results about the existence of distinguished factorizations for reasonable morphisms of schemes are among the most useful auxiliary tools in several situations. 
For the purposes of the present paper, we are interested in the abstract setting of $\Scal$-fibered categories, regarded as contravariant pseudo-functors with domain $\Scal$. There is a general factorization principle to turn collections of categories parametrized by $\Scal$ into actual $\Scal$-fibered categories: roughly speaking, one constructs inverse image functors along arbitrary morphisms of $\Scal$ by first defining inverse images for distinguished classes of morphisms and then pasting the partial constructions together. However, as the reader can guess, this idea cannot really work in practice unless the partial constructions are compatible in a precise sense.

In order to illustrate the general problem with a concrete example, consider the category $\Var_k$ of quasi-projective algebraic varieties over a field $k$. Among all morphisms in $\Var_k$ there exist two distinguished classes with particularly nice properties: the class of smooth morphisms and the class of closed immersions. It is known that every morphism $f: T \rightarrow S$ in $\Var_k$ admits some factorization of the form
\begin{equation}\label{intro:fact-f}
	f: T \xrightarrow{t} P \xrightarrow{p} S
\end{equation}
with $t$ a closed immersion and $p$ a smooth morphism. We want to use these morphisms as our basic building blocks. So suppose that we are given a collection of categories $\left\{\Hbb(S)\right\}_{S \in \Var_k}$ that we want to turn into a $\Var_k$-fibered category: this amounts to constructing inverse image functors along all morphisms of $\Var_k$ compatibly with composition in the natural way. Suppose also that we already know how to do this separately for the classes of smooth morphisms and closed immersions. Given a general morphism $f: T \rightarrow S$ in $\Var_k$, one can try to define the inverse image functor $f^*: \Hbb(S) \rightarrow \Hbb(T)$ as the composite
\begin{equation*}
	f^*: \Hbb(S) \xrightarrow{p^*} \Hbb(P) \xrightarrow{t^*} \Hbb(T)
\end{equation*}
according to a factorization of the above form \eqref{intro:fact-f}. The whole point is to show that the result is independent of the chosen factorization of $f$ and that it is compatible with composition of morphisms in $\Var_k$. Assuming to have a positive answer to the former question, let us focus on the latter.
Given a second morphism $g: S \rightarrow V$ in $\Var_k$, with factorization
\begin{equation*}
	g: S \xrightarrow{s} Q \xrightarrow{q} V
\end{equation*}
into a closed immersion $s$ followed by a smooth morphism $q$, one gets a factorization of the composite morphism $g f: T \rightarrow V$ as
\begin{equation*}
	g f: T \xrightarrow{t} P \xrightarrow{p} S \xrightarrow{s} Q \xrightarrow{q} V.
\end{equation*}
Since the latter does not quite have the correct form, one is led to choose a further factorization of the composite morphism $s p: P \rightarrow Q$ as
\begin{equation*}
	s p: P \xrightarrow{h} L \xrightarrow{l} Q 
\end{equation*}
with $h$ a closed immersion and $l$ a smooth morphism. In this way, one can finally write $g f$ as
\begin{equation*}
	g f: T \xrightarrow{ht} L \xrightarrow{lq} V
\end{equation*}
where now, by construction, $ht$ is a closed immersion while $lq$ is a smooth morphism. At this point, if one wants to compare the composite of inverse images
\begin{equation*}
	f^* g^*: \Hbb(V) \xrightarrow{q^*} \Hbb(Q) \xrightarrow{s^*} \Hbb(S) \xrightarrow{p^*} \Hbb(P) \xrightarrow{t^*} \Hbb(T)
\end{equation*}
with the inverse image of the composite
\begin{equation*}
	(g f)^*: \Hbb(V) \xrightarrow{q^*} \Hbb(Q) \xrightarrow{l^*} \Hbb(L) \xrightarrow{p^*} \Hbb(P) \xrightarrow{t^*} \Hbb(T), 
\end{equation*}
one needs to compare the two possible paths of inverse image functors along the commutative square
\begin{equation}\label{intro:square-PLSQ}
	\begin{tikzcd}
		P \arrow{r}{h} \arrow{d}{p} & L \arrow{d}{l} \\
		S \arrow{r}{s} & Q.
	\end{tikzcd}
\end{equation}
The theory of \textit{exchange structures} developed in \cite[\S~1.2]{Ayo07a} provides a rigorous method for this: it is based on the datum of additional natural isomorphisms attached to squares of the form \eqref{intro:square-PLSQ} which are required to satisfy a suitable compatibility condition with respect to vertical and horizontal concatenation of such squares.

The main goal of this paper is to adapt the factorization principle to other construction problems within the theory of $\Scal$-fibered categories: specifically, we are interested in morphisms of $\Scal$-fibered categories, monoidal structures on $\Scal$-fibered categories, and monoidal morphisms of such. Compared to the case of $\Scal$-fibered categories just discussed, these questions turn out to be somewhat more elementary, since the task becomes to construct natural transformations between already existing functors rather than to construct functors themselves: coherence of natural transformations is a property rather than an additional structure.

Our main motivation behind this work comes from the theory of perverse sheaves over quasi-projective algebraic varieties, where one wants to exploit the factorization of morphisms into closed immersions and smooth morphisms as displayed above. To fix notation, let $k$ be a subfield of the complex numbers, and consider again the category $\Var_k$ of quasi-projective $k$-varieties. To every $S \in \Var_k$ one can attach the algebraically constructible $\Q$-linear derived category $D^b_c(S,\Q)$, a tensor-triangulated category. As $S$ varies, these enjoy a full six functor formalism and, in particular, assemble into a monoidal $\Var_k$-fibered category. The theory of \cite{BBD82} yields a perverse $t$-structure on $D^b_c(S,\Q)$, whose heart is the $\Q$-linear category of \textit{perverse sheaves}, denoted $\Perv(S)$. Beilinson's classical result \cite[Thm.~1.3]{Bei87} asserts that the triangulated category $D^b_c(S,\Q)$ is canonically equivalent to $D^b(\Perv(S))$. This brings the question of how much of the six functor formalism can be reconstructed from the abelian categories $\Perv(S)$.
What makes this question non-trivial is the fact that most functors belonging to the six functor formalism are not $t$-exact with respect to the perverse $t$-structure. For instance, the usual tensor product functor
\begin{equation*}
	- \otimes - : D^b_c(S,\Q) \times D^b_c(S,\Q) \rightarrow D^b_c(S,\Q)
\end{equation*}
does not respect the perverse $t$-structure unless $\dim(S) = 0$. As a remedy, in our paper \cite{Ter24Fib} we showed that the entire structure of the constructible derived categories as a monoidal $\Var_k$-fibered category can be rephrased in terms of the external tensor product functors
\begin{equation*}
	- \boxtimes - : D^b_c(S_1,\Q) \times D^b_c(S_2,\Q) \rightarrow D^b_c(S_1 \times S_2,\Q),
\end{equation*}
which are always $t$-exact (see \cite[Prop.~4.2.8]{BBD82}). The usual associativity, symmetry, and (to some extent) unitarity properties of the single tensor product functors can be expressed in terms of $t$-exact functors, so they can be tested directly on the level of perverse sheaves.
However, a fundamental part of the structure of monoidal $\Var_k$-fibered category is encoded in the monoidality isomorphisms for inverse image functors. Since the inverse image under a general morphism in $\Var_k$ does not enjoy any $t$-exactness property, one cannot hope to express monoidality purely in terms of $t$-exact functors just by replacing the usual tensor product by the external one. On the other hand, it turns out that perverse sheaves are sufficiently well-behaved with respect to smooth morphisms and closed immersions: for a smooth morphism $p: P \rightarrow S$, the inverse image $p^*$ is $t$-exact up to appropriate shiftings (see \cite[\S~4.2.4]{BBD82}), while, for a closed immersion $t: T \rightarrow P$, it is the direct image $t_*$ that is $t$-exact (see \cite[Cor.~4.1.3]{BBD82}). This suggests that, when working with the external tensor product, it should be possible to reconstruct monoidality under a general morphism $f: T \rightarrow S$ in terms of natural isomorphisms of $t$-exact functors, by means of a factorization of the form \eqref{intro:fact-f}. This turns out to be the case, but with a caveat: rephrasing monoidality of inverse images under closed immersions as a property of the corresponding direct images is not a purely formal step, since it crucially requires an important property of the six functor formalism known as \textit{localization}. 

What we have just explained in the setting of perverse sheaves could be repeated, more generally, for systems of triangulated categories satisfying the six functor formalism and endowed with reasonable perverse $t$-structures. The most prominent example is M. Saito's celebrated theory of \textit{mixed Hodge modules}, developed in \cite{Sai90}. In this case, the triangulated categories of coefficients are actually defined as the bounded derived categories of the perverse hearts, and their functoriality is constructed starting from the level of abelian categories. Our work enlightens important category-theoretic features of Saito's construction, especially around the monoidal structure on mixed Hodge modules, which are not addressed systematically in \cite{Sai90}. 

We ultimately intend to apply our abstract results to the emerging theory of \textit{perverse Nori motives}, a further motivic refinement of Saito's theory, recently introduced F. Ivorra and S. Morel in \cite{IM19}. As in the case of mixed Hodge modules, one has a system of abelian categories of enhanced perverse sheaves and one wants to equip their bounded derived categories with a six functor formalism. In this new setting, however, the perverse hearts are defined by a purely categorical procedure, which forces one to be very careful about the formal properties of functors and natural transformations relating them. In our article \cite{Ter24Nori}, exploiting the general techniques developed here, we construct a monoidal structure on perverse Nori motives by working directly with the external tensor product on the perverse hearts. The only part of the monoidal structure falling out of the application range of the present paper is the unit constraint; in \cite{Ter24Emb} we develop further reduction techniques to treat this specific problem.


\subsection*{Main results}

Both for sake of notational clarity and in order not to restrict the application range of our results unnecessarily, we formulate everything in the abstract language of monoidal fibered categories. Throughout the paper, we work over a fixed base category $\Scal$; for our constructions involving the external tensor product, we need to assume that $\Scal$ admits binary products. 
Since our main results are motivated by geometric situations, in the course of the paper we are led to introduce more structure and assumptions on the base category $\Scal$: we suppose to be given two non-full subcategories $\Scal^{sm}$ and $\Scal^{cl}$ mimicking the behavior of smooth morphisms and closed immersions in $\Var_k$; for part of our constructions, we also need to assume the existence of special morphisms in $\Scal^{sm}$ mimicking the behavior of open immersions. 

Our final goal is to obtain a factorization method adapted to the theory of perverse sheaves. The argument developed in \cite[\S~4.2]{IM19} has served both as an inspiration and as a model for our abstract results. We aim at expressing various structures on $\Scal$-fibered categories in terms of inverse images under $\Scal^{sm}$ and direct images under $\Scal^{cl}$. In order to make this idea work, we need to introduce suitable assumptions on the $\Scal$-fibered categories considered: roughly speaking, we want to reproduce the setting of stable homotopy $2$-functors from \cite{Ayo07a,Ayo07b} for what concerns the localization axiom, but avoiding any use of the homotopy and stability axioms. We do this by introducing the notion of \textit{localic} triangulated $\Scal$-fibered category, which takes our setting close to that of geometric pullback formalisms introduced in \cite{DrewGal}.

Even if we are mainly interested in applying the factorization method to the construction of monoidal $\Scal$-fibered categories and monoidal morphisms, it is useful to start from the case of morphisms of $\Scal$-fibered categories: first, this case is interesting in its own, and also implicitly needed in the study of monoidality for morphisms; second, the argument for monoidal $\Scal$-fibered categories is just a notationally more intricate variant of the argument for morphisms of $\Scal$-fibered categories. It seems natural to divide our main construction into two main steps: in the first one, we obtain a general factorization method based on inverse images under $\Scal^{sm}$ and under $\Scal^{cl}$; in the second one, we explain how to partially replace inverse images under $\Scal^{cl}$ with the corresponding direct images in the localic case. 
We describe these two constructions axiomatically via the notions of \textit{$\Scal$-skeleton} and of \textit{$\Scal$-core} of a morphism, respectively. While the first notion makes sense in general, the second one works only for certain morphisms between localic $\Scal$-fibered categories that it seems natural to call \textit{localic}: this covers, for example, all morphisms of stable homotopy $2$-functors in the sense of \cite{Ayo10}. The skeleton of a morphism of $\Scal$-fibered categories consists of its restrictions to the underlying $\Scal^{sm}$-fibered and $\Scal^{cl}$-fibered categories, subject to a natural coherence condition with respect to squares of the form \eqref{intro:square-PLSQ}. For the core of a localic morphism, the only difference is that one regards the underlying morphism of $\Scal^{cl}$-fibered categories as a morphism of $\Scal^{cl,op}$-fibered categories by replacing inverse images under $\Scal^{cl}$ by the corresponding direct images; in this case, the coherence condition between the two restrictions is slightly less obvious - the rough idea being to decompose a square of the form \eqref{intro:square-PLSQ} into a Cartesian square and two triangles and analyze the three pieces separately.
After this preamble, we can state our first main result as follows:
\begin{thm*}[\Cref{prop_mor-skel} and \Cref{prop_mor-core}]
	Let $R: \Hbb_1 \rightarrow \Hbb_2$ be a morphism of $\Scal$-fibered categories. Then the following statements hold:
	\begin{enumerate}
		\item Suppose that $\Scal$ satisfies \Cref{hyp_skel}. Then the morphism $R$ is completely determined by its restrictions to the underlying $\Scal^{sm}$-fibered and $\Scal^{cl}$-fibered categories. Conversely, such a pair of restrictions gives rise to a morphism over $\Scal$ if and only if it defines an $\Scal$-skeleton.
		\item Suppose that $\Scal$ satisfies the stronger \Cref{hyp_core}. In addition, suppose that the $\Scal$-fibered categories $\Hbb_1$ and $\Hbb_2$ are localic and that $R$ is a localic morphism. 
		
		Then the morphism $R$ is completely determined by its restrictions to the underlying $\Scal^{sm}$-fibered and $\Scal^{cl,op}$-fibered categories. Conversely, such a pair of restrictions gives rise to a (localic) morphism over $\Scal$ if and only if it defines an $\Scal$-core.
	\end{enumerate}
\end{thm*}

We treat the case of monoidal structures on $\Scal$-fibered categories by applying the same method to monoidality isomorphisms: first of all, we show that these are always determined by those under $\Scal^{sm}$ and under $\Scal^{cl}$; then, we explain how monoidality isomorphisms under $\Scal^{cl}$ can be partially rephrased in terms of direct images in the localic setting. We formulate this precisely via the notions of \textit{external tensor skeleton} and of \textit{external tensor core}, respectively; the coherence conditions with respect to squares of the form \eqref{intro:square-PLSQ} needed in the two cases are analogous to those for $\Scal$-skeleta and $\Scal$-cores.
As the adjective "external" suggests, we have chosen to work systematically with the external tensor product in place of the usual internal tensor product; this is legitimate in view of our work \cite{Ter24Fib}, where the theory of \textit{external tensor structures} is developed in detail.
Compared to the internal tensor product, the external tensor product is better suited to factorization questions, for at least two reasons: first, because it often enjoys better adjointability properties (see \Cref{rem:whatifITS}); second, because it allows us to interpret monoidal structures on $\Scal$-fibered as suitable $2$-parameter families of morphisms of $\Scal$-fibered categories, so that we can deduce our results for monoidal structures from the analogous results for morphisms quite directly. For sake of brevity, we will not make such an interpretation explicit here; the interested reader can obtain it with little difficulty. Let us mention that the notion of external tensor core is adapted to external tensor structures satisfying suitable adjointability properties analogous to those of localic $\Scal$-fibered categories and localic morphisms: we call them \textit{localic} external tensor structures. This said, we can state our second main result as follows:
\begin{thm*}[\Cref{thm_ETSskel} and \Cref{thm_ETScore}]
	Let $\Hbb$ be an $\Scal$-fibered category endowed with an external tensor structure $(\boxtimes,m)$. Then the following statements hold:
	\begin{enumerate}
		\item Suppose that $\Scal$ satisfies \Cref{hyp_skel}. Then the external tensor structure $(\boxtimes,m)$ is completely determined by its restrictions to the underlying $\Scal^{sm}$-fibered and $\Scal^{cl}$-fibered categories. Conversely, such a pair of restrictions gives rise to an external tensor structure over $\Scal$ if and only if it defines an external tensor skeleton.
		\item Suppose that $\Scal$ satisfies the stronger \Cref{hyp_core}. In addition, suppose that the $\Scal$-fibered category $\Hbb$ is localic and that $(\boxtimes,m)$ is localic. 
		
		Then the external tensor structure $(\boxtimes,m)$ is completely determined by its restrictions to the underlying $\Scal^{sm}$-fibered and $\Scal^{cl,op}$-fibered categories. Conversely, such a pair of restrictions gives rise to a (localic) external tensor structure over $\Scal$ if and only if it defines an external tensor core.
	\end{enumerate}
\end{thm*}
 
Finally, we introduce analogous notions and obtain analogous results about the monoidality of morphisms between monoidal $\Scal$-fibered categories; again, we formulate everything in terms of external tensor structures. We can state our third main result as follows:

\begin{thm*}[\Cref{prop_mor_ETSkel} and \Cref{thm:ETScore-mor}]
	For $j = 1,2$ let $\Hbb_j$ be an $\Scal$-fibered category endowed with an external tensor structure $(\boxtimes_j,m_j)$; in addition, let $R: \Hbb_1 \rightarrow \Hbb_2$ be a morphism of $\Scal$-fibered categories, and let $\rho$ be an external tensor structure on $R$ (with respect to $(\boxtimes_1,m_1)$ and $(\boxtimes_2,m_2)$). Then the following statements hold:
	\begin{enumerate}
		\item Suppose that $\Scal$ satisfies \Cref{hyp_skel}. Then the external tensor structure $\rho$ is completely determined by its restrictions to the underlying $\Scal^{sm}$-fibered and $\Scal^{cl}$-fibered categories. Conversely, such a pair of restrictions gives rise to an external tensor structure over $\Scal$ if and only if it defines an external tensor skeleton.
		\item Suppose that $\Scal$ satisfies the stronger \Cref{hyp_core}. In addition, suppose that the $\Scal$-fibered categories $\Hbb_1$ and $\Hbb_2$ are localic, that $(\boxtimes_1,m_1)$ and $(\boxtimes_2,m_2)$ are localic, and that $R$ is a localic morphism. 
		
		Then the external tensor structure $\rho$ is completely determined by its restrictions to the underlying $\Scal^{sm}$-fibered and $\Scal^{cl,op}$-fibered categories. Conversely, such a pair of restrictions gives rise to a (localic) external tensor structure over $\Scal$ if and only if it defines an external tensor core.
	\end{enumerate}
\end{thm*}

Along the way, in order to complete our main results on monoidal structures, we discuss how the natural coherence conditions of associativity and commutativity constraints for the external tensor product (as formulated in \cite[\S\S~3, 4]{Ter24Fib}) can be also checked over the subcategories $\Scal^{sm}$ and $\Scal^{cl}$. On the other hand, we do not include unit constraints into the picture, because they are not adapted to the setting of external tensor cores but only to that of external tensor skeleta.

\subsection*{Structure of the paper}

Throughout the paper, we work over a fixed small category $\Scal$; at the beginning of each section, we specify the natural conditions that $\Scal$ has to satisfy in order to make our constructions possible. 

\Cref{sect:adj-fib-cat} is mostly devoted to recalling the general conventions and notation that we employ in the sequel: in particular, we review the notions of $\Scal$-fibered categories (and morphisms thereof) and external tensor structures (both on single $\Scal$-fibered categories and on morphisms of such). We also discuss the natural notions of left-adjointability and right-adjointability in each setting (\Cref{defn:Sfib-adj}, \Cref{defn:mor-Sfib-adj} and \Cref{defn:ETS-adj}), and we prove some preliminary results about these (\Cref{lem_H-adjoint}, \Cref{lem_R-adjoint} and \Cref{lem:ETS-adj}).

Starting from \Cref{sect_mor-skel}, we assume to be given two distinguished subcategories $\Scal^{sm}$ and $\Scal^{cl}$ of $\Scal$ mimicking the usual properties of closed immersions and smooth morphisms of quasi-projective algebraic varieties (\Cref{hyp_skel}). In \Cref{sect_mor-skel} we introduce the notion of \textit{$\Scal$-skeleton} of a morphism of $\Scal$-fibered categories (\Cref{defn:skel}), and we show that giving a morphism of $\Scal$-fibered categories is equivalent to giving the underlying $\Scal$-skeleton (\Cref{prop_mor-skel}).

In \Cref{sect_mor-core} we work under a slightly more specific set of assumptions on the base category $\Scal$ and on its subcategories $\Scal^{sm}$ and $\Scal^{cl}$ (\Cref{hyp_core}). In the first place, we introduce the notion of \textit{localic} $\Scal$-fibered category (\Cref{defn_Hlocal}) and prove their basic properties (\Cref{lem_Hlocal}). We then introduce the notion of \textit{$\Scal$-core} for localic morphisms between localic $\Scal$-fibered categories (\Cref{defn:core}), and we show that giving such a morphism is equivalent to giving the underlying $\Scal$-core (\Cref{prop_mor-core}). 

In \Cref{sect_ETSkel} we go back to the general setting considered in \Cref{sect_mor-skel}; in order to consider external tensor structures on $\Scal$-fibered categories, we assume that $\Scal$ admits binary products. Following the same route as in \Cref{sect_mor-skel}, we introduce the notion of \textit{external tensor skeleton} on an $\Scal$-fibered category (\Cref{defn:ETSkel}), and we show that giving an external tensor structure is equivalent to giving the underlying external tensor skeleton (\Cref{thm_ETSskel}); a similar result applies to external tensor structures on morphisms of $\Scal$-fibered categories (\Cref{prop_mor_ETSkel}).

Finally, in \Cref{sect_ETScores} we reinforce again the assumptions on $\Scal$ as done in \Cref{sect_mor-core}, and we only consider localic $\Scal$-fibered categories; as in \Cref{sect_ETSkel}, we assume that $\Scal$ admits binary products. Following the same route as in \Cref{sect_mor-core}, we introduce the notion of \textit{external tensor core} on a localic $\Scal$-fibered category $\Hbb$ (\Cref{defn:ETCore}), and we show that giving a localic external tensor structure on $\Hbb$ is equivalent to giving the underlying external tensor core (\Cref{thm_ETScore}); a similar result applies to external tensor structures on localic morphisms of localic $\Scal$-fibered categories (\Cref{thm:ETScore-mor}).

\subsection*{Acknowledgments}

The contents of this paper correspond to the second chapter of my Ph.D. thesis, written at the University of Freiburg under the supervision of Annette Huber-Klawitter. It is a pleasure to thank her for many useful discussions around the subject of this article, as well as for her constant support and encouragement.

\section*{Notation and conventions}

\begin{itemize}
	\item Unless otherwise dictated, categories are assumed to be small with respect to some fixed universe.
	\item Given a category $\Ccal$, the notation $C \in \Ccal$ means that $C$ is an object of $\Ccal$.
	\item Given categories $\Ccal_1, \dots, \Ccal_n$, we let $\Ccal_1 \times \cdots \times \Ccal_n$ denote their direct product category.
\end{itemize}

\section{Adjointability of fibered categories and tensor structures}\label{sect:adj-fib-cat}

Throughout this paper, we work over a fixed small category $\Scal$; for the moment $\Scal$ could be completely general, except in the third part where we need to assume that $\Scal$ admits fibered products.

The goal of the first section is twofold: on the one side, we recall our basic conventions and notation about $\Scal$-fibered categories and external tensor structures from \cite[\S\S~1, 2]{Ter24Fib}; on the other side, we discuss various notions of adjointability and collect some basic results on that. To avoid possible confusion with the notation of the following sections, here we fix a (possibly non-full) subcategory $\Scal'$ of $\Scal$.

\subsection{Adjointable fibered categories}

As done in \cite[\S~1]{Ter24Fib}, we follow the conventions on fibered categories used in \cite{Ayo07a,Ayo07b} and \cite{CisDeg}. In detail, we define an \textit{$\Scal$-fibered category} $\Hbb$ as the datum of
\begin{itemize}
	\item for every $S \in \Scal$, a category $\Hbb(S)$,
	\item for every morphism $f: T \rightarrow S$ in $\Scal$, a functor 
	\begin{equation*}
		f^*: \Hbb(S) \rightarrow \Hbb(T),
	\end{equation*}
    called the \textit{inverse image functor along $f$},
	\item for every pair of composable morphisms $f: T \rightarrow S$, $g: S \rightarrow V$, a natural isomorphism of functors $\Hbb(V) \rightarrow \Hbb(T)$
	\begin{equation*}
		\conn = \conn_{f,g}: (gf)^* A \xrightarrow{\sim} f^* g^* A
	\end{equation*}
	called the \textit{connection isomorphism} at $(f,g)$
\end{itemize}
such that the following conditions are satisfied:
\begin{enumerate}
	\item[($\Scal$-fib-0)] For every $S \in \Scal$, we have $\id_S^* = \id_{\Hbb(S)}$.
	\item[($\Scal$-fib-1)] For every triple of composable morphisms $f: T \rightarrow S$, $g: S \rightarrow V$, $h: V \rightarrow W$ in $\Scal$, the diagram of functors $\Hbb(W) \rightarrow \Hbb(T)$
	\begin{equation*}
		\begin{tikzcd}
			(hgf)^* A \arrow{rr}{\conn_{f,hg}} \arrow{d}{\conn_{gf,h}} && f^* (hg)^* \arrow{d}{\conn_{g,h}} A \\
			(gf)^* h^* A \arrow{rr}{\conn_{f,g}} && f^* g^* h^* A
		\end{tikzcd}
	\end{equation*}
	is commutative.
\end{enumerate}

\begin{rem}
	\begin{enumerate}
		\item The strict unitarity of axiom ($\Scal$-fib-0) is mostly a matter of notational convenience; it is not a severe requirement (see for example \cite[Lemma~2.5]{DelVoe}). 
		\item As explained in \cite[Lemma~1.3]{Ter24Fib}, axiom ($\Scal$-fib-1) is designed in such a way that all composites of connection isomorphisms (and inverses thereof) are coherent; in practice, this means that one can treat connection isomorphisms as equalities. For sake of clarity, we write connection isomorphisms explicitly in the rest of this section; starting from \Cref{sect_mor-skel}, we will systematically write them as equalities in order to lighten our notation.
	\end{enumerate}
\end{rem}

Note that any $\Scal$-fibered category defines by restriction an $\Scal'$-fibered category.

Here is the first adjointability notion that we need to discuss:

\begin{defn}\label{defn:Sfib-adj}
	We say that an $\Scal$-fibered category $\Hbb$ is \textit{left-$\Scal'$-adjointable} (resp. \textit{right-$\Scal'$-adjointable}) if, for  every arrow $f: T \rightarrow S$ in $\Scal'$, the functor $f^*: \Hbb(S) \rightarrow \Hbb(T)$ admits a left adjoint $f_{\#}$ (resp. a right adjoint $f_*$). 
	For sake of brevity, in the case where $\Scal' = \Scal$ we simply say "left-adjointable" (resp. "right-adjointable") instead. 
\end{defn}

\begin{rem}\label{rem:adj-id}
	Of course, in view of axiom ($\Scal$-fib-0), for every $S \in \Scal'$ the inverse image functor $\id_S^* = \id_{\Hbb(S)}$ has a canonical left and right adjoint given by $\id_{\Hbb(S)}$ itself. We will always implicitly choose this as the left and right adjoint for the functor $\id_S^*$.
\end{rem}

The following result explains how the adjoints of the inverse image functors naturally assemble into an ${\Scal'}^{op}$-fibered category:

\begin{lem}\label{lem_H-adjoint}
	Let $\Hbb$ an $\Scal$-fibered category.
	\begin{enumerate}
		\item Suppose that $\Hbb$ is left-$\Scal'$-adjointable. Then, associating
		\begin{itemize}
			\item to every morphism $f: T \rightarrow S$ in $\Scal'$, the functor
			\begin{equation*}
				f_{\#}: \Hbb(T) \rightarrow \Hbb(S),
			\end{equation*}
			\item to every pair of composable morphisms $f: T \rightarrow S$ and $g: S \rightarrow V$ in $\Scal'$, the natural isomorphism of functors $\Hbb(T) \rightarrow \Hbb(V)$
			\begin{equation*}
				\begin{tikzcd}[font=\small]
					\overline{\conn} = \overline{\conn}_{f,g}: & (gf)_{\#} A \arrow{d}{\eta} \\
					& (gf)_{\#} f^* f_{\#} A \arrow{d}{\eta} \\
					& (gf)_{\#} f^* g^* g_{\#} f_{\#} A & (gf)_{\#} (gf)^* g_{\#} f_{\#} A \arrow{l}{\conn_{f,g}} \arrow{r}{\epsilon} & g_{\#} f_{\#} A
				\end{tikzcd}
			\end{equation*}
		\end{itemize}
		makes $\Hbb$ into an ${\Scal'}^{op}$-fibered category.
		\item Suppose that $\Hbb$ is right-adjointable. Then, associating
		\begin{itemize}
			\item to every morphism $f: T \rightarrow S$ in $\Scal'$, the functor
			\begin{equation*}
				f_*: \Hbb(T) \rightarrow \Hbb(S),
			\end{equation*}
			\item to every pair of composable morphisms $f: T \rightarrow S$ and $g: S \rightarrow V$ in $\Scal'$, the natural isomorphism of functors $\Hbb(T) \rightarrow \Hbb(V)$
			\begin{equation*}
				\begin{tikzcd}[font=\small]
					\overline{\conn} = \overline{\conn}_{f,g}: & (gf)_* A \arrow{d}{\eta} \\
					& g_* g^* (gf)_* A \arrow{d}{\eta} \\
					& g_* f_* f^* g^* (gf)_* A &  g_* f_* (gf)^* (gf)_* A \arrow{l}{\conn_{f,g}} \arrow{r}{\epsilon} & g_* f^* A
				\end{tikzcd}
			\end{equation*}
		\end{itemize}
		makes $\Hbb$ into an ${\Scal'}^{op}$-fibered category.
	\end{enumerate}
\end{lem}
\begin{proof}
	We only treat the case where $\Hbb$ is left-adjointable; the other case is analogous.
	
	Since condition (${\Scal'}^{op}$-fib-0) holds by hypothesis (see \Cref{rem:adj-id}), we only need to check that the natural isomorphisms in the statement satisfy condition (${\Scal'}^{op}$-fib-1): given three composable morphisms $f: T \rightarrow S$, $g: S \rightarrow V$ and $h: V \rightarrow W$ in $\Scal'$, we have to show that the diagram of functors $\Hbb(T) \rightarrow \Hbb(W)$
	\begin{equation*}
		\begin{tikzcd}
			(hgf)_{\#} A \arrow{r}{\overline{\conn}_{f,hg}} \arrow{d}{\overline{\conn}_{gf,h}} & (hg)_{\#} f_{\#} A \arrow{d}{\overline{\conn}_{g,h}} \\
			h_{\#} (gf)_{\#} A \arrow{r}{\overline{\conn}_{g,f}} & h_{\#} g_{\#} f_{\#} A
		\end{tikzcd}
	\end{equation*}
	is commutative. Expanding the various definitions, we obtain the more explicit diagram
	\begin{equation*}
		\begin{tikzcd}[font=\tiny]
			(hgf)_{\#} A \arrow{r}{\eta} \arrow{d}{\eta} & (hgf)_{\#} f^* f_{\#} A \arrow{r}{\eta} & (hgf)_{\#} f^* (hg)^* (hg)_{\#} f_{\#} A & (hgf)_{\#} (hgf)^* (hg)_{\#} f_{\#} A \arrow{l}{\conn_{f,hg}} \arrow{r}{\epsilon} & (hg)_{\#} f_{\#} A \arrow{dddd}{\overline{\conn}_{g,h}} \\
			(hgf)_{\#} (gf)^* (gf)_{\#} A \arrow{d}{\eta} \\
			(hgf)_{\#} (gf)^* h^* h_{\#} (gf)_{\#} A  \\
			(hgf)_{\#} (hgf)^* h_{\#} (gf)_{\#} A \arrow{u}{\conn_{gf,h}} \arrow{d}{\epsilon} \\
			h_{\#} (gf)_{\#} A \arrow{rrrr}{\overline{\conn}_{f,g}} &&&& h_{\#} g_{\#} f_{\#} A
		\end{tikzcd}
	\end{equation*} 
    that we decompose as
    \begin{equation*}
	    \begin{tikzcd}[font=\tiny]
	    	\bullet \arrow{rr}{\eta} \arrow{d}{\eta} && \bullet \arrow{r}{\eta} \arrow{d}{\eta} & \bullet \arrow{dd}{\overline{\conn}_{g,h}} & \bullet \arrow{l}{\conn_{f,hg}} \arrow{r}{\epsilon} \arrow[bend left=50]{dddl}{\overline{\conn}_{g,h}} & \bullet \arrow{dddd}{\overline{\conn}_{g,h}} \\
	    	\bullet \arrow{d}{\eta} \arrow{r}{\overline{\conn}_{f,g}} & (hgf)_{\#} (gf)^* g_{\#} f_{\#} A \arrow{r}{\conn_{f,g}} \arrow{d}{\eta} & (hgf)_{\#} f^* g^* g_{\#} f_{\#} A \arrow{d}{\eta}  \\
	    	\bullet  \arrow{r}{\overline{\conn}_{f,g}} & (hgf)_{\#} (gf)^* h^* h_{\#} g_{\#} f_{\#} A \arrow{r}{\conn_{f,g}} & (hgf)_{\#} f^* g^* h^* h_{\#} g_{\#} f_{\#} A & (hgf)_{\#} f^* (hg)^* h_{\#} g_{\#} f_{\#} A \arrow{l}{\conn_{g,h}} \\
	    	\bullet \arrow{u}{\conn_{gf,h}} \arrow{d}{\epsilon} \arrow{rrr}{\overline{\conn}_{f,g}} &&& (hgf)_{\#} (hgf)^* h_{\#} g_{\#} f_{\#} A \arrow{drr}{\epsilon} \arrow{u}{\conn_{f,hg}} \arrow{ull}{\conn_{gf,h}} \\
	    	\bullet \arrow{rrrrr}{\overline{\conn}_{f,g}} &&&&& \bullet
	    \end{tikzcd}
    \end{equation*}
    Here, the central triangle is commutative by axiom ($\Scal$-fib-1) while the other pieces are commutative by naturality and by construction. This proves the claim.
\end{proof}

\subsection{Adjointable morphisms of fibered categories}

We go on by reviewing morphisms of $\Scal$-fibered categories. Let $\Hbb_1$ and $\Hbb_2$ be two $\Scal$-fibered categories; for every composable pair of morphisms $f: T \rightarrow S$ and $g: S \rightarrow V$ in $\Scal$, we write the corresponding connection isomorphisms for $\Hbb_1$ and $\Hbb_2$ as $\conn_{f,g}^{(1)}$ and $\conn_{f,g}^{(2)}$, respectively.
As done in \cite[\S~1]{Ter24Fib}, we follow the conventions used in \cite{Ayo10} and \cite{CisDeg}. We say that a morphism of $\Scal$-fibered categories $R: \Hbb_1 \rightarrow \Hbb_2$ is the datum of
\begin{itemize}
	\item for every $S \in \Scal$, a functor $R_S: \Hbb_1(S) \rightarrow \Hbb_2(S)$,
	\item for every morphism $f: T \rightarrow S$ in $\Scal$, a natural isomorphism of functors $\Hbb_1(S) \rightarrow \Hbb_2(T)$
	\begin{equation*}
		\theta = \theta_f: f^* R_S(A) \xrightarrow{\sim} R_T (f^* A),
	\end{equation*}
	called the \textit{$R$-transition isomorphism} along $f$
\end{itemize}
such that the following condition is satisfied:
\begin{enumerate}
	\item[(mor-$\Scal$-fib)] For every pair of composable morphisms $f: T \rightarrow S$ and $g: S \rightarrow V$, the diagram of functors $\Hbb_1(V) \rightarrow \Hbb_2(T)$
	\begin{equation*}
		\begin{tikzcd}
			(gf)^* R_V(A) \arrow{rr}{\theta_{gf}} \arrow{d}{\conn^{(2)}_{f,g}} && R_T ((gf)^* A) \arrow{d}{\conn^{(1)}_{f,g}} \\
			f^* g^* R_V(A) \arrow{r}{\theta_g} & f^* R_S (g^* A) \arrow{r}{\theta_f} & R_T (f^* g^* A)
		\end{tikzcd}
	\end{equation*}
	is commutative.
\end{enumerate}
For sake of clarity, it will be convenient to write a morphism $R: \Hbb_1 \rightarrow \Hbb_2$ as a pair $(\Rcal,\theta)$ consisting of 
\begin{itemize}
	\item a family of functors $\Rcal := \left\{R_S: \Hbb_1(S) \rightarrow \Hbb_2(S) \right\}_{S \in \Scal}$,
	\item a system of natural isomorphisms $\theta = \left\{\theta_f: f^* \circ R_S \xrightarrow{\sim} R_T \circ f^* \right\}_{f: T \rightarrow S}$ satisfying condition (mor-$\Scal$-fib) above. 
\end{itemize}
If we are already given a family $\Rcal$ as above, we say that such a system of natural isomorphisms $\theta$ defines an \textit{$\Scal$-structure} on $\Rcal$.

\begin{rem}\label{rem:theta}
	As explained in \cite[Lemma~1.8]{Ter24Fib}, axiom (mor-$\Scal$-fib) is designed in such a way that all composites of connection isomorphisms and $R$-transition isomorphisms (and inverses thereof) are coherent. 
\end{rem}

Note that any morphism of $\Scal$-fibered categories determines by restriction a morphism of $\Scal'$-fibered categories. 

The notion of adjointability for $\Scal$-fibered categories motivates a similar notion for morphisms:

\begin{defn}\label{defn:mor-Sfib-adj}
	Let $\Hbb_1$ and $\Hbb_2$ be two left-$\Scal'$-adjointable (resp. right-$\Scal'$-adjointable) $\Scal$-fibered categories.
	\begin{enumerate}
		\item We say that a morphism of $\Scal$-fibered categories $R = (\Rcal,\theta): \Hbb_1 \rightarrow \Hbb_2$ is \textit{left-$\Scal'$-adjointable} (resp. \textit{right-$\Scal'$-adjointable}) if, for every morphism $f: T \rightarrow S$ in $\Scal'$, the natural transformation of functors $\Hbb_1(T) \rightarrow \Hbb_2(S)$
		\begin{equation}\label{theta'-l}
			f_{\#} R_T(A) \xrightarrow{\eta} f_{\#} R_T(f^* f_{\#} A) \xleftarrow{\theta_f} f_{\#} f^* R_S(f_{\#} A) \xrightarrow{\epsilon} R_S(f_{\#} A)
		\end{equation}
		is invertible (resp. the natural transformation of functors $\Hbb_1(T) \rightarrow \Hbb_2(S)$ 
		\begin{equation}\label{theta'-r}
			R_T(f_* A) \xrightarrow{\eta} f_* f^* R_T(f_* A) \xrightarrow{\theta_f} f_* R_S(f^* f_* A) \xrightarrow{\epsilon} f_* R_S(A)
		\end{equation}
		is invertible). In this case, we let $\bar{\theta} = \bar{\theta}_f$ denote the natural transformation \eqref{theta'-l} (resp. \eqref{theta'-r}).
		\item If we are already given a family of functors $\Rcal := \left\{R_S: \Hbb_1(S) \rightarrow \Hbb_2(S)\right\}_{S \in \Scal}$, we say that an $\Scal$-structure $\theta$ on $\Rcal$ is \textit{left-$\Scal'$-adjointable} (resp. \textit{right-$\Scal'$-adjointable}) if so is the resulting morphism of $\Scal$-fibered categories $(\Rcal,\theta): \Hbb_1 \rightarrow \Hbb_2$.
	\end{enumerate}
\end{defn}

\begin{lem}\label{lem_R-adjoint}
	Let $R = (\Rcal,\theta): \Hbb_1 \rightarrow \Hbb_2$ be a left-$\Scal'$-adjointable (resp. right-$\Scal'$-adjointable) morphism of left-adjointable (resp. right-adjointable) $\Scal$-fibered categories.
	Then the natural isomorphisms \eqref{theta'-l} (resp. the inverses of the natural isomorphisms \eqref{theta'-r}) define a morphism of ${\Scal'}^{op}$-fibered categories $(\Rcal,\bar{\theta}): \Hbb_1 \rightarrow \Hbb_2$.
\end{lem}
\begin{proof}
	We only treat the left-adjointable case; the other case is analogous.
	
	We need to check that the natural isomorphisms in the statement satisfy condition (mor-${\Scal'}^{op}$-fib): given two composable morphisms $f: T \rightarrow S$ and $g: S \rightarrow V$ in $\Scal'$, we have to show that the diagram of functors $\Hbb_1(T) \rightarrow \Hbb_2(V)$
	\begin{equation*}
		\begin{tikzcd}
			(gf)_{\#} R_T(A) \arrow{rr}{\bar{\theta}_{gf}} \arrow[equal]{d} && R_V((gf)_{\#} A) \arrow[equal]{d} \\
			g_{\#} f_{\#} R_T(A) \arrow{r}{\bar{\theta}_f} & g_{\#} R_S(f_{\#} A) \arrow{r}{\bar{\theta}_g} & R_V(g_{\#} f_{\#} A)
		\end{tikzcd}
	\end{equation*}
	is commutative. Expanding the definitions of the horizontal arrows, we obtain the more explicit diagram
	\begin{equation*}
	    \begin{tikzcd}[font=\tiny]
	    	(gf)_{\#} R_T(A) \arrow{d}{\overline{\conn}^{(2)}} \arrow{r}{\eta} & (gf)_{\#} R_T((gf)^* (gf)_{\#} A) && (gf)_{\#} (gf)^* R_V((gf)_{\#} A) \arrow{ll}{\theta} \arrow{r}{\epsilon} & R_V((gf)_{\#} A) \arrow{d}{\overline{\conn}^{(1)}} \\
	    	g_{\#} f_{\#} R_T(A) \arrow{d}{\eta} &&&& R_V(g_{\#} f_{\#} A) \\
	    	g_{\#} f_{\#} R_T(f^* f_{\#} A) & g_{\#} f_{\#} f^* R_S(f_{\#} A) \arrow{l}{\theta} \arrow{r}{\epsilon} & g_{\#} R_S(f_{\#} A) \arrow{r}{\eta} & g_{\#} R_S(g^* g_{\#} f_{\#} A) \arrow{r}{\theta} & g_{\#} g^* R_V(g_{\#} f_{\#} A) \arrow{u}{\epsilon} \\
		\end{tikzcd}
	\end{equation*}
    that we decompose as
    \begin{equation*}
    	\begin{tikzcd}[font=\small]
    		\bullet \arrow{d}{\overline{\conn}^{(2)}} \arrow{r}{\eta} & \bullet \arrow{d}{\overline{\conn}^{(1)}} \arrow{dr}{\overline{\conn}^{(1)}} && \bullet \arrow{ll}{\theta} \arrow{d}{\overline{\conn}^{(1)}} \arrow{r}{\epsilon} & \bullet \arrow{d}{\overline{\conn}^{(1)}} \\
    		\bullet \arrow{r}{\eta} \arrow{ddd}{\eta} & g_{\#} f_{\#} R_T((gf)^* (gf)_{\#} A) \arrow{d}{\overline{\conn}^{(1)}} & (gf)_{\#} R_T((gf)^* g_{\#} f_{\#} A) \arrow{dl}{\overline{\conn}^{(2)}} & (gf)_{\#} (gf)^* R_V(g_{\#} f_{\#} A) \arrow{l}{\theta} \arrow{r}{\epsilon} \arrow{d}{\overline{\conn}^{(2)}} & \bullet \\
    		& g_{\#} f_{\#} R_T((gf)^* g_{\#} f_{\#} A) \arrow[equal]{d} && g_{\#} f_{\#} (gf)^* R_V(g_{\#} f_{\#} A) \arrow{ll}{\theta} \arrow[equal]{d} \\
    		 & g_{\#} f_{\#} R_T(f^* g^* g_{\#} f_{\#} A) & g_{\#} f_{\#} f^* R_S(g^* g_{\#} f_{\#} A) \arrow{l}{\theta} \arrow{dr}{\epsilon} & g_{\#} f_{\#} f^* g^* R_V(g_{\#} f_{\#} A) \arrow{l}{\theta} \arrow{dr}{\epsilon} \\
    		\bullet \arrow{ur}{\eta} & \bullet \arrow{l}{\theta} \arrow{ur}{\eta} \arrow{r}{\epsilon} & \bullet \arrow{r}{\eta} & \bullet \arrow{r}{\theta} & \bullet \arrow{uuu}{\epsilon}
    	\end{tikzcd}
    \end{equation*}
    Here, the central rectangle is commutative by axiom (mor-$\Scal$-fib) while the other pieces are commutative by naturality and by construction. This proves the claim.
\end{proof}

\subsection{Adjointable external tensor structures}

Following \cite[\S~2]{Ter24Fib}, we say that an \textit{external tensor structure} $(\boxtimes,m)$ on an $\Scal$-fibered category $\Hbb$ is the datum of
\begin{itemize}
	\item for every $S_1, S_2 \in \Scal$, a functor
	\begin{equation*}
		-\boxtimes- = - \boxtimes_{S_1,S_2} -: \Hbb(S_1) \times \Hbb(S_2) \rightarrow \Hbb(S_1 \times S_2),
	\end{equation*}
	called the \textit{external tensor product functor} over $(S_1,S_2)$,
	\item for every choice of morphisms $f_i: T_i \rightarrow S_i$ in $\Scal$, $i = 1,2$, a natural isomorphism of functors $\Hbb(S_1) \times \Hbb(S_2) \rightarrow \Hbb(T_1 \times T_2)$
	\begin{equation*}
		m = m_{f_1,f_2}: f_1^* A_1 \boxtimes f_2^* A_2 \xrightarrow{\sim} (f_1 \times f_2)^*(A_1 \boxtimes A_2),
	\end{equation*}
	called the \textit{external monoidality isomorphism} along $(f_1,f_2)$
\end{itemize}
satisfying the following condition:
\begin{enumerate}
	\item[($m$ETS)] For every choice of composable morphisms $f_i: T_i \rightarrow S_i$, $g_i: S_i \rightarrow V_i$, $i = 1,2$, the diagram of functors $\Hbb(V_1) \times \Hbb(V_2) \rightarrow \Hbb(T_1 \times T_2)$
	\begin{equation*}
		\begin{tikzcd}
			(g_1 f_1)^* A_1 \boxtimes (g_2 f_2)^* A_2  \arrow{rr}{m} \arrow[equal]{d} && (g_1 f_1 \times g_2 f_2)^* (A_1 \boxtimes A_2)  \arrow[equal]{d} \\
			f_1^* g_1^* A_1 \boxtimes f_2^* g_2^* A_2  \arrow{r}{m} & (f_1 \times f_2)^* (g_1^* A_1 \boxtimes g_2^* A_2) \arrow{r}{m} & (f_1 \times f_2)^* (g_1 \times g_2)^* (A_1 \boxtimes A_2)
		\end{tikzcd}
	\end{equation*}	
	is commutative.
\end{enumerate}

Note that any external tensor structure on $\Hbb$ defines by restriction an external tensor structure on the underlying $\Scal'$-fibered category.

The notion of adjointability for $\Scal$-fibered categories motivates a similar notion for external tensor structures:

\begin{defn}\label{defn:ETS-adj}
	Suppose that the $\Scal$-fibered category $\Hbb$ is left-$\Scal'$-adjointable (resp. right-$\Scal'$-adjointable), and let $(\boxtimes,m)$ be an external tensor structure on $\Hbb$.
	We say that $(\boxtimes,m)$ is \textit{left-$\Scal'$-adjointable} (resp. \textit{right-$\Scal'$-adjointable}) if, for every choice of morphisms $f_i: T_i \rightarrow S_i$ in $\Scal'$, $i = 1,2$, the natural transformation of functors $\Hbb(T_1) \times \Hbb(T_2) \rightarrow \Hbb(S_1 \times S_2)$
	\begin{equation}\label{iso:m'-left}
		\begin{tikzcd}[font=\small]
			(f_1 \times f_2)_{\#}(A_1 \boxtimes A_2) \arrow{d}{\eta} & f_{1,\#} A_1 \boxtimes f_{2,\#} A_2 \\
			(f_1 \times f_2)_{\#}(f_1^* f_{1,\#} A_1 \boxtimes f_2^* f_{2,\#} A_2) \arrow{r}{m} & (f_1 \times f_2)_{\#} (f_1 \times f_2)^* (f_{1,\#} A_1 \boxtimes f_{2,\#} A_2) \arrow{u}{\epsilon}
		\end{tikzcd}
	\end{equation}	
	is invertible (resp. the natural transformation of functors $\Hbb(T_1) \times \Hbb(T_2) \rightarrow \Hbb(S_1 \times S_2)$
	\begin{equation}\label{iso:m'-right}
		\begin{tikzcd}[font=\small]
			f_{1,*} A_1 \boxtimes f_{2,*} A_2 \arrow{d}{\eta} & (f_1 \times f_2)_* (A_1 \boxtimes A_2) \\
			(f_1 \times f_2)_* (f_1 \times f_2)^* (f_{1,*} A_1 \boxtimes f_{2,*} A_2) & (f_1 \times f_2)_* (f_1^* f_{1,*} A_1 \boxtimes f_2^* f_{2,*} A_2) \arrow{u}{\epsilon} \arrow{l}{m}
		\end{tikzcd}
	\end{equation}
	is invertible). In this case, we let $\bar{m} = \bar{m}_{f_1,f_2}$ denote the inverse of the natural isomorphism \eqref{iso:m'-left} (resp. the natural isomorphism \eqref{iso:m'-right}).
\end{defn}

One can easily check that an external tensor structure $(\boxtimes,m)$ on $\Hbb$ is left-$\Scal'$-adjointable if and only if, for every choice of morphisms $f_i: T_i \rightarrow S_i$ in $\Scal$, the natural transformation of functors $\Hbb_1(T_1) \times \Hbb_1(S_2) \rightarrow \Hbb_2(S_1 \times S_2)$
\begin{equation*}
	\begin{tikzcd}[font=\small]
		(f_1 \times \id_{S_2})_{\#}(A_1 \boxtimes B_2) \arrow{d}{\eta} & f_{1,\#} A_1 \boxtimes B_2 \\
		(f_1 \times \id_{S_2})_{\#}(f_1^* f_{1,\#} A_1 \boxtimes B_2) \arrow{r}{m} & (f_1 \times \id_{S_2})_{\#} (f_1 \times \id_{S_2})^* (f_{1,\#} A_1 \boxtimes B_2) \arrow{u}{\epsilon}
	\end{tikzcd}
\end{equation*}
and the natural transformation of functors $\Hbb_1(S_1) \times \Hbb_1(T_2) \rightarrow \Hbb(S_1 \times S_2)$
\begin{equation*}
	\begin{tikzcd}[font=\small]
		(\id_{S_1} \times f_2)_{\#}(B_1 \boxtimes A_2) \arrow{d}{\eta} & B_1 \boxtimes f_{2,\#} A_2 \\
		(\id_{S_1} \times f_2)_{\#}(B_1 \boxtimes f_2^* f_{2,\#} A_2) \arrow{r}{m} & (\id_{S_1} \times f_2)_{\#} (\id_{S_1} \times f_2)^* (B_1 \boxtimes f_{2,\#} A_2) \arrow{u}{\epsilon}
	\end{tikzcd}
\end{equation*}
are both invertible. Thus the notion of left-adjointability just introduced recovers the formulation of the \textit{projection formulae} in the language of external tensor structures as stated in \cite[\S~7]{Ter24Fib}. Analogous considerations apply to the notion of right-adjointability.

Given a left-$\Scal'$-adjointable (resp. right-$\Scal'$-adjointable) external tensor structure on a left-adjointable (resp. right adjointable) $\Scal$-fibered category, the following result describes how to convert it into an external tensor structure on the corresponding ${\Scal'}^{op}$-fibered category:

\begin{lem}\label{lem:ETS-adj}
	Suppose that the $\Scal$-fibered category $\Hbb$ is left-$\Scal'$-adjointable (resp. right-$\Scal'$-adjointable), and let $(\boxtimes,m)$ be a left-$\Scal'$-adjointable (resp. right-$\Scal'$-adjointable) external tensor structure on $\Hbb$. Then the natural isomorphisms \eqref{iso:m'-left} (resp. the natural isomorphisms \eqref{iso:m'-right}) define an external tensor structure $(\boxtimes,\bar{m})$ on the ${\Scal'}^{op}$-fibered category $\Hbb$ obtained via \Cref{lem_H-adjoint}.
\end{lem}
\begin{proof}
	We only treat the left-adjointable case; the other case is analogous. 
	
	We need to check that condition ($m$ETS) is satisfied: given, for $i = 1,2$, two composable morphisms $f_i: T_i \rightarrow S_i$ and $g_i: S_i \rightarrow V_i$ in $\Scal$, we have to show that the diagram of functors $\Hbb(T_1) \times \Hbb(T_2) \rightarrow \Hbb(S_1 \times S_2)$
	\begin{equation*}
		\begin{tikzcd}
			(g_1 f_1)_{\#} A_1 \boxtimes (g_2 f_2)_{\#} A_2 \arrow{rr}{\bar{m}} \arrow[equal]{d} && (g_1 f_1 \times g_2 f_2)_{\#}(A_1 \boxtimes A_2) \arrow[equal]{d} \\
			g_{1,\#} f_{1,\#} A_1 \boxtimes g_{2,\#} f_{2,\#} A_2 \arrow{r}{\bar{m}} & (g_1 \times g_2)_{\#} (f_{1,\#} A_1 \boxtimes f_{2,\#} A_2) \arrow{r}{\bar{m}} & (g_1 \times g_2)_{\#} (f_1 \times f_2)_{\#} (A_1 \boxtimes A_2)
		\end{tikzcd}
	\end{equation*}
	is commutative. Unwinding the various definitions, we obtain the diagram
	\begin{equation*}
		\begin{tikzcd}[font=\tiny]
			(g_{12} f_{12})_{\#} (g_{12} f_{12})^* ((g_1 f_1)_{\#} A_1 \boxtimes (g_2 f_2)_{\#} A_2) \arrow{d}{\epsilon} && (g_{12} f_{12})_{\#} ((g_1 f_1)^* (g_1 f_1)_{\#} A_1 \boxtimes (g_2 f_2)^* (g_2 f_2)_{\#} A_2) \arrow{ll}{m} \\
			(g_1 f_1)_{\#} A_1 \boxtimes (g_2 f_2)_{\#} A_2 \arrow{d}{\eta} && (g_{12} f_{12})_{\#} (A_1 \boxtimes A_2) \arrow{d}{\eta} \arrow{u}{\eta} \\
			(g_1 f_1)_{\#} f_1^* f_{1,\#} A_1 \boxtimes (g_2 f_2)_{\#} f_2^* f_{2,\#} A_2 \arrow{d}{\eta} && (g_{12} f_{12})_{\#} f_{12}^* f_{12,\#} (A_1 \boxtimes A_2) \arrow{d}{\eta} \\
			(g_1 f_1)_{\#} f_1^* g_1^* g_{1,\#} f_{1,\#} A_1 \boxtimes (g_2 f_2)_{\#} f_2^* g_2^* g_{2,\#} f_{2,\#} A_2 \arrow[equal]{d} && (g_{12} f_{12})_{\#} f_{12}^* g_{12}^* g_{12,\#} f_{12,\#} (A_1 \boxtimes A_2) \arrow[equal]{d} \\
			(g_1 f_1)_{\#} (g_1 f_1)^* g_{1,\#} f_{1,\#} A_1 \boxtimes (g_2 f_2)_{\#} (g_2 f_2)^* g_{2,\#} f_{2,\#} A_2 \arrow{d}{\epsilon} && (g_{12} f_{12})_{\#} (g_{12} f_{12})^* g_{12,\#} f_{12,\#} (A_1 \boxtimes A_2) \arrow{d}{\epsilon} \\
			g_{1,\#} f_{1,\#} A_1 \boxtimes g_{2,\#} f_{2,\#} A_2 && g_{12,\#} f_{12,\#} (A_1 \boxtimes A_2) \arrow{d}{\eta} \\
			g_{12,\#} g_{12}^* (g_{1,\#} f_{1,\#} A_1 \boxtimes g_{2,\#} f_{2,\#} A_2) \arrow{u}{\epsilon} && g_{12,\#} f_{12,\#} (f_1^* f_{1,\#} A_1 \boxtimes f_2^* f_{2,\#} A_2) \arrow{d}{m} \\
			g_{12,\#} (g_1^* g_{1,\#} f_{1,\#} A_1 \boxtimes g_2^* g_{2,\#} f_{2,\#} A_2) \arrow{u}{m} & g_{12,\#} (f_{1,\#} A_1 \boxtimes f_{2,\#} A_2) \arrow{l}{\eta} & g_{12,\#} f_{12,\#} f_{12}^* (f_{1,\#} A_1 \boxtimes f_{2,\#} A_2) \arrow{l}{\epsilon}
		\end{tikzcd}
	\end{equation*}
    that we decompose as
    \begin{equation*}
    	\begin{turn}{90}
    		\begin{tikzcd}[font=\tiny]
    			\bullet \arrow{d}{\epsilon} &&&& \bullet \arrow{llll}{m} \\
    			\bullet \arrow{d}{\eta} &&&& \bullet \arrow{d}{\eta} \arrow{u}{\eta} \\
    			\bullet \arrow{d}{\eta} &&& (g_{12} f_{12})_{\#} f_{12}^* g_{12}^* g_{12,\#} (f_{1,\#} A_1 \boxtimes f_{2,\#} A_2) \arrow[equal]{d} \arrow[equal]{dr} & \bullet \arrow{d}{\eta} \\
    			\bullet \arrow[equal]{d} \arrow{r}{\bar{m}} & (g_{12} f_{12})_{\#} (f_1^* g_1^* g_{1,\#} f_{1,\#} A_1 \boxtimes f_2^* g_2^* g_{2,\#} f_{2,\#} A_2) \arrow{r}{m} \arrow[equal]{d} & (g_{12} f_{12})_{\#} f_{12}^* (g_1^* g_{1,\#} f_{1,\#} A_1 \boxtimes g_2^* g_{2,\#} f_{2,\#} A_2) \arrow{r}{m} & (g_{12} f_{12})_{\#} f_{12}^* g_{12}^* (g_{1,\#} f_{1,\#} A_1 \boxtimes g_{2,\#} f_{2,\#} A_2) \arrow[equal]{d} & \bullet \arrow[equal]{d} \\
    			\bullet \arrow{d}{\epsilon} \arrow{r}{\bar{m}} & (g_{12} f_{12})_{\#} ((g_1 f_1)^* g_{1,\#} f_{1,\#} A_1 \boxtimes (g_2 f_2)^* g_{2,\#} f_{\#} A_2) \arrow{rr}{m} && (g_{12} f_{12})_{\#} (g_{12} f_{12})^* (g_{1,\#} f_{1,\#} A_1 \boxtimes g_{2,\#} f_{2,\#} A_2) & \bullet \arrow{d}{\epsilon} \\
    			\bullet &&& (g_{12} f_{12})_{\#} (g_{12} f_{12})^* g_{12,\#} (f_{1,\#} A_1 \boxtimes f_{2,\#} A_2) \arrow[equal]{ur} \arrow[equal]{u} & \bullet \arrow{d}{\eta} \\
    			\bullet \arrow{u}{\epsilon} &&&& \bullet \arrow{d}{m} \\
    			\bullet \arrow{u}{m} && \bullet \arrow{ll}{\eta} && \bullet \arrow{ll}{\epsilon}
    		\end{tikzcd}
    	\end{turn}
    \end{equation*}
    Here, the central rectangle is commutative by axiom ($m$ETS) for $(\boxtimes,m)$ and the right-most central piece is commutative by naturality; as the reader can easily check, the upper and lower pieces are commutative by construction. We omit the details.
\end{proof}

\begin{rem}\label{rem:whatifITS}
	It does not look particularly interesting to develop the analogous results for internal tensor structures: the main reason is that most internal tensor structures encountered in concrete situations are not left-$\Scal'$-adjointable nor right-$\Scal'$-adjointable even if the corresponding external tensor structures are so. 
	To fix ideas, let $(\boxtimes,m)$ be a left-$\Scal$-adjointable external tensor structure on a left-$\Scal'$-adjointable $\Scal$-fibered category $\Hbb$; let $(\otimes,m)$ denote the internal tensor structure on $\Hbb$ corresponding to $(\boxtimes,m)$ under the bijection of \cite[Thm.~6.1]{Ter24Fib}. Recall from \cite[Lemma~2.11]{Ter24Fib} that the functor
	\begin{equation*}
		- \otimes -: \Hbb(S) \times \Hbb(S) \rightarrow \Hbb(S)
	\end{equation*}
    is defined by the formula
    \begin{equation*}
    	A \otimes B := \Delta_S^* (A \boxtimes B)
    \end{equation*}
    and that, for every arrow $f: T \rightarrow S$ in $\Scal$, the internal monoidality isomorphism $m^i_f$ is defined by the formula
    \begin{equation*}
    	m^i_f: f^* A \otimes f^* B := \Delta_T^* (f^* A \boxtimes f^* B) \xrightarrow{m_f} \Delta_T^* (f \times f)^* (A \boxtimes B) = f^* \Delta_S^* (A \boxtimes B) =: (f \times f)^* (A \otimes B).
    \end{equation*}
    Then one should try to see whether, for an arrow $f: T \rightarrow S$ in $\Scal'$, the natural transformation of functors $\Hbb(T) \times \Hbb(T) \rightarrow \Hbb(S)$
    \begin{equation*}
    	f_{\#} (A \otimes B) \xrightarrow{\eta} f_{\#} (f^* f_{\#} A \otimes f^* f_{\#} B) \xrightarrow{m^i_f} f_{\#} f^* (f_{\#} A \otimes f_{\#} B) \xrightarrow{\epsilon }f_{\#} A \otimes f_{\#} B 
    \end{equation*}
    is invertible. Expanding the various definitions, one finds the composite
    \begin{equation*}
    	\begin{tikzcd}[font=\small]
    		f_{\#} \Delta_T^* (A \boxtimes B) \arrow{d} \\
    		f_{\#} \Delta_T^* (f^* f_{\#} A \boxtimes f^* f_{\#} B) \arrow{r}{m_f} &  f_{\#} \Delta_T^* (f \times f)^* (f_{\#} A \boxtimes f_{\#} B) \arrow[equal]{r} & f_{\#} f^* \Delta_S^* (f_{\#} A \boxtimes f_{\#} B) \arrow{d}{\epsilon} \\ 
    		&& \Delta_S^* (f_{\#} A \boxtimes f_{\#} B).
    	\end{tikzcd}
    \end{equation*}
    As one can easily check, the latter coincides with the composite
    \begin{equation*}
    	\begin{tikzcd}[font=\small]
    		f_{\#} \Delta_T^* (A \boxtimes B) \arrow{d}{\eta} \\ 
    		f_{\#} \Delta_T^* (f \times f)^* (f \times f)_{\#} (A \boxtimes B) \arrow[equal]{r} & f_{\#} f^* \Delta_S^* (f \times f)_{\#} (A \boxtimes B) \arrow{r}{\epsilon} & \Delta_S^* (f \times f)_{\#} (A \boxtimes B) \arrow{d}{\overline{m}_f} \\ 
    		&& \Delta_S^* (f_{\#} A \boxtimes f_{\#} B),
    	\end{tikzcd}
    \end{equation*}
    which is invertible if and only of so is the Beck--Chevalley transformation
    \begin{equation*}
    	f_{\#} \Delta_T^* (A \boxtimes B) \xrightarrow{\eta} f_{\#} \Delta_T^* (f \times f)^* (f \times f)_{\#} (A \boxtimes B) = f_{\#} f^* \Delta_S^* (f \times f)_{\#} (A \boxtimes B) \xrightarrow{\epsilon} \Delta_S^* (f \times f)_{\#} (A \boxtimes B)
    \end{equation*}
    associated to the commutative diagram in $\Scal$
    \begin{equation}\label{dia:TSdelta}
    	\begin{tikzcd}
    		T \arrow{r}{\Delta_T} \arrow{d}{f} & T \times T \arrow{d}{f \times f} \\
    		S \arrow{r}{\Delta_S} & S \times S.
    	\end{tikzcd}
    \end{equation}
    In most concrete geometric settings, one would try to conclude by applying some base-change result; but the diagram \eqref{dia:TSdelta} is not Cartesian unless $f$ is a monomorphism! In the case where $\Scal = \Var_k$ is the category of algebraic varieties over a field $k$ and $\Hbb$ underlies a monoidal stable homotopy $2$-functor in the sense of \cite{Ayo07a,Ayo07b}, the internal tensor structure on $\Hbb$ is $\Scal'$-left-adjointable for $\Scal'$ the class of open immersions and $\Scal'$-right-adjointable for $\Scal'$ the class of closed immersions; there does not seem to be much more to say outside these two cases.
\end{rem}

Let us also check that adjointability of external tensor structure is compatible with external associativity and commutativity constraints as introduced in \cite[\S\S~3, 4]{Ter24Fib}. Let $(\boxtimes,m)$ be an external tensor structure on $\Hbb$. 
Recall that an \textit{external associativity constraint} $a$ on $(\boxtimes,m)$ is the datum of
\begin{itemize}
	\item for every $S_1, S_2, S_3 \in \Scal$, a natural isomorphism between functors $\Hbb(S_1) \times \Hbb(S_2) \times \Hbb(S_3) \rightarrow \Hbb(S_1 \times S_2 \times S_3)$
	\begin{equation*}
		a = a_{S_1,S_2,S_3}: (A_1 \boxtimes A_2) \boxtimes A_3 \xrightarrow{\sim} A_1 \boxtimes (A_2 \boxtimes A_3)
	\end{equation*}
\end{itemize}
satisfying the following conditions:
\begin{enumerate}
	\item[($a$ETS-1)] For every $S_1, S_2, S_3, S_4 \in \Scal$, the diagram of functors $\Hbb(S_1) \times \Hbb(S_2) \times \Hbb(S_3) \times \Hbb(S_4) \rightarrow \Hbb(S_1 \times S_2 \times S_3 \times S_4)$
	\begin{equation*}
		\begin{tikzcd}
			((A_1 \boxtimes A_2) \boxtimes A_3) \boxtimes A_4 \arrow{d}{a} \arrow{rr}{a} && (A_1 \boxtimes A_2) \boxtimes (A_3 \boxtimes A_4) \arrow{d}{a} \\
			(A_1 \boxtimes (A_2 \boxtimes A_3)) \boxtimes A_4 \arrow{r}{a} & A_1 \boxtimes ((A_2 \boxtimes A_3) \boxtimes A_4) \arrow{r}{a} & A_1 \boxtimes (A_2 \boxtimes (A_3 \boxtimes A_4))
		\end{tikzcd}
	\end{equation*}
	is commutative.
	\item[($a$ETS-2)] For every three morphisms $f_i: T_i \rightarrow S_i$ in $\Scal$, $i = 1,2,3$, the diagram of functors $\Hbb(S_1) \times \Hbb(S_2) \times \Hbb(S_3) \rightarrow \Hbb(T_1 \times T_2 \times T_3)$
	\begin{equation*}
		\begin{tikzcd}
			(f_1^* A_1 \boxtimes f_2^* A_2) \boxtimes f_3^* A_3 \arrow{r}{m} \arrow{d}{a} & (f_1 \times f_2)^* (A_1 \boxtimes A_2) \boxtimes f_3^* A_3 \arrow{r}{m} & (f_1 \times f_2 \times f_3)^* ((A_1 \boxtimes A_2) \boxtimes A_3) \arrow{d}{a} \\
			f_1^* A_1 \boxtimes (f_2^* A_2 \boxtimes f_3^* A_3) \arrow{r}{m} & f_1^* A_1 \boxtimes (f_2 \times f_3)^* (A_2 \boxtimes A_3) \arrow{r}{m} & (f_1 \times f_2 \times f_3)^* (A_1 \boxtimes (A_2 \boxtimes A_3))  
		\end{tikzcd}
	\end{equation*}
	is commutative.
\end{enumerate}
In the same spirit, an \textit{external commutativity constraint} $c$ on $(\boxtimes,m)$ is the datum of
\begin{itemize}
	\item for every $S_1, S_2 \in \Scal$, a natural isomorphism between functors $\Hbb(S_1) \times \Hbb(S_2) \rightarrow \Hbb(S_1 \times S_2)$
	\begin{equation*}
		c = c_{S_1,S_2}: A_1 \boxtimes A_2 \xrightarrow{\sim} \tau^*(A_2 \boxtimes A_1),
	\end{equation*}
	where $\tau = \tau_{(12)}: S_1 \times S_2 \rightarrow S_2 \times S_1$ denotes the permutation isomorphism
\end{itemize}
satisfying the following conditions:
\begin{enumerate}
	\item[($c$ETS-1)]  For every $S_1, S_2 \in \Scal$, the diagram of functors $\Hbb(S_1) \times \Hbb(S_2) \rightarrow \Hbb(S_1 \times S_2)$
	\begin{equation*}
		\begin{tikzcd}
			A_1 \boxtimes A_2 \arrow{r}{c} \arrow[equal]{dr} & \tau^*(A_2 \boxtimes A_1) \arrow{d}{c} \\
			& \tau^* \tau^*(A_1 \boxtimes A_2)
		\end{tikzcd}
	\end{equation*}
	is commutative.
	\item[($c$ETS-2)] For every choice of of morphisms $f_i: T_i \rightarrow S_i$ in $\Scal$, $i = 1,2$, the diagram of functors $\Hbb(S_1) \times \Hbb(S_2) \rightarrow \Hbb(T_2 \times T_1)$
	\begin{equation*}
		\begin{tikzcd}
			f_1^* A_1 \boxtimes f_2^* A_2 \arrow{rr}{c} \arrow{d}{m} && \tau^* (f_2^* A_2 \boxtimes f_1^* A_1) \arrow{d}{m} \\
			(f_1 \times f_2)^* (A_1 \boxtimes A_2)\arrow{r}{c} & (f_1 \times f_2)^* \tau^* (A_2 \boxtimes A_1) \arrow[equal]{r} & \tau^* (f_2 \times f_1)^* (A_2 \boxtimes A_1)
		\end{tikzcd}
	\end{equation*}
	is commutative.
\end{enumerate}

The following result shows that external associativity and commutativity constraints behave well under the passage just described:

\begin{lem}\label{lem:adj-asso-comm}
	Let $\Hbb$ be a left-$\Scal'$-adjointable (resp. right-$\Scal'$-adjointable) $\Scal$-fibered category. Let $(\boxtimes,m)$ be a left-$\Scal'$-adjointable (resp. right-$\Scal'$-adjointable) external tensor structure on $\Hbb$, and let $(\boxtimes,\bar{m})$ denote the external tensor structure over ${\Scal'}^{op}$ obtained from it via \Cref{lem:ETS-adj}. Then:
	\begin{enumerate}
		\item Every external associativity constraint $a$ on $(\boxtimes,m)$ also defines an associativity constraint on $(\boxtimes,\bar{m})$.
		\item Every external commutativity constraint $c$ on $(\boxtimes,m)$ also defines a commutativity constraint on $(\boxtimes,\bar{m})$.
	\end{enumerate}
\end{lem}
\begin{proof}
	We only treat the right-adjointable case; the other case is analogous.
	\begin{enumerate}
		\item Let $a$ be an external associativity constraint on $(\boxtimes,m)$. We need to check that it satisfies conditions ($a$ETS-1) and ($a$ETS-2) with respect to $(\boxtimes,\bar{m})$. Since condition ($a$ETS-1) is clearly independent on whether we regard $\Hbb$ as a fibered category over $\Scal'$ or over ${\Scal'}^{op}$, we only need to check condition ($a$ETS-2): given morphisms $f_i: T_i \rightarrow S_i$ in $\Scal'$, $i = 1,2,3$, we have to show that the diagram of functors $\Hbb(T_1) \times \Hbb(T_2) \times \Hbb(T_3) \rightarrow \Hbb(S_1 \times S_2 \times S_3)$
		\begin{equation*}
			\begin{tikzcd}
				(f_{1,*} A_1 \boxtimes f_{2,*} A_2) \boxtimes f_{3,*} A_3 \arrow{r}{\bar{m}} \arrow{d}{a} & (f_1 \times f_2)_*(A_1 \boxtimes A_2) \boxtimes f_{3,*} A_3 \arrow{r}{\bar{m}} & (f_1 \times f_2 \times f_3)_* ((A_1 \boxtimes A_2) \boxtimes A_3)  \arrow{d}{a} \\					
				f_{1,*} A_1 \boxtimes (f_{2,*} A_2 \boxtimes f_{3,*} A_3) \arrow{r}{\bar{m}} & f_{1,*} A_1 \boxtimes (f_2 \times f_3)_*(A_2 \boxtimes A_3) \arrow{r}{\bar{m}} & (f_1 \times f_2 \times f_3)_*(A_1 \boxtimes (A_2 \boxtimes A_3))  
			\end{tikzcd}
		\end{equation*}
	    is commutative. To this end, we decompose it as
		\begin{equation*}
			\begin{tikzcd}[font=\tiny]
				\bullet \arrow{rr}{\bar{m}} \arrow{ddd}{a} \arrow{dr}{\eta} && \bullet \arrow{rr}{\bar{m}} && \bullet \arrow{ddd}{a} \\
				& f_{123,*} f_{123}^* ((B_1 \boxtimes B_2) \boxtimes B_3) \arrow{d}{a} & f_{123,*} (f_{12}^* (B_1 \boxtimes B_2) \boxtimes f_3^* B_3) \arrow{l}{m} & f_{123,*} ((f_1^* B_1 \boxtimes f_2^* B_2) \boxtimes f_3^* B_3) \arrow{d}{a} \arrow{l}{m} \arrow{ur}{\epsilon} \\
				& f_{123,*} f_{123}^* (B_1 \boxtimes (B_2 \boxtimes B_3)) & f_{123,*} (f_1^* B_1 \boxtimes f_{23}^* (B_2 \boxtimes B_3)) \arrow{l}{m} & f_{123,*} (f_1^* B_1 \boxtimes (f_2^* B_2 \boxtimes f_2^* B_3)) \arrow{l}{m} \arrow{dr}{\epsilon} \\
				\bullet \arrow{rr}{\bar{m}} \arrow{ur}{\eta} && \bullet \arrow{rr}{\bar{m}} && \bullet
			\end{tikzcd}
		\end{equation*}
		where, for notational convenience, we have set $B_i := f_i^* A_i$, $i = 1,2,3$. Since the left-most and rightmost pieces are commutative by naturality, while the central piece is commutative by axiom ($a$ETS-2) for $(\boxtimes,m)$, it now suffices to show that the upper and lower pieces are commutative as well. This can be checked directly by expanding the definition of the external monoidality isomorphisms $\bar{m}$ and filling the resulting diagram appropriately; we leave the details to the interested reader.
		\item Let $c$ be an external commutativity constraint on $(\boxtimes,m)$. We need to check that it satisfies conditions ($c$ETS-1) and ($c$ETS-2) with respect to $(\boxtimes,\bar{m})$. Since condition ($c$ETS-1) is clearly independent on whether we regard $\Hbb$ as a fibered category over $\Scal'$ or over ${\Scal'}^{op}$, we only need to check condition ($c$ETS-2): given morphisms $f_i: T_i \rightarrow S_i$ in $\Scal'$, $i = 1,2$, we have to show that the diagram of functors $\Hbb(T_1) \times \Hbb(T_2) \rightarrow \Hbb(S_1 \times S_2)$
		\begin{equation*}
			\begin{tikzcd}
				f_{1,*} A_1 \boxtimes f_{2,*} A_2 \arrow{rr}{c} \arrow{d}{\bar{m}} && \tau_*(f_{2,*} A_2 \boxtimes f_{1,*} A_1) \arrow{d}{\bar{m}} \\
				(f_1 \times f_2)_*(A_1 \boxtimes A_2)\arrow{r}{c} & (f_1 \times f_2)_* \tau_* (A_2 \boxtimes A_1) \arrow[equal]{r} & \tau_* (f_2 \times f_1)_*(A_2 \boxtimes A_1)
			\end{tikzcd}
		\end{equation*}
		(where $\tau = \tau_{(12)}: S_1 \times S_2 \xrightarrow{\sim} S_2 \times S_1$ and $\tau = \tau_{(12)}: T_1 \times T_2 \xrightarrow{\sim} T_2 \times T_1$ denote the permutation isomorphisms) is commutative. Expanding the definition of the horizontal arrows (and writing $f_{12} := f_1 \times f_2$ and $f_{21} := f_2 \times f_1$ for convenience) we obtain the outer part of the diagram
		\begin{equation*}
			\begin{tikzcd}[font=\small]
				f_{1,*} A_1 \boxtimes f_{2,*} A_2 \arrow{rr}{\eta} \arrow{d}{\eta} && \tau_* \tau^* (f_{1,*} A_1 \boxtimes f_{2,*} A_2) & \tau_* (f_{2,*} A_2 \boxtimes f_{1,*} A_1) \arrow{l}{c} \arrow{d}{\eta} \\
				f_{12,*} f_{12}^* (f_{1,*} A_1 \boxtimes f_{2,*} A_2) &&& \tau_* f_{21,*} f_{21}^* (f_2^* A_2 \boxtimes f_1^* A_1) \\
				f_{12,*} (f_1^* f_{1,*} A_1 \boxtimes f_2^* f_{2,*} A_2) \arrow{u}{m} \arrow{d}{\epsilon} &&& \tau_* f_{21,*} (f_2^* f_{2,*} A_2 \boxtimes f_1^* f_{1,*} A_1) \arrow{u}{m} \arrow{d}{\epsilon} \\
				f_{12,*} (A_1 \boxtimes A_2) \arrow{r}{\eta} & f_{12,*} \tau_* \tau^* (A_1 \boxtimes A_2) & f_{12,*} \tau_* (A_2 \boxtimes A_1) \arrow{l}{c} \arrow[equal]{r} & \tau_* f_{21,*} (A_2 \boxtimes A_1)
			\end{tikzcd}
		\end{equation*}
	    that we decompose as
	    \begin{equation*}
	    	\begin{tikzcd}[font=\tiny]
	    		\bullet \arrow{rr}{\eta} \arrow{d}{\eta} && \bullet \arrow{dr}{\eta} && \bullet \arrow{ll}{c} \arrow{d}{\eta} \\
	    		\bullet \arrow{r}{\eta} & f_{12,*} \tau_* \tau^* f_{12}^* (f_{1,*} A_1 \boxtimes f_{2,*} A_2) \arrow[equal]{r} & f_{12,*} \tau_* f_{21}^* \tau^* (f_{1,*} A_1 \boxtimes f_{2,*} A_2) \arrow[equal]{r} & \tau_* f_{21,*} f_{21}^* \tau^* (f_{1,*} A_1 \boxtimes f_{2,*} A_2) & \bullet \arrow{l}{c} \\
	    		&&& f_{12,*} \tau_* f_{21}^* (f_{2,*} A_2 \boxtimes f_{1,*} A_1) \arrow{ul}{c} \arrow[equal]{ur} \\
	    		\bullet \arrow{uu}{m} \arrow{d}{\epsilon} \arrow{r}{\eta} & f_{12,*} \tau_* \tau^* (f_1^* f_{1,*} A_1 \boxtimes f_2^* f_{2,*} A_2) \arrow{uu}{m} \arrow{dr}{\epsilon} && f_{12,*} \tau_* (f_2^* f_{2,*} A_2 \boxtimes f_1^* f_{1,*} A_1) \arrow{u}{m} \arrow{ll}{c} \arrow[equal]{r} \arrow{d}{\epsilon} & \bullet \arrow{uu}{m} \arrow{d}{\epsilon} \\
	    		\bullet \arrow{rr}{\eta} && \bullet & \bullet \arrow{l}{c} \arrow[equal]{r} & \bullet
	    	\end{tikzcd}
	    \end{equation*}
        Here, the central piece is commutative by axiom ($c$ETS-2) for $(\boxtimes,m)$ while the other pieces are commutative by naturality and by construction. This proves the claim.
	\end{enumerate}
\end{proof}

\begin{rem}
	In \cite[\S~6]{Ter24Fib} we introduced the natural compatibility condition ($ac$ETS) between external associativity and commutativity constraints on a given external tensor structure: this condition does not involve inverse image functors except those along permutation isomorphisms. As long as $\Scal'$ contains all (permutation) isomorphisms in $\Scal$, the validity of ($ac$ETS) does not depend on whether one regards $\Hbb$ as an $\Scal'$-fibered or as an ${\Scal'}^{op}$-fibered category.
\end{rem}

Finally, we combine the results collected so far for morphisms of $\Scal$-fibered categories and for external tensor structures in order to obtain a similar result for external tensor structures on morphisms.

For $j = 1,2$ let $\Hbb_j$ be an $\Scal$-fibered category endowed with an external tensor structure $(\boxtimes_j,m_j)$; moreover, let $R = (\Rcal,\theta): \Hbb_1 \rightarrow \Hbb_2$ be a morphism of $\Scal$-fibered categories. Recall from \cite[\S~8]{Ter24Fib} that an \textit{external tensor structure} $\rho$ on $R$ (with respect to $(\boxtimes_1,m_1)$ and $(\boxtimes_2,m_2)$) is the datum of
\begin{itemize}
	\item for every $S_1, S_2 \in \Scal$, a natural isomorphism of functors $\Hbb_1(S_1) \times \Hbb_1(S_2) \rightarrow \Hbb_2(S_1 \times S_2)$
	\begin{equation*}
		\rho = \rho_{S_1,S_2}: R_{S_1}(A_1) \boxtimes_2 R_{S_2}(A_2) \xrightarrow{\sim} R_{S_1 \times S_2}(A_1 \boxtimes_1 A_2),
	\end{equation*}
	called the \textit{external $R$-monoidality isomorphism} at $(S_1,S_2)$,
\end{itemize}
satisfying the following condition:
\begin{enumerate}
	\item[(mor-ETS)] For every two morphisms $f_i: T_i \rightarrow S_i$ in $\Scal$, $i = 1,2$, the natural diagram of functors $\Hbb_1(S_1) \times \Hbb_1(S_2) \rightarrow \Hbb_2(T_1 \times T_2)$
	\begin{equation*}
		\begin{tikzcd}
			f_1^* R_{S_1}(A_1) \boxtimes_2 f_2^* R_{S_2}(A_2) \arrow{r}{m_2} \arrow{d}{\theta} & (f_1 \times f_2)^* (R_{S_1}(A_1) \boxtimes_" R_{S_2}(A_2)) \arrow{r}{\rho} & (f_1 \times f_2)^* R_{S_1 \times S_2}(A_1 \boxtimes_1 A_2) \arrow{d}{\theta} \\
			R_{T_1}(f_1^* A_1) \boxtimes_2 R_{T_2}(f_2^* A_2) \arrow{r}{\rho} & R_{T_1 \times T_2}(f_1^* A_1 \boxtimes_1 f_2^* A_2) \arrow{r}{m_1} & R_{T_1 \times T_2}((f_1 \times f_2)^*(A_1 \boxtimes_1 A_2))
		\end{tikzcd}
	\end{equation*}
	is commutative.
\end{enumerate}

As usual, note that any external tensor structure on $R$ defines by restriction an external tensor structure on the underlying morphism of $\Scal'$-fibered categories. 

We do not need to introduce any notion of adjointability for external tensor structures on morphisms, as the following result shows:

\begin{lem}\label{lem:rho-adj}
	For $j = 1,2$ let $\Hbb_j$ be a left-$\Scal'$-adjointable (resp. right-$\Scal'$-adjointable) $\Scal$-fibered category endowed with a left-$\Scal'$-adjointable (resp. right-$\Scal'$-adjointable) external tensor structure $(\boxtimes_j,m_j)$, and let $(\boxtimes_j,\bar{m}_j)$ denote the external tensor structure on the ${\Scal'}^{op}$-fibered category $\Hbb$ obtained from it via \Cref{lem:ETS-adj}. 
	Moreover, let $R = (\Rcal,\theta): \Hbb_1 \rightarrow \Hbb_2$ be a left-$\Scal'$-adjointable (resp. right-$\Scal'$-adjointable) morphism of $\Scal$-fibered categories, and let $(\Rcal,\bar{\theta}): \Hbb_1 \rightarrow \Hbb_2$ denote the morphism of ${\Scal'}^{op}$-fibered categories obtained from it via \Cref{lem_R-adjoint}. 
	
	Suppose that we are given an external tensor structure $\rho$ on $(\Rcal,\theta)$ (with respect to $(\boxtimes_1,m_1)$ and $(\boxtimes_2,m_2)$). 
	Then $\rho$ also defines an external tensor structure on $(\Rcal,\bar{\theta})$ (with respect to the external tensor structures $(\boxtimes_1,\bar{m}_1)$ and $(\boxtimes_2,\bar{m}_2)$).
\end{lem}
\begin{proof}
	We only prove the thesis in the left-adjointable case; the other case is analogous. We need to check that condition (mor-ETS) is satisfied with respect to the morphism $(\Rcal,\bar{\theta})$: for every choice of morphisms $f_i: T_i \rightarrow S_i$ in $\Scal'$, $i = 1,2$, we have to show that the diagram of functors $\Hbb_1(T_1) \times \Hbb_1(T_2) \rightarrow \Hbb_2(S_1 \times S_2)$
	\begin{equation*}
		\begin{tikzcd}
			f_{1,\#} R_{T_1}(A_1) \boxtimes_2 f_{2,\#} R_{T_2}(A_2) \arrow{r}{\bar{m}_2} \arrow{d}{\bar{\theta}} & (f_1 \times f_2)_{\#} (R_{T_1}(A_1) \boxtimes_2 R_{T_2}(A_2)) \arrow{r}{\bar{\rho}} & (f_1 \times f_2)_{\#} R_{T_1 \times T_2}(A_1 \boxtimes_1 A_2) \arrow{d}{\bar{\theta}} \\
			R_{S_1}(f_{1,\#} A_1) \boxtimes_2 R_{S_2}(f_{2,\#} A_2) \arrow{r}{\bar{\rho}} & R_{S_1 \times S_2}(f_{1,\#} A_1 \boxtimes_1 f_{2,\#} A_2) \arrow{r}{\bar{m}_1} & R_{S_1 \times S_2}((f_1 \times f_2)_{\#} (A_1 \boxtimes_1 A_2))
		\end{tikzcd}
	\end{equation*}
    is commutative. Unwinding the various definitions, we obtain the more explicit diagram
    \begin{equation*}
    	\begin{tikzcd}[font=\tiny]
    		f_{12,\#} f_{12}^* (f_{1,\#} R_{T_1}(A_1) \boxtimes_2 f_{2,\#} R_{T_2}(A_2)) \arrow{d}{\epsilon} & f_{12,\#} (f_1^* f_{1,\#} R_{T_1}(A_1) \boxtimes_2 f_2^* f_{2,\#} R_{T_2}(A_2)) \arrow{l}{m_2} \arrow{r}{\epsilon} & f_{12,\#} (R_{T_1}(A_1) \boxtimes_2 R_{T_2}(A_2)) \arrow{d}{\rho} \\
    		f_{1,\#} R_{T_1}(A_1) \boxtimes_2 f_{2,\#} R_{T_2}(A_2) \arrow{d}{\eta} && f_{12,\#} R_{T_{12}}(A_1 \boxtimes_1 A_2) \arrow{d}{\eta} \\
    		f_{1,\#} R_{T_1}(f_1^* f_{1,\#} A_1) \boxtimes_2 f_{2,\#} R_{T_2}(f_2^* f_{2,\#} A_2) && f_{12,\#} R_{T_{12}}(f_{12}^* f_{12,\#} (A_1 \boxtimes_1 A_2)) \\
    		f_{1,\#} f_1^* R_{S_1}(f_{1,\#} A_1) \boxtimes_2 f_{2,\#} f_2^* R_{S_2}(f_{2,\#} A_2) \arrow{u}{\theta} \arrow{d}{\epsilon} && f_{12,\#} f_{12}^* R_{S_{12}}(f_{12,\#} (A_1 \boxtimes_1 A_2)) \arrow{u}{\theta} \arrow{d}{\epsilon} \\
    		R_{S_1}(f_{1,\#} A_1) \boxtimes_2 R_{S_2}(f_{2,\#} A_2) \arrow{d}{\rho} && R_{S_{12}}(f_{12,\#} (A_1 \boxtimes_1 A_2)) \\
    		R_{S_{12}}(f_{1,\#} A_1 \boxtimes_1 f_{2,\#} A_2) & R_{S_{12}}(f_{12,\#} f_{12}^* (f_{1,\#} A_1 \boxtimes_1 f_{2,\#} A_2)) \arrow{l}{\epsilon} & R_{S_{12}}(f_{12,\#} (f_1^* f_{1,\#} A_1 \boxtimes_1 f_2^* f_{2,\#} A_2)) \arrow{l}{m_1} \arrow{u}{\epsilon}
    	\end{tikzcd}
    \end{equation*}
    that we decompose as
    \begin{equation*}
    	\begin{tikzcd}[font=\small]
    		\bullet \arrow{d}{\epsilon} &&& \bullet \arrow{lll}{m_2} \arrow{r}{\epsilon} \arrow{ddll}{\bar{\theta}} & \bullet \arrow{d}{\rho} \arrow{dl}{\eta} \\
    		\bullet \arrow{d}{\eta} \arrow{r}{\eta} & f_{1,\#} f_1^* f_{1,\#} R_{T_1}(A_1) \boxtimes_2 f_{2,\#} f_2^* f_{2,\#} R_{T_2}(A_2) \arrow{urr}{\bar{m}_2} \arrow[bend right]{ddl}{\bar{\theta}} && f_{12,\#} (R_{T_1}(f_1^* f_{1,\#} A_1) \boxtimes_2 R_{T_2}(f_2^* f_{2,\#} A_2)) \arrow{d}{\rho} & \bullet \arrow{d}{\eta} \arrow{dl}{\eta} \\
    		\bullet & f_{12,\#} (f_1^* R_{S_1}(f_{1,\#} A_1) \boxtimes_2 f_2^* R_{S_2}(f_{2,\#} A_2)) \arrow{urr}{\theta} \arrow{d}{m_2} && f_{12,\#} R_{T_{12}}(f_1^* f_{1,\#} A_1 \boxtimes_1 f_2^* f_{2,\#} A_2) \arrow{d}{m_1} & \bullet \arrow[bend left]{ddl}{\bar{\theta}} \\
    		\bullet \arrow{u}{\theta} \arrow{d}{\epsilon} \arrow{ur}{\bar{m}_2} & f_{12,\#} f_{12}^* (R_{S_1}(f_{1,\#} A_1) \boxtimes_2 R_{S_2}(f_{2,\#} A_2)) \arrow{d}{\rho} \arrow{dl}{\epsilon} && f_{12,\#} R_{T_{12}}(f_{12}^* (f_{1,\#} A_1 \boxtimes_1 f_{2,\#} A_2)) \arrow{ddll}{\bar{\theta}} \arrow{ur}{\bar{m}_1} & \bullet \arrow{u}{\theta} \arrow{d}{\epsilon} \\
    		\bullet \arrow{d}{\rho} & f_{12,\#} f_{12}^* R_{S_{12}}(f_{1,\#} A_1 \boxtimes_1 f_{2,\#} A_2) \arrow{urr}{\theta} \arrow{dl}{\epsilon} && R_{S_{12}}(f_{12,\#} f_{12}^* f_{12,\#}(A_1 \boxtimes_1 A_2)) \arrow{dr}{\epsilon} & \bullet \\
    		\bullet & \bullet \arrow{l}{\epsilon} \arrow{urr}{\bar{m}_1} &&& \bullet \arrow{lll}{m_1} \arrow{u}{\epsilon}
    	\end{tikzcd}
    \end{equation*}
    Here, the central piece is commutative by axiom (mor-ETS) for $(\Rcal,\theta)$, while all the remaining pieces are commutative by naturality and by construction. This concludes the proof.
\end{proof}

\begin{rem}
	In the second part of \cite[\S~8]{Ter24Fib} we studied the natural compatibility conditions (mor-$a$ETS) and (mor-$c$ETS) between a given external tensor structure on the morphism $R: \Hbb_1 \rightarrow \Hbb_2$ and two given external associativity or commutativity constraints on $\Hbb_1$ and $\Hbb_2$, respectively: condition (mor-$a$ETS) does not involve any inverse image functors, while condition (mor-$c$ETS) only involves inverse images along permutation isomorphisms. As long as $\Scal'$ contains all (permutation) isomorphisms in $\Scal$, the validity of (mor-$a$ETS) and (mor-$c$ETS) does not depend on whether one regards $\Hbb_1$ and $\Hbb_2$ as $\Scal'$-fibered or as ${\Scal'}^{op}$-fibered categories.
\end{rem}

\section{The skeleton of a morphism of fibered categories}\label{sect_mor-skel}

In this section we focus on morphisms of $\Scal$-fibered categories: the main goal is to show how a given morphism can be recovered from a smaller amount of data involving inverse images under two special classes of morphisms only. 

\subsection{Axiomatics of factorization systems}

In the first place, we have to spell out the precise conditions on the two classes of morphisms involved in order to get good factorization properties for morphisms of $\Scal$. The assumption that $\Scal$ admits arbitrary finite products is not used for the results of this section; we do need, however, to require at least certain fibered products to exist in $\Scal$. In detail:

\begin{hyp}\label{hyp_skel}
	We assume to be given two (possibly non-full) subcategories $\Scal^{sm}$ and $\Scal^{cl}$ of $\Scal$ satisfying the following conditions:
	\begin{enumerate}
		\item[(i)] Both $\Scal^{sm}$ and $\Scal^{cl}$ have the same objects as $\Scal$ and contain all isomorphisms. 
		\item[(ii)] Let $p: P \rightarrow S$ be a morphism in $\Scal^{sm}$. Then, for every morphism $f: T \rightarrow S$ in $\Scal$, the fibered product $P_T := P \times_S T$ exists in $\Scal$. Moreover, for every Cartesian square in $\Scal$ of the form
		\begin{equation*}
			\begin{tikzcd}
				P_T \arrow{r}{f'} \arrow{d}{p'} & P \arrow{d}{p} \\
				T \arrow{r}{f} & S
			\end{tikzcd}
		\end{equation*} 
		with $p$ in $\Scal^{sm}$, the morphism $p'$ belongs to $\Scal^{sm}$ as well. 
		\item[(iii)] For every Cartesian square in $\Scal$ of the form
		\begin{equation*}
			\begin{tikzcd}
				P_Z \arrow{r}{z'} \arrow{d}{p'} & P \arrow{d}{p} \\
				Z \arrow{r}{z} & S
			\end{tikzcd}
		\end{equation*}
		with $p$ in $\Scal^{sm}$ and $z$ in $\Scal^{cl}$, the morphism $z'$ belongs to $\Scal^{cl}$ as well. Moreover, for every commutative diagram in $\Scal$ of the form
		\begin{equation*}
			\begin{tikzcd}
				Z \arrow{r}{z} \arrow[bend right]{dr}{gz} & S \arrow{d}{g} \\
				& V
			\end{tikzcd}
		\end{equation*}
		with $gz$ in $\Scal^{cl}$, the morphism $z$ belongs to $\Scal^{cl}$ as well.
		\item[(iv)] Every morphism $f: T \rightarrow S$ in $\Scal$ admits a factorization of the form
		\begin{equation*}
			f: T \xrightarrow{t} P \xrightarrow{p} S
		\end{equation*}
		with $t$ in $\Scal^{cl}$ and $p$ in $\Scal^{sm}$.
	\end{enumerate}
	We say that a morphism in $\Scal$ is a \textit{smooth morphism} if it belongs to $\Scal^{sm}$ and a \textit{closed immersion} if it belongs to $\Scal^{cl}$.
\end{hyp}

\begin{ex}
	The main example is when $\Scal = \Var_k$ is the category of quasi-projective algebraic varieties over a field $k$, and $\Scal^{sm}, \Scal^{cl}$ denote the usual subcategories of smooth morphisms and closed immersions, respectively. In this case, properties (i)--(iii) are well-known, while the validity of property (iv) is explained in \cite[\S~1.3.5]{Ayo07a}. 
	
	Of course, in the case $\Scal = \Var_k$ there is no reason to require only fibered products with respect to a smooth morphism to exist in condition (ii). We have chosen to use the weak formulation of (ii) not only because it suffices for all our constructions, but also in order to include the example when $\Scal = \Sm_k$ is the category of smooth quasi-projective $k$-varieties: indeed, fibered-products of smooth varieties need no be smooth in general, but they are so as soon as one of the two morphisms involved is smooth.
\end{ex}

For the rest of the present paper we work under \Cref{hyp_skel}; in \Cref{sect_mor-core} we will need to strengthen it slightly for some constructions, but for the moment we prefer to work under the minimal necessary assumptions.

Even if the results of the present section do not require any adjointability notions, we need to keep track of certain adjointability properties in view of later applications. To this end, it is convenient ot introduce some terminology:

\begin{defn}\label{defn:geom}
	\begin{enumerate}
		\item Let $\Hbb$ be an $\Scal$-fibered category.
		\begin{enumerate}
			\item We say that $\Hbb$ is \textit{smooth-adjointable} if it is left-$\Scal^{sm}$-adjointable in the sense of \Cref{defn:Sfib-adj},
			\item We say that $\Hbb$ is \textit{closed-adjointable} if it is right-$\Scal^{cl}$-adjointable in the sense of \Cref{defn:Sfib-adj}.
		\end{enumerate} 
	    \item Let $R: \Hbb_1 \rightarrow \Hbb_2$ be a morphism of $\Scal$-fibered categories.
	    \begin{enumerate}
		    \item In case $\Hbb_1$ and $\Hbb_2$ are smooth-adjointable, we say that $R$ is \textit{smooth-adjointable} if it is left-$\Scal^{sm}$-adjointable in the sense of \Cref{defn:mor-Sfib-adj}(1),
		    \item In case $\Hbb_1$ and $\Hbb_2$ are closed-adjointable, we say that $R$ is \textit{closed-adjointable} if it is right-$\Scal^{cl}$-adjointable in the sense of \Cref{defn:mor-Sfib-adj}(1).
	    \end{enumerate}
        Similarly for $\Scal$-structures on a collection $\Rcal = \left\{R_S: \Hbb_1 \rightarrow \Hbb_2\right\}_{S \in \Scal}$.
        \item Let $(\boxtimes,m)$ be an external tensor structure on an $\Scal$-fibered category $\Hbb$.
        \begin{enumerate}
        	\item In case $\Hbb$ is smooth-adjointable, we say that $(\boxtimes,m)$ is \textit{smooth-adjointable} if it is left-$\Scal^{sm}$-adjointable in the sense of \Cref{defn:ETS-adj}.
        	\item In case $\Hbb$ is closed-adjointable, we say that $(\boxtimes,m)$ is \textit{closed-adjointable} if it is left-$\Scal^{cl}$-adjointable in the sense of \Cref{defn:ETS-adj}.
        \end{enumerate}
	\end{enumerate}
\end{defn}

\subsection{Skeleta of morphisms}

For the rest of this section, we fix two $\Scal$-fibered categories $\Hbb_1$ and $\Hbb_2$ as well as a family of functors $\Rcal = \left\{R_S: \Hbb_1(S) \rightarrow \Hbb_2(S)\right\}_{S \in \Scal}$. 

Our goal is to turn $\Rcal$ into a morphism of $\Scal$-fibered categories by constructing an $\Scal$-structure $\theta$ on it; we want to achieve this in a relatively cheap way. 

The following definition is inspired by the method of \textit{exchange structures} developed by Ayoub in \cite[\S~1.2]{Ayo07a}; the main difference is that we apply it to morphisms rather than to single fibered categories, which considerably simplifies the discussion. 

\begin{defn}\label{defn:skel}
	An \textit{$\Scal$-skeleton} $(\theta^{sm},\theta^{cl})$ on $\Rcal$ is the datum of
	\begin{itemize}
		\item an $\Scal^{sm}$-structure $\theta^{sm}$ on $\Rcal$,
		\item an $\Scal^{cl}$-structure $\theta^{cl}$ on $\Rcal$
	\end{itemize}
	satisfying the following exchange condition:
	\begin{enumerate}
		\item[(ex-$\Scal$-fib)] For every commutative square in $\Scal$
		\begin{equation*}
			\begin{tikzcd}
				Q \arrow{r}{h} \arrow{d}{q} & P \arrow{d}{p} \\
				Z \arrow{r}{z} & S
			\end{tikzcd}
		\end{equation*}
		with $p,q$ smooth morphisms and $z,h$ closed immersions, the diagram of functors $\Hbb_1(S) \rightarrow \Hbb_2(Q)$
		\begin{equation*}
			\begin{tikzcd}
				q^* z^* R_S(A) \arrow{r}{\theta^{cl}} \arrow[equal]{d} & q^* R_Z (z^* A) \arrow{r}{\theta^{sm}} & R_Q (q^* z^* A) \arrow[equal]{d} \\
				h^* p^* R_S(A) \arrow{r}{\theta^{sm}} & h^* R_P (p^* A) \arrow{r}{\theta^{cl}} & R_Q (h^* p^* A)
			\end{tikzcd}
		\end{equation*}
		is commutative.
	\end{enumerate}
	In case the $\Scal$-fibered categories $\Hbb_1$ and $\Hbb_2$ are smooth-adjointable (in the sense of \Cref{defn:geom}(1)), we say that an $\Scal$-skeleton on $\Rcal$ is \textit{smooth-adjointable} if the underlying $\Scal^{sm}$-structure is left-adjointable in the sense of \Cref{defn:mor-Sfib-adj}(2).
\end{defn}

\begin{rem}\label{rem_theta-sm=thetca-cl}
	Let $f: T \rightarrow S$ be a morphism in $\Scal$ which is both a closed immersion and a smooth morphism. Recall from \cite[Rem.~1.7]{Ter24Fib} that axioms ($\Scal$-fib-0) and (mor-$\Scal$-fib) together imply formally the equality $\theta_{\id_S} = \id_{R_S}$.
	Taking this into account, condition (ex-$\Scal$-fib) applied to the commutative square
	\begin{equation*}
		\begin{tikzcd}
			T \arrow{r}{f} \arrow{d}{f} & S \arrow{d}{\id} \\
			S \arrow{r}{\id} & S
		\end{tikzcd}
	\end{equation*}
	asserts that the two natural isomorphisms $\theta_f: f^* \circ R_S \xrightarrow{\sim} R_T \circ f^*$ obtained by regarding $f$ either as a smooth morphism or as a closed immersion must in fact coincide.
\end{rem}

The following is the main result about $\Scal$-skeleta:

\begin{prop}\label{prop_mor-skel}
	Every $\Scal$-structure on $\Rcal$ determines by restriction an $\Scal$-skeleton on it. Conversely, every $\Scal$-skeleton on $\Rcal$ uniquely extends to an $\Scal$-structure on it.
	
	Moreover, if $\Hbb_1$ and $\Hbb_2$ are smooth-adjointable, then a given $\Scal$-structure on $\Rcal$ is smooth-adjointable if an only if the underlying $\Scal$-skeleton is smooth-adjointable.
\end{prop}

As the reader might expect, the main idea behind the proof is to exploit the factorization of morphisms in $\Scal$ described in \Cref{hyp_skel}(iv) as a way to define $R$-transition isomorphisms along arbitrary morphisms of $\Scal$. In order to make this idea rigorous, we need to introduce some notation:

\begin{nota}\label{nota:Fact}
	For every morphism $f: T \rightarrow S$ in $\Scal$, we introduce the \textit{factorization category} $\Fact_{\Scal}(f)$ where
	\begin{itemize}
		\item objects are triples $(P;t,p)$ where $P$ is an object of $\Scal$, $t: T \rightarrow P$ is a closed immersion, and $p: P \rightarrow S$ is a smooth morphism such that $p \circ t = f$,
		\item a morphism $q: (P';t',p') \rightarrow (P;t,p)$ is the datum of a smooth morphism $q: P' \rightarrow P$ such that $q \circ t' = t$ and $p \circ q = p'$.
	\end{itemize}
\end{nota}

In the first place, let us point out that the conditions of \Cref{hyp_skel} imply nice categorical properties of the factorization categories $\Fact_{\Scal}(f)$:

\begin{lem}\label{lem:Fact}
	For every morphism $f: T \rightarrow S$ in $\Scal$, the category $\Fact_{\Scal}(f)$ is non-empty and connected.
\end{lem}
\begin{proof}
	The fact that $\Fact_{\Scal}(f)$ is non-empty is \Cref{hyp_skel}(iv). To see that it is connected, we prove a more precise result, namely: every two objects $(P;t,p)$ and $(P';t',p')$ in $\Fact_{\Scal}(f)$ are dominated by a third one $(P'';t'',p'')$. Indeed, it suffices to take $P'' = P \times_S P'$ (which exists in $\Scal$ by \Cref{hyp_skel}(ii)) together with the obvious morphisms $t'': T \rightarrow P''$ (which is a closed immersion by \Cref{hyp_skel}(iii)) and $p'': P'' \rightarrow S$ (which is a smooth morphism by \Cref{hyp_skel}(ii)). 
\end{proof}

\begin{proof}[Proof of \Cref{prop_mor-skel}]
	Suppose that we are given an $\Scal$-structure $\theta$ on $\Rcal$, and let us check that its restrictions to $\Scal^{sm}$ and $\Scal^{cl}$ define an $\Scal$-skeleton. Since axioms (mor-$\Scal^{sm}$-fib) and (mor-$\Scal^{cl}$-fib) are both special cases of axiom (mor-$\Scal$-fib), we only need to check that condition (ex-$\Scal$-fib) holds: given a commutative square in $\Scal$
	\begin{equation*}
		\begin{tikzcd}
			Q \arrow{r}{h} \arrow{d}{q} & P \arrow{d}{p} \\
			Z \arrow{r}{z} & S
		\end{tikzcd}
	\end{equation*}
	with $p,q$ smooth morphisms and $z,h$ closed immersions, we have to show that the diagram of functors $\Hbb_1(S) \rightarrow \Hbb_2(Q)$
	\begin{equation*}
		\begin{tikzcd}
			q^* z^* R_S(A) \arrow{r}{\theta} \arrow[equal]{d} & q^* R_Z (z^* A) \arrow{r}{\theta} & R_Q (q^* z^* A) \arrow[equal]{d} \\
			h^* p^* R_S(A) \arrow{r}{\theta} & h^* R_P (p^* A) \arrow{r}{\theta} & R_Q (h^* p^* A)
		\end{tikzcd}
	\end{equation*}
	is commutative. We claim that, in fact, the conclusion is true with no hypothesis on the morphisms $z,h,p,q$. Indeed, if $f$ denotes the composite morphism $p \circ h = z \circ q$, the above diagram coincides by definition with the outer rectangle of the diagram
	\begin{equation*}
		\begin{tikzcd}[font=\small]
			q^* z^* R_S(A) \arrow{r}{\theta} \arrow[equal]{d} & q^* R_Z (z^* A) \arrow{r}{\theta} & R_Q(q^* z^* A) \arrow[equal]{d} \\
			f^* R_S(A) \arrow{rr}{\theta} && R_Q(f^* A) \\
			h^* p^* R_S(A) \arrow[equal]{u} \arrow{r}{\theta} & h^* R_P(p^* A) \arrow{r}{\theta} & R_Q(h^* p^* A) \arrow[equal]{u}
		\end{tikzcd}
	\end{equation*}
	where both pieces are commutative by axiom (mor-$\Scal$-fib).
	
	Conversely, suppose that we are given an $\Scal$-skeleton $(\theta^{sm},\theta^{cl})$ on $\Rcal$, and let us show that it extends uniquely to an $\Scal$-structure $\theta$: we have to construct $R$-transition isomorphisms along arbitrary morphisms of $\Scal$ and check that they satisfy condition (mor-$\Scal$-fib). We divide the construction into two main steps:
	\begin{enumerate}
		\item[(Step 1)] To begin with, fix an arrow $f: T \rightarrow S$ in $\Scal$. For every object $(P;t,p) \in \Fact_{\Scal}(f)$, consider the natural isomorphism of functors $\Hbb_1(S) \rightarrow \Hbb_2(T)$
		\begin{equation*}
			\theta_f^{(t,p)}: f^* R_S(A) = t^* p^* R_S(A) \xrightarrow{\theta^{sm}_p} t^* R_P(p^* A) \xrightarrow{\theta^{cl}_t} R_T(t^* p^* A) = R_T(f^* A).
		\end{equation*}
		We see directly from axiom (mor-$\Scal$-fib) that, if an $\Scal$-structure on $\Rcal$ extending the data exists, then the $R$-transition isomorphism $\theta_f$ necessarily coincides with $\theta_f^{(t,p)}$: this implies that the sought-after $\Scal$-structure $\theta$ on $\Rcal$ is unique, provided it exists.
		
		We claim that the natural isomorphism $\theta_f^{(t,p)}$ does not depend on the choice of $(P;t,p) \in \Fact_{\Scal}(f)$. To this end, by the connectedness of $\Fact_{\Scal}(f)$ (\Cref{lem:Fact}), it suffices to show that, given an arrow $q: (P';t',p') \rightarrow (P;t,p)$ in $\Fact_{\Scal}(f)$, the equality $\theta_f^{(t',p')} = \theta_f^{(t,p)}$ holds. By definition, this amounts to showing that the outer part of the diagram
		\begin{equation*}
			\begin{tikzcd}[font=\small]				
				& t^* p^* R_S(A) \arrow{r}{\theta^{sm}} \arrow[equal]{d} & t^* R_P(p^* A) \arrow{rr}{\theta^{cl}} \arrow[equal]{d} && R_T(t^* p^* A) \arrow[equal]{d} \\
				f^* R_S(A) \arrow[equal]{ur} \arrow[equal]{dr} & {t'}^* q^* p^* R_S(A) \arrow{r}{\theta^{sm}} & {t'}^* q^* R_P(p^* A) \arrow{r}{\theta^{sm}} & {t'}^* R_{P'}(q^* p^* A) \arrow{r}{\theta^{cl}} & R_T({t'}^* q^* p^* A) & R_S(f^* A) \arrow[equal]{ul} \arrow[equal]{dl} \\
				& {t'}^* {p'}^* R_S(A) \arrow{rr}{\theta^{sm}} \arrow[equal]{u} && {t'}^* R_{P'}({p'}^* A) \arrow{r}{\theta^{cl}} \arrow[equal]{u} & R_T({t'}^* {p'}^* A) \arrow[equal]{u}
			\end{tikzcd}
		\end{equation*}
		is commutative. But the left-most and right-most pieces are commutative by \cite[Lemma~1.3]{Ter24Fib}, the two squares are commutative by naturality, the lower-left rectangle is commutative by axiom (mor-$\Scal^{sm}$-fib), while the upper-right rectangle is commutative by axiom (ex-$\Scal$-fib) applied to the commutative square
		\begin{equation*}
			\begin{tikzcd}
				T \arrow{r}{t'} \arrow{d}{\id_T} & P' \arrow{d}{q} \\
				T \arrow{r}{t} & P
			\end{tikzcd}
		\end{equation*}
		taking into account the equality $\theta_{\id_T} = \id_{R_T}$ (\Cref{rem:theta}). Thus we get a well-defined natural isomorphism of functors $\Hbb_1(S) \rightarrow \Hbb_2(T)$
		\begin{equation}\label{theta_f:proof}
			\theta = \theta_f: f^* R_S(A) \xrightarrow{\sim} R_T(f^* A).
		\end{equation}
		\item[(Step 2)] In order to prove that the natural isomorphisms \eqref{theta_f:proof} define an $\Scal$-structure on $\Rcal$, we need to check that they satisfy axiom (mor-$\Scal$-fib): given two composable morphisms $f: T \rightarrow S$ and $g: S \rightarrow V$ in $\Scal$, we have to show that the diagram of functors $\Hbb_1(V) \rightarrow \Hbb_2(T)$
		\begin{equation*}
			\begin{tikzcd}
				(gf)^* R_V(A) \arrow{rr}{\theta} \arrow[equal]{d} && R_T ((gf)^* A) \arrow[equal]{d} \\
				f^* g^* R_V(A) \arrow{r}{\theta} & f^* R_S (g^* A) \arrow{r}{\theta} & R_T (f^* g^* A)
			\end{tikzcd}
		\end{equation*}
		is commutative. To this end, choose factorizations $(P;t,p) \in \Fact_{\Scal}(f)$ and $(Q;s,q) \in \Fact_{\Scal}(g)$, and then choose a further factorization $(L;h,l) \in \Fact_{\Scal}(s \circ p)$; in this way, we also obtain the factorization $(L;ht,ql) \in \Fact_{\Scal}(g \circ f)$.
		Expanding the definition of $\theta_f = \theta_f^{(t,p)}$ and $\theta_g = \theta_g^{(s,q)}$, we can express the above diagram more explicitly as the outer part of the diagram
		\begin{equation*}
			\begin{tikzcd}[font=\small]
				(gf)^* R_V(A) \arrow[equal]{r} & (ht)^* (ql)^* R_V(A) \arrow{r}{\theta^{sm}} & (ht)^* R_L((ql)^* A) \arrow{r}{\theta^{cl}} & R_T((ht)^* (ql)^* A) & R_T((gf)^* A) \arrow[equal]{l} \\
				f^* g^* R_V(A) \arrow[equal]{u} \arrow[equal]{d} & t^* h^* l^* q^* R_V(A) \arrow[equal]{u} \arrow[equal]{dl} && R_T(t^* h^* l^* q^* A) \arrow[equal]{u} \arrow[equal]{dr} & R_T(f^* g^* A) \arrow[equal]{u} \arrow[equal]{d} \\
				t^* p^* s^* q^* R_V(A) \arrow{r}{\theta^{sm}} & t^* p^* s^* R_Q(q^* A) \arrow{r}{\theta^{cl}} & t^* p^* R_S(s^* q^* A) \arrow{r}{\theta^{sm}} & t^* R_P(p^* s^* q^* A) \arrow{r}{\theta^{cl}} & R_T(t^* p^* s^* q^* A)	
			\end{tikzcd}
		\end{equation*}
		where the two lateral pieces are commutative by \cite[Lemma~1.3]{Ter24Fib}. Hence we are reduced to showing the commutativity of the central piece of the diagram. The latter is the outer rectangle of the diagram
		\begin{equation*}
			\begin{tikzcd}[font=\small]
				(ht)^* (ql)^* R_V(A) \arrow{rr}{\theta^{sm}} && (ht)^* R_L((ql)^* A) \arrow{rr}{\theta^{cl}} && R_T((ht)^* (ql)^* A) \\
				t^* h^* l^* q^* R_V(A) \arrow{r}{\theta^{sm}} \arrow[equal]{u} \arrow[equal]{d} & t^* h^* l^* R_Q(q^* A) \arrow{r}{\theta^{sm}} \arrow[equal]{d} & t^* h^* R_L(l^* q^* A) \arrow{r}{\theta^{cl}} \arrow[equal]{u} & t^* R_P(h^* l^* q^* A) \arrow{r}{\theta^{cl}} \arrow[equal]{d} & R_T(t^* h^* l^* q^* A) \arrow[equal]{u} \arrow[equal]{d} \\
				t^* p^* s^* q^* R_V(A) \arrow{r}{\theta^{sm}} & t^* p^* s^* R_Q(q^* A) \arrow{r}{\theta^{cl}} & t^* p^* R_S(s^* q^* A) \arrow{r}{\theta^{sm}} & t^* R_P(p^* s^* q^* A) \arrow{r}{\theta^{cl}} & R_T(t^* p^* s^* q^* A)
			\end{tikzcd}
		\end{equation*}
		where the two upper pieces are commutative by axioms (mor-$\Scal^{sm}$-fib) and (mor-$\Scal^{cl}$-fib), the two lateral lower pieces are commutative by naturality, and the central lower piece is commutative by axiom (ex-$\Scal$-fib).
	\end{enumerate}
	Finally, the fact that an $\Scal$-structure on the family $\Rcal$ is smooth-adjointable if and only if so is the underlying $\Scal$-skeleton is obvious in view of \Cref{rem_theta-sm=thetca-cl}. This concludes the proof.
\end{proof}

We conclude with the following observation, which will be useful in the course of the next section:

\begin{lem}\label{lem_Scl-str}
	Let $\theta^{sm}$ and $\theta^{cl}$ be an $\Scal^{sm}$-structure and an $\Scal^{cl}$-structure on $\Rcal$. Then the pair $(\theta^{sm},\theta^{cl})$ defines an $\Scal$-skeleton if and only if it satisfies the following two conditions:
	\begin{enumerate}
		\item[(C-$\Scal$-fib)] For every Cartesian square in $\Scal$
		\begin{equation*}
			\begin{tikzcd}
				P_Z \arrow{r}{z'} \arrow{d}{p'} & P \arrow{d}{p} \\
				Z \arrow{r}{z} & S
			\end{tikzcd}
		\end{equation*}
		with $p$ a smooth morphism and $z$ a closed immersion, the diagram of functors $\Hbb_1(S) \rightarrow \Hbb_2(P_Z)$
		\begin{equation*}
			\begin{tikzcd}
				{p'}^* z^* R_S(A) \arrow{r}{\theta^{cl}} \arrow[equal]{d} & {p'}^* R_Z (z^* A) \arrow{r}{\theta^{sm}} & R_{P_Z} ({p'}^* z^* A) \arrow[equal]{d} \\
				{z'}^* p^* R_S(A) \arrow{r}{\theta^{sm}} & {z'}^* R_P (p^* A) \arrow{r}{\theta^{sm}} & R_{P_Z} ({z'}^* p^* A)
			\end{tikzcd}
		\end{equation*}
		is commutative.
		\item[(T-$\Scal$-fib)] For every commutative triangle in $\Scal$
		\begin{equation*}
			\begin{tikzcd}
				Q \arrow{r}{h} \arrow{dr}{q} & P \arrow{d}{p} \\
				& S
			\end{tikzcd}
		\end{equation*}
		with $p,q$ smooth morphisms and $h$ a closed immersion, the diagram of functors $\Hbb_1(S) \rightarrow \Hbb_2(Q)$
		\begin{equation*}
			\begin{tikzcd}
				q^* R_S(A) \arrow{rr}{\theta^{sm}} \arrow[equal]{d} && R_Q (q^* A) \arrow[equal]{d} \\
				h^* p^* R_S(A) \arrow{r}{\theta^{sm}} & h^* R_P (p^* A) \arrow{r}{\theta^{cl}} & R_S (h^* p^* A)
			\end{tikzcd}
		\end{equation*}
		is commutative.
	\end{enumerate}
\end{lem}
\begin{proof}
	By definition, the pair $(\theta^{sm},\theta^{cl})$ defines an $\Scal$-skeleton if and only if it satisfies condition (ex-$\Scal$-fib) from \Cref{defn:skel}.	
	
	Condition (C-$\Scal$-fib) is the restriction of (ex-$\Scal$-fib) to the class of Cartesian squares, while condition (T-$\Scal$-fib) is equivalent to the restrictions of (ex-$\Scal$-fib) to the class of squares of the form
	\begin{equation*}
		\begin{tikzcd}
			Q \arrow{r}{h} \arrow{d}{q} & P \arrow{d}{p} \\
			S \arrow{r}{\id_S} & S.
		\end{tikzcd}
	\end{equation*}
	Thus condition (ex-$\Scal$-fib) trivially implies both (C-$\Scal$-fib) and (T-$\Scal$-fib).
	
	Conversely, suppose that the latter two conditions are satisfied, and let us show that condition (ex-$\Scal$-fib) holds: given a commutative square
	\begin{equation*}
		\begin{tikzcd}
			Q \arrow{r}{h} \arrow{d}{q} & P \arrow{d}{p} \\
			Z \arrow{r}{z} & S
		\end{tikzcd}
	\end{equation*}
	with $z,h$ closed immersions and $p,q$ smooth morphisms, we have to show that the diagram of functors $\Hbb_1(S) \rightarrow \Hbb_2(Q)$
	\begin{equation}\label{dia:qzhp}
		\begin{tikzcd}
			q^* z^* R_S(A) \arrow{r}{\theta^{cl}} \arrow[equal]{d} & q^* R_Z (z^* A) \arrow{r}{\theta^{sm}} & R_Q (q^* z^* A) \arrow[equal]{d} \\
			h^* p^* R_S(A) \arrow{r}{\theta^{sm}} & h^* R_P (p^* A) \arrow{r}{\theta^{cl}} & R_Q (h^* p^* A)
		\end{tikzcd}
	\end{equation} 
    is commutative. To this end, form the Cartesian square
	\begin{equation*}
		\begin{tikzcd}
			P_Z \arrow{r}{z'} \arrow{d}{p'} & P \arrow{d}{p} \\
			Z \arrow{r}{z} & S
		\end{tikzcd}
	\end{equation*}
	(which exists by \Cref{hyp_skel}(ii)) and let $i: Q \rightarrow P_Z$ be the morphism induced by $q$ and $h$ (which is a closed immersion by \Cref{hyp_skel}(iii)); in this way, we obtain the commutative diagram
	\begin{equation*}
		\begin{tikzcd}
			Q \arrow[bend right]{ddr}{q} \arrow[bend left]{drr}{h} \arrow{dr}{i} \\
			& P_Z \arrow{r}{z'} \arrow{d}{p'} & P \arrow{d}{p} \\
			& Z \arrow{r}{z} & S.
		\end{tikzcd}
	\end{equation*}
    We can then decompose the diagram \eqref{dia:qzhp} above as
    \begin{equation*}
    	\begin{tikzcd}[font=\small]
    		\bullet \arrow[equal]{ddd} \arrow[equal]{dr} \arrow{rr}{\theta^{cl}} && \bullet \arrow{rrr}{\theta^{sm}} \arrow[equal]{d} &&& \bullet \arrow[equal]{ddd} \arrow[equal]{dl} \\
    		& i^* {p'}^* z^* R_S(A) \arrow{r}{\theta^{cl}} \arrow[equal]{d} & i^* {p'}^* R_Z(z^* A) \arrow{r}{\theta^{sm}} & i^* R_{P_Z} ({p'}^* z^* A) \arrow{r}{\theta^{cl}} \arrow[equal]{d} & R_Q (i^* {p'}^* z^* A) \arrow[equal]{d} \\
    		& i^* {z'}^* p^* R_S(A) \arrow{r}{\theta^{sm}} & i^* {z'}^* R_P (p^* A) \arrow{r}{\theta^{cl}} & i^* R_{P_Z} ({z'}^* p^* A) \arrow{r}{\theta^{cl}} & R_Q (i^* {z'}^* p^* A) \\
    		\bullet \arrow[equal]{ur} \arrow{rr}{\theta^{sm}} && \bullet \arrow{rrr}{\theta^{cl}} \arrow[equal]{u} &&& \bullet \arrow[equal]{ul}
    	\end{tikzcd}
    \end{equation*}
	Here, the left-most and right-most pieces are commutative by \cite[Lemma~1.3]{Ter24Fib}, the upper-left and lower-left pieces are commutative by naturality, the upper-right piece is commutative by axiom (T-$\Scal$-fib), the lower-right piece is commutative by axiom ($\Scal^{cl}$-fib-1), the central rectangle is commutative by axiom (T-$\Scal$-fib), and the central square is commutative again by naturality. This concludes the proof.
\end{proof}

\section{The core of a localic morphism of fibered categories}\label{sect_mor-core}

In this section, we introduce a variant of the notion of $\Scal$-skeleton which is more adapted to the setting of perverse sheaves: we want to rephrase the axioms of $\Scal$-skeleta in such a way that inverse images under closed immersions are replaced by the corresponding direct images. It turns out that one still needs to use inverse images under closed immersions at some point; however, this does not cause problems in our applications.

\subsection{Localic fibered categories and localic morphisms}

First of all, we need to introduce some complements to the previous hypotheses about the base category $\Scal$ and its subcategories $\Scal^{cl}$ and $\Scal^{sm}$. In detail:

\begin{hyp}\label{hyp_core}
	We keep the assumptions and notation of \Cref{hyp_skel}. In addition, we assume the following:
	\begin{enumerate}
		\item[(i)] The category $\Scal$ has an initial object $\emptyset$. 
		\item[(ii)] For every closed immersion $z: Z \rightarrow S$, there exists a distinguished smooth morphism $u: U \rightarrow S$ with the property that, for every $T \in \Scal$, the map
		\begin{equation*}
			u \circ -: \Hom_{\Scal}(T,U) \rightarrow \Hom_{\Scal}(T,S)
		\end{equation*} 
		identifies $\Hom_{\Scal}(T,U)$ with the subset of $\Hom_{\Scal}(T,S)$ consisting of all those morphisms $f: T \rightarrow S$ for which the diagram
		\begin{equation*}
			\begin{tikzcd}
				\emptyset \arrow{r} \arrow{d} & T \arrow{d}{f} \\
				Z \arrow{r}{z} & S		
			\end{tikzcd}
		\end{equation*}
		is Cartesian. We call $u$ the \textit{open immersion} complementary to $z$, and we call $U$ the \textit{open complement} of $Z$ in $S$.
	\end{enumerate}
\end{hyp}

\begin{rem}\label{rem:Hloc-u-mono}
	Note that, for every closed immersion $z: Z \rightarrow S$ in $\Scal$, the complementary open immersion $u: U \rightarrow S$ is a monomorphism in $\Scal$ by definition.
\end{rem}

\begin{ex}
	As for \Cref{hyp_skel}, the main examples of interest for us are when $\Scal$ is one among the category $\Var_k$ of quasi-projective $k$-varieties or its subcategory $\Sm_k$ of smooth $k$-varieties; in both cases, $\Scal^{sm}$ and $\Scal^{cl}$ are the subcategories of smooth morphisms and closed immersions, respectively. 
\end{ex}

Throughout this section, we work under \Cref{hyp_skel} and \Cref{hyp_core}. 

While the results of \Cref{sect_mor-skel} apply to general $\Scal$-fibered categories, the results presented below are specific to the setting of triangulated $\Scal$-fibered categories (as axiomatized in \cite[\S~9]{Ter24Fib}). In fact, our main applications concern the theory of stable homotopy $2$-functors developed in \cite{Ayo07a} and its variant based on coefficient systems in \cite{DrewGal}. In order not to discuss the specifics of these theories here, we have chosen to summarize the part of the axioms of \cite[Defn.~1.4.1]{Ayo07a} revolving around open and closed immersions in the following self-contained definition:

\begin{defn}\label{defn_Hlocal}
	We say that a triangulated $\Scal$-fibered category $\Hbb$ is \textit{localic} if it satisfies the following conditions:
	\begin{enumerate}
		\item[(a)] $\Hbb(\emptyset)$ is the zero triangulated category.
		\item[(b)] $\Hbb$ is smooth-adjointable in the sense of \Cref{defn:geom}(1)(a). Moreover, for every Cartesian square in $\Scal$
		\begin{equation*}
			\begin{tikzcd}
				P_T \arrow{r}{f'} \arrow{d}{p'} & P \arrow{d}{p} \\
				T \arrow{r}{f} & S
			\end{tikzcd}
		\end{equation*}
		with $p$ a smooth morphism, the natural transformation of functors $\Hbb(P) \rightarrow \Hbb(T)$
		\begin{equation*}
			p'_{\#} {f'}^* A \xrightarrow{\eta} p'_{\#} {f'}^* p^* p_{\#} A = p'_{\#} {p'}^* f^* p_{\#} A \xrightarrow{\epsilon} f^* p_{\#} A
		\end{equation*}
		is invertible.
		\item[(c)] $\Hbb$ is closed-adjointable in the sense of \Cref{defn:geom}(1)(b). Moreover, for every closed immersion $z: Z \rightarrow S$ in $\Scal$, the functor $z_*: \Hbb(Z) \rightarrow \Hbb(S)$ is fully faithful.
		\item[(d)] For every closed immersion $z: Z \rightarrow S$ in $\Scal$ with complementary open immersion $u: U \rightarrow S$, the pair of functors $(z^*,u^*)$ is conservative.
	\end{enumerate}
	Via \Cref{lem_H-adjoint}, we also regard the underlying $\Scal^{sm}$-fibered category and the underlying $\Scal^{cl}$-fibered category of $\Hbb$ as an $\Scal^{sm,op}$-fibered category and an $\Scal^{cl,op}$-fibered category, respectively.
\end{defn}

\begin{ex}\label{ex:st2f}
	In the case where $\Scal = \Var_k$ (with $\Scal^{cl}$ and $\Scal^{sm}$ the subcategories of closed immersions and smooth morphisms, respectively), the underlying $\Scal$-fibered category of any stable homotopy $2$-functor is localic in virtue of the so-called \textit{localization axiom}; see \cite[\S~1.4.1]{Ayo07a}. Our axiomatic approach allows us to treat the localization properties of 'stable homotopy $2$-functors over $\Sm_k$' as well.
\end{ex}

The next lemma summarizes the basic properties of localic $\Scal$-fibered categories. In the setting of stable homotopy $2$-functors, all these results are established in \cite[\S~1.4.4]{Ayo07a}; we just rearrange the argument appearing there in order to make sure that the same results hold in our setting.

\begin{lem}\label{lem_Hlocal}
	Let $\Hbb$ be a localic triangulated $\Scal$-fibered category. Let $z: Z \rightarrow S$ be a closed immersion in $\Scal$ with complementary open immersion $u: U \rightarrow S$. Then:
	\begin{enumerate}
		\item The functor $u_{\#}: \Hbb(U) \rightarrow \Hbb(S)$ is fully faithful.
		\item The composite functors $z^* \circ u_{\#}$ and $u^* \circ z_*$ vanish.
		\item There exists a canonical localization distinguished triangle of functors $\Hbb(S) \rightarrow \Hbb(S)$
		\begin{equation*}
			u_{\#} u^* A \xrightarrow{\epsilon} A \xrightarrow{\eta} z_* z^* A \xrightarrow{+}.
		\end{equation*}
		\item For every Cartesian square in $\Scal$ of the form
		\begin{equation*}
			\begin{tikzcd}
				P_Z \arrow{r}{z'} \arrow{d}{p_Z} & P \arrow{d}{p} \\
				Z \arrow{r}{z} & S
			\end{tikzcd}
		\end{equation*}
		with $p$ a smooth morphism and $z$ a closed immersion, the natural transformation of functors $\Hbb(Z) \rightarrow \Hbb(P)$
		\begin{equation*}
			p^* z_* A \xrightarrow{\eta} z'_* {z'}^* p^* z_* A = z'_* p_Z^* z^* z_* A \xrightarrow{\epsilon} z'_* p_Z^* A
		\end{equation*}
		is invertible.
	\end{enumerate}
\end{lem}
\begin{proof}
	\begin{enumerate}
		\item Saying that $u_{\#}$ is fully faithful is equivalent to saying that the unit natural transformation $\eta: \id_{\Hbb(S)} \rightarrow u^* u_{\#}$ is invertible. In order to show that this is the case, it suffices to apply \Cref{defn_Hlocal}(b) to the square
		\begin{equation*}
			\begin{tikzcd}
				U \arrow{r}{\id_U} \arrow{d}{\id_U} & U \arrow{d}{u} \\
				U \arrow{r}{u} & S
			\end{tikzcd}
		\end{equation*}
		which is Cartesian since $U$ is a monomorphism (\Cref{rem:Hloc-u-mono}).
		\item To prove the vanishing of $z^* \circ u_{\#}$, it suffices to apply \Cref{defn_Hlocal}(a)(b) to the Cartesian square
		\begin{equation*}
			\begin{tikzcd}
				\emptyset \arrow{r} \arrow{d} & U \arrow{d}{u} \\
				Z \arrow{r}{z} & S.
			\end{tikzcd}
		\end{equation*}
		The vanishing of $u^* \circ z_*$ follows from this by adjunction.
		\item For every object $A \in \Hbb(S)$, complete the co-unit morphism $\epsilon: u_{\#} u^* A \rightarrow A$ to a distinguished triangle in $\Hbb(S)$
		\begin{equation*}
			u_{\#} u^* A \xrightarrow{\epsilon} A \rightarrow C(A) \xrightarrow{+1}.
		\end{equation*}
		The fully faithfulness of the functor $u_{\#}$, together with the triangular identity, implies that the morphism $u^* u_{\#} u^* A \xrightarrow{\epsilon} u^* A$ is an isomorphism inverse to $u^* A \xrightarrow{\eta} u^* u_{\#} u^* A$: hence the object $u^* C(A)$, being the cone of this morphism, must vanish.
		As a consequence, for every morphism $\phi: A \rightarrow B$ in $\Hbb(S)$, applying the criterion of \cite[Prop.~1.1.9]{BBD82} to the vanishing
		\begin{equation*}
			\Hom_{\Hbb(S)}(u_{\#} u^* A, C(B)[-1]) = \Hom_{\Hbb(U)}(u^* A, u^* C(B)[-1]) = 0,
		\end{equation*}
		one deduces the existence of a unique morphism $C(\phi): C(A) \rightarrow C(B)$ fitting into a morphism of triangles of the form
		\begin{equation*}
			\begin{tikzcd}
				u_{\#} u^* A \arrow{r}{\epsilon} \arrow{d}{u_{\#} u^* \phi} & A \arrow{r} \arrow{d}{\phi} & C(A) \arrow{r}{+1} \arrow{d}{C(\phi)} & {} \\
				u_{\#} u^* B \arrow{r}{\epsilon} & B \arrow{r} & C(B) \arrow{r}{+1} & {}. \\
			\end{tikzcd}
		\end{equation*}
		Hence the assignment $A \rightsquigarrow C(A)$ defines a functor $C: \Hbb(S) \rightarrow \Hbb(S)$ fitting into a canonical distinguished triangle of functors $\Hbb(S) \rightarrow \Hbb(S)$
		\begin{equation*}
			u_{\#} u^* A \xrightarrow{\epsilon} A \rightarrow C(A) \xrightarrow{+1}.
		\end{equation*}
		Note that the composite natural isomorphism of functors $\Hbb(S) \rightarrow \Hbb(S)$
		\begin{equation*}
			u_{\#} u^* A \xrightarrow{\epsilon} A \xrightarrow{\eta} z_* z^* A
		\end{equation*}
		vanishes, since we have
		\begin{equation}\label{Hloc-z^*u_sh=0}
			\Hom_{\Hbb(S)}(u_{\#} u^* A, z_* z^* A) = \Hom_{\Hbb(Z)}(z^* u_{\#} u^* A, z^* A) = 0
		\end{equation}
		as $z^* \circ u_{\#} = 0$. Therefore, for every $A \in \Hbb(S)$, there exists a morphism $\psi(A): C(A) \rightarrow z_* z^* A$ fitting into a morphism of triangles
		\begin{equation*}
			\begin{tikzcd}
				u_{\#} u^* A \arrow{r}{\epsilon} \arrow{d} & A \arrow{r} \arrow{d}{\eta} & C(A) \arrow{r}{+1} \arrow{d}{\psi(A)} & {} \\
				0 \arrow{r} & z_* z^* A \arrow{r} & z_* z^* A \arrow{r}{+1} & {}.
			\end{tikzcd}
		\end{equation*}
		Applying the criterion of \cite[Prop.~1.1.9]{BBD82} to the vanishing \eqref{Hloc-z^*u_sh=0}, one deduces that such a morphism $\psi(A)$ is unique and, similarly, one show that it is functorial in $A$. Thus one obtains a canonical natural transformation of functors $\Hbb(S) \rightarrow \Hbb(S)$
		\begin{equation*}
			\psi: C(A) \rightarrow z_* z^* A.
		\end{equation*}
		To check that it is object-wise invertible, one can use the conservativity of the pair $(z^*,u^*)$: then it suffices to note that, by the fully faithfulness of $z_*$ in \Cref{defn_Hlocal}(c) and the triangular identity, the unit morphism $z^* A \xrightarrow{\eta} z^* z_* z^* A$ is an isomorphism inverse to $z^* z_* z^* A \xrightarrow{\epsilon} z^* A$.
		\item Since the co-unit arrow $z'_* p_Z^* z^* z_* A \xrightarrow{\epsilon} z'_* p_Z^* A$ is invertible by the fully faithfulness of $z_*$, it suffices to show that the unit arrow $p^* z_* A \xrightarrow{\eta} z'_* {z'}^* p^* z_* A$ is invertible as well. To this end, form the Cartesian diagram
		\begin{equation*}
			\begin{tikzcd}
				P_U \arrow{r}{u'} \arrow{d}{p_U} & P \arrow{d}{p} \\
				U \arrow{r}{u} & S
			\end{tikzcd}
		\end{equation*}
		complementary to the one in the statement. As a consequence of the previous point, the invertibility of the above unit arrow is equivalent to the vanishing of the composite functor $u'_{\#} {u'}^* \circ p^* \circ z_*$. To show this, it suffices to use the existence of a connection isomorphism ${u'}^* \circ p^* \simeq p_U^* \circ u^*$ and the vanishing of $u^* \circ z_*$ established in the second point.
	\end{enumerate}
\end{proof}

From now until the end of this section, we fix two localic $\Scal$-fibered categories $\Hbb_1$ and $\Hbb_2$, as well as a family a triangulated functors $\Rcal = \left\{R_S: \Hbb_1(S) \rightarrow \Hbb_2(S) \right\}_{S \in \Scal}$. 

Our goal is to modify the notion of $\Scal$-skeleton on $\Rcal$ by replacing inverse images under closed immersions with the corresponding direct images; to this end, we need to have a closer look at the adjointability properties of $\Scal^{cl}$-structures and $\Scal^{sm}$-structures according to \Cref{defn:mor-Sfib-adj}(2). Here is the key observation:

\begin{lem}\label{lem_Scl_Scl-op}
	Suppose that we are given an $\Scal^{sm}$-structure on $\Rcal$. Then the following statements hold:
	\begin{enumerate}
		\item Every possible $\Scal^{cl}$-structure on $\Rcal$ is right-adjointable.
		\item If the given $\Scal^{sm}$-structure is left-adjointable, then every possible $\Scal^{cl,op}$-structure on $\Rcal$ is left-adjointable.
	\end{enumerate}
\end{lem}
\begin{proof}
	It suffices to show that, given a closed immersion $z: Z \rightarrow S$ in $\Scal$, the following claims hold:
	\begin{enumerate}
		\item Given a natural isomorphism of functors $\Hbb_1(S) \rightarrow \Hbb_2(Z)$
		\begin{equation*}
			\theta_z: z^* R_S(A) \xrightarrow{\sim} R_Z (z^* A)
		\end{equation*}
		the natural transformation of functors $\Hbb_1(Z) \rightarrow \Hbb_2(S)$
		\begin{equation*}
			\bar{\theta}_z: R_S (z_* B) \xrightarrow{\eta} z_* z^* R_S (z_* B) \xrightarrow{\theta_z} z_* R_Z (z^* z_* B) \xrightarrow{\epsilon} z_* R_Z(B)
		\end{equation*}
		is always invertible.
		\item Given a natural isomorphism of functors $\Hbb_1(Z) \rightarrow \Hbb_2(S)$
		\begin{equation*}
			\bar{\theta}_z: R_S (z_* B) \xrightarrow{\sim} z_* R_Z(B)
		\end{equation*}
		the natural transformation of functors $\Hbb_1(S) \rightarrow \Hbb_2(Z)$
		\begin{equation*}
			\theta_z: z^* R_S(A) \xrightarrow{\eta} z^* R_S (z_* z^* A) \xrightarrow{\bar{\theta}_z} z^* z_* R_Z (z^* A) \xrightarrow{\epsilon} R_Z (z^* A)
		\end{equation*}
		is invertible, provided that the given $\Scal^{sm}$-structure is left-adjointable.
	\end{enumerate}
	Let $u: U \rightarrow S$ denote the open immersion complementary to $z$. We argue as follows:
	\begin{enumerate}
		\item We only have to show that the unit arrow $R_S \circ z_* \xrightarrow{\eta} z_* z^* \circ R_S \circ z_*$ is invertible. By the localization triangle and the fully faithfulness of $u_{\#}$ (\Cref{lem_Hlocal}(1)), this is equivalent to showing the vanishing $u^* \circ  R_S \circ z_* = 0$. This follows from the existence of a natural isomorphism $u^* \circ  R_S \simeq R_S \circ u^*$ (given by the hypothesis) together with the vanishing $u^* \circ z_* = 0$ (\Cref{lem_Hlocal}(2)).
		\item As above, we only have to show that the unit arrow $z^* \circ R_S \xrightarrow{\eta} z^* \circ R_S \circ z_* z^*$ is invertible. By the localization triangle, this is equivalent to showing the vanishing $z^* \circ R_S \circ  u_{\#} u^* = 0$. This follows from the natural isomorphism $u_{\#} \circ R_U \simeq R_S \circ u_{\#}$ (given by the hypothesis) together with the vanishing $z^* \circ u_{\#} = 0$ (\Cref{lem_Hlocal}(2)).
	\end{enumerate}
    This proves the claim and concludes the proof.
\end{proof}

This discussion leads naturally to the following definition:

\begin{defn}
	Let $\Hbb_1$ and $\Hbb_2$ be two localic $\Scal$-fibered categories. 
	\begin{enumerate}
		\item We say that a triangulated morphism of $\Scal$-fibered categories $R: \Hbb_1 \rightarrow \Hbb_2$ is \textit{localic} if it is smooth-adjointable in the sense of \Cref{defn:geom}(2).
		\item If we are only given a collection of triangulated functors $\Rcal = \left\{R_S: \Hbb_1(S) \rightarrow \Hbb_2(S) \right\}_{S \in \Scal}$, we say that an $\Scal$-structure $\theta$ on $\Rcal$ is \textit{localic} if so is the resulting morphism of $\Scal$-fibered categories $(\Rcal,\theta): \Hbb_1 \rightarrow \Hbb_2$.
	\end{enumerate} 
\end{defn}

\subsection{Cores of localic morphisms}

After the above preliminaries, we can define the natural variant of skeleta of morphisms as follows:

\begin{defn}\label{defn:core}
	An \textit{$\Scal$-core} on the family $\Rcal$ is the datum of
	\begin{itemize}
		\item a left-adjointable $\Scal^{sm}$-structure on $\Rcal$,
		\item an $\Scal^{cl,op}$-structure on $\Rcal$
	\end{itemize}
	satisfying the following exchange conditions:
	\begin{enumerate}
		\item[(C'-$\Scal$-fib)] For every Cartesian square in $\Scal$
		\begin{equation*}
			\begin{tikzcd}
				P_Z \arrow{r}{z'} \arrow{d}{p'} & P \arrow{d}{p} \\
				Z \arrow{r}{z} & S
			\end{tikzcd}
		\end{equation*}
		with $p$ a smooth morphism and $z$ a closed immersion, the diagram of functors $\Hbb_1(Z) \rightarrow \Hbb_2(P)$
		\begin{equation*}
			\begin{tikzcd}
				p^* z_* R_Z(A) \isoarrow{d} & p^* R_S (z_* A) \arrow{l}{\bar{\theta}^{cl}} \arrow{r}{\theta^{sm}} & R_{P_Z} (p^* z^* A) \isoarrow{d} \\
				{z'}_* {p'}^* R_S(A) \arrow{r}{\theta^{sm}} & {z'}^* R_{P_Z} ({p'}^* A)  & R_P ({z'}_* {p'}^* A) \arrow{l}{\bar{\theta}^{cl}}
			\end{tikzcd}
		\end{equation*}
		is commutative.
		\item[(T-$\Scal$-fib)] For every commutative triangle in $\Scal$
		\begin{equation*}
			\begin{tikzcd}
				Q \arrow{r}{h} \arrow{dr}{q} & P \arrow{d}{p} \\
				& S
			\end{tikzcd}
		\end{equation*}
		with $p,q$ smooth morphisms and $h$ a closed immersion, the diagram of functors $\Hbb_1(S) \rightarrow \Hbb_2(Q)$
		\begin{equation*}
			\begin{tikzcd}
				q^* R_S(A) \arrow{rr}{\theta^{sm}} \arrow[equal]{d} && R_Q (q^* A) \arrow[equal]{d} \\
				h^* p^* R_S(A) \arrow{r}{\theta^{sm}} & h^* R_P (p^* A) \arrow{r}{\theta^{cl}} & R_S (h^* p^* A)
			\end{tikzcd}
		\end{equation*}
		is commutative.
	\end{enumerate}
\end{defn}

\begin{thm}\label{prop_mor-core}
	Every localic $\Scal$-structure on $\Rcal$ determines by restriction and adjunction an $\Scal$-core on it. Conversely, every $\Scal$-core on $\Rcal$ uniquely extends to a localic $\Scal$-structure on it.
\end{thm}
\begin{proof}
	In view of \Cref{prop_mor-skel}, it suffices to show that localic $\Scal$-structures and $\Scal$-cores on $\Rcal$ determine one another by adjunction.
	To this end, in view of \Cref{lem_Scl-str} and \Cref{lem_Scl_Scl-op}, it suffices to show that the exchange conditions (C-$\Scal$-fib) and (C'-$\Scal$-fib) are equivalent.
	
	Assume that condition (C-$\Scal$-fib) holds, and let us prove that (C'-$\Scal$-fib) holds as well: given a Cartesian square
	\begin{equation*}
		\begin{tikzcd}
			P_Z \arrow{r}{z'} \arrow{d}{p'} & P \arrow{d}{p} \\
			Z \arrow{r}{z} & S
		\end{tikzcd}
	\end{equation*}
	with $z$ a closed immersion and $p$ a smooth morphism, we have to show that the diagram of functors $\Hbb_1(Z) \rightarrow \Hbb_2(P)$
	\begin{equation*}
		\begin{tikzcd}
			p^* z_* R_Z(A) \isoarrow{d} & p^* R_S (z_* A) \arrow{l}{\bar{\theta}^{cl}} \arrow{r}{\theta^{sm}} & R_P (p^* z_* A) \isoarrow{d} \\
			{z'}_* {p'}^* R_Z(A) \arrow{r}{\theta^{sm}} & {z'}_* R_{P_Z} ({p'}^* A) & R_P ({z'}_* {p'}^* A) \arrow{l}{\bar{\theta}^{cl}}
		\end{tikzcd}
	\end{equation*}
    is commutative. Unwinding the various definitions, we obtain the more explicit diagram
    \begin{equation*}
    	\begin{tikzcd}[font=\small]
    		p^* z_* R_Z(A) \arrow{d}{\eta} & p^* z_* R_Z (z^* z_* A) \arrow{l}{\epsilon} & p^* z_* z^* R_S (z_* A) \arrow{l}{\theta^{cl}} & p^* R_S (z_* A) \arrow{l}{\eta} \arrow{d}{\theta^{sm}}  \\
    		{z'}_* {z'}^* p^* z_* R_Z(A) \arrow[equal]{d} &&& R_P (p^* z_* A) \arrow{d}{\eta} \\
    		 {z'}_* {p'}^* z^* z_* R_Z(A) \arrow{d}{\epsilon} &&& R_P ({z'}_* {z'}^* p^* z_* A) \arrow[equal]{d} \\
    		{z'}_* {p'}^* R_Z(A) \arrow{d}{\theta^{sm}}  &&& R_P ({z'}_* {p'}^* z^* z_* A) \arrow{d}{\epsilon} \\
    		{z'}_* R_{P_Z} ({p'}^* A) & {z'}_* R_{P_Z} ({z'}^* {z'}_* {p'}^* A) \arrow{l}{\epsilon} & {z'}_* {z'}^* R_{P_Z} ({z'}_* {p'}^* A) \arrow{l}{\theta^{cl}} & R_{P_Z} ({z'}_* {p'}^* A) \arrow{l}{\eta}
    	\end{tikzcd}
    \end{equation*}
    that we decompose as
	\begin{equation*}
		\begin{tikzcd}[font=\small]
			\bullet \arrow{d}{\eta} & \bullet \arrow{l}{\epsilon} \arrow{d}{\eta} & \bullet \arrow{l}{\theta^{cl}} \arrow{d}{\eta} & \bullet \arrow{l}{\eta} \arrow{dr}{\theta^{sm}} \arrow{d}{\eta} \\
			\bullet \arrow[equal]{dd} & {z'}_* {z'}^* p^* z_* R_Z (z^* z_* A) \arrow{l}{\epsilon} \arrow[equal]{d} & {z'}_* {z'}^* p^* z_* z^* R_S (z_* A) \arrow{l}{\theta^{cl}} & {z'}_* {z'}^* p^* R_S (z_* A) \arrow{d}{\theta^{sm}} \arrow{l}{\eta} & \bullet \arrow{d}{\eta} \arrow{dl}{\eta} \\
			& {z'}_* {p'}^* z^* z_* R_Z (z^* z_* A) \arrow{d}{\epsilon} \arrow{dl}{\epsilon} && {z'}_* {z'}^* R_P (p^* z_* A) \arrow{d}{\eta} & \bullet \arrow[equal]{dd} \arrow{dl}{\eta} \\
			\bullet \arrow{d}{\epsilon} & {z'}_* {p'}^* R_Z (z^* z_* A) \arrow{d}{\theta^{sm}} \arrow{dl}{\epsilon} && {z'}_* {z'}^* R_P ({z'}_* {z'}^* p^* z_* A) \arrow[equal]{d} \\
			\bullet \arrow{dr}{\theta^{sm}}  & {z'}_* R_{P_Z} ({p'}^* z^* z_* A) \arrow{d}{\epsilon} & {z'}_* R_{P_Z} ({z'}^* {z'}_* {p'}^* z^* z_* A) \arrow{l}{\epsilon} \arrow{d}{\epsilon} & {z'}_* {z'}^* R_P ({z'}_* {p'}^* z^* z_* A) \arrow{d}{\epsilon} \arrow{l}{\theta^{cl}} & \bullet \arrow{d}{\epsilon} \arrow{l}{\eta} \\
			& \bullet & \bullet \arrow{l}{\epsilon} & \bullet \arrow{l}{\theta^{cl}} & \bullet \arrow{l}{\eta}
		\end{tikzcd}
	\end{equation*}
	Since all boundary pieces in the latter diagram are commutative by naturality, we are reduced to considering the inner rectangle. By construction, it then suffices to show that the diagram of functors $\Hbb_1(S) \rightarrow \Hbb_2(P_Z)$
	\begin{equation*}
		\begin{tikzcd}[font=\small]
			{z'}^* p^* z_* R_Z (z^* A)  \arrow[equal]{d} & {z'}^* p^* z_* z^* R_S(A) \arrow{l}{\theta^{cl}} & {z'}^* p^* R_S(A) \arrow{l}{\eta}  \arrow{d}{\theta^{sm}} \\
			{p'}^* z^* z_* R_Z (z^* A)  \arrow{d}{\epsilon} && {z'}^* R_P (p^* A) \arrow{d}{\eta} \\
			{p'}^* R_Z (z^* A) \arrow{d}{\theta^{sm}} && {z'}^* R_P ({z'}_* {z'}^* p^* A) \\
			R_{P_Z} ({p'}^* z^* A) & R_{P_Z} ({z'}^* {z'}_* {p'}^* z^* A) \arrow{l}{\epsilon} & {z'}^* R_P ({z'}_* {p'}^* z^* A) \arrow{l}{\theta^{cl}} \arrow[equal]{u}
		\end{tikzcd}
	\end{equation*}
    is commutative. To this end, it suffices to decompose it as
	\begin{equation*}
		\begin{tikzcd}[font=\small]
			\bullet  \arrow[equal]{d} & \bullet \arrow[equal]{d} \arrow{l}{\theta^{cl}} \\
			\bullet  \arrow{dd}{\epsilon} & {p'}^* z^* z_* z^* R_S(A) \arrow{l}{\theta^{cl}}  & \bullet \arrow{ul}{\eta} \arrow[equal]{dl} \arrow{d}{\theta^{sm}} \\
			& {p'}^* z^* R_S(A) \arrow{u}{\eta} \arrow{dl}{\theta^{cl}} & \bullet \arrow{dl}{\theta^{cl}} \arrow{dd}{\eta} \\
			\bullet \arrow{d}{\theta^{sm}} & R_{P_Z} ({z'}^* p^* A) \\
			\bullet \arrow[equal]{ur} & R_{P_Z} ({z'}^* {z'}_* {z'}_* p^* A) \arrow{u}{\epsilon} & \bullet \arrow{l}{\theta^{cl}} \\
			& \bullet \arrow{ul}{\epsilon} \arrow[equal]{u} & \bullet \arrow{l}{\theta^{cl}} \arrow[equal]{u}
		\end{tikzcd}
	\end{equation*}
	where the central piece is commutative by condition (C-$\Scal$-fib) while all the other inner pieces are commutative by naturality and by construction.
	
	Conversely, assume that condition (C'-$\Scal$-fib) holds, and let us prove that (C-$\Scal$-fib) holds as well: given a Cartesian square
	\begin{equation*}
		\begin{tikzcd}
			P_Z \arrow{r}{z'} \arrow{d}{p'} & P \arrow{d}{p} \\
			Z \arrow{r}{z} & S
		\end{tikzcd}
	\end{equation*}
	with $z$ a closed immersion and $p$ a smooth morphism, we have to show that the diagram of functors $\Hbb_1(Z) \rightarrow \Hbb_2(P)$
	\begin{equation*}
		\begin{tikzcd}
			{z'}^* p^* R_S(A) \arrow[equal]{d} \arrow{r}{\theta^{sm}} & {z'}^* R_P (p^* A) \arrow{r}{\theta^{cl}} & R_{P_Z} ({z'}^* p^* A) \arrow[equal]{d} \\
			{p'}^* z^* R_S(A) \arrow{r}{\theta^{cl}} & {p'}^* R_Z (z^* A) \arrow{r}{\theta^{sm}} & R_{P_Z} ({p'}^* z^* A)
		\end{tikzcd}
	\end{equation*}
    is commutative. Unwinding the various definitions, we obtain the more explicit diagram
    \begin{equation*}
    	\begin{tikzcd}[font=\small]
    		{z'}^* R_P (p^* A) \arrow{r}{\eta} \arrow{d}{\eta} & {z'}^* R_P ({z'}_* {z'}^* p^* A) \arrow{r}{\bar{\theta}^{cl}} & {z'}^* {z'}_* R_{P_Z} ({z'}^* p^* A) \arrow{r}{\epsilon} & R_{P_Z} ({z'}^* p^* A) \arrow{d}{\eta} \\
    		{z'}^* p^* R_S(A) \arrow{d}{\eta} &&& R_{P_Z} ({z'}^* p^* z_* z^* A) \isoarrow{d} \\
    		{z'}^* p^* z_* z^* R_S(A) \isoarrow{d} &&& R_{P_Z} ({z'}^* {z'}_* {p'}^* z^* A) \arrow{d}{\epsilon} \\
    		{z'}^* {z'}_* {p'}^* z^* R_S(A) \arrow{d}{\epsilon} &&& R_{P_Z} ({p'}^* z^* A)  \\
    		{p'}^* z^* R_S(A) \arrow{r}{\eta} & {p'}^* z^* R_S (z_* z^* A) \arrow{r}{\bar{\theta}^{cl}} & {p'}^* z^* z_* R_Z (z^* A) \arrow{r}{\epsilon} & {p'}^* R_Z (z^* A) \arrow{u}{\theta^{sm}}
    	\end{tikzcd}
    \end{equation*}
    that we can decompose as
    \begin{equation*}
    	\begin{tikzcd}[font=\small]
    		& \bullet \arrow{r}{\eta} \arrow{d}{\eta} & \bullet \arrow{r}{\bar{\theta}^{cl}} \arrow{d}{\eta} & \bullet \arrow{r}{\epsilon} \arrow{d}{\eta} & \bullet \arrow{d}{\eta} \\
    		\bullet \arrow{ur}{\theta^{sm}} \arrow{d}{\eta} & {z'}^* R_P(p^* z_* z^* A) \arrow{r}{\eta} & {z'}^* R_P ({z'}_* {z'}^* p^* z_* z^* A) \arrow{r}{\bar{\theta}^{cl}} & {z'}^* {z'}_* R_{P_Z} ({z'}^* p^* z_* z^* A) \arrow{r}{\epsilon} \isoarrow{d} & \bullet \isoarrow{dd} \\
    		\bullet \isoarrow{dd} \arrow{ur}{\theta^{sm}} \arrow{dr}{\eta} & {z'}^* p^* R_S (z_* z^* A) \arrow{u}{\theta^{sm}} \arrow{d}{\eta} && {z'}^* {z'}_* R_{P_Z} ({z'}^* {z'}_* {p'}^* z^* A) \arrow{d}{\epsilon} \arrow{dr}{\epsilon} \\
    		& {z'}^* p^* z_* z^* R_S (z_* z^* A) \isoarrow{d} && {z'}^* {z'}_* R_{P_Z} ({p'}^* z^* A) & \bullet \arrow{d}{\epsilon} \\
    		\bullet \arrow{d}{\epsilon} \arrow{r}{\eta} & {z'}^* {z'}_* {p'}^* z^* R_S (z_* z^* A) \arrow{d}{\epsilon} \arrow{r}{\bar{\theta}^{cl}} & {z'}^* {z'}_* {p'}^* z^* z_* R_Z (z^* A) \arrow{d}{\epsilon} \arrow{r}{\epsilon} & {z'}^* {z'}_* {p'}^* R_Z (z^* A) \arrow{d}{\epsilon} \arrow{u}{\theta^{sm}} \arrow{ur}{\theta^{sm}} & \bullet  \\
    		\bullet \arrow{r}{\eta} & \bullet \arrow{r}{\bar{\theta}^{cl}} & \bullet \arrow{r}{\epsilon} & \bullet \arrow{ur}{\theta^{sm}}
    	\end{tikzcd}
    \end{equation*}
    Since all pieces in the boundary of the latter diagram are commutative by naturality, we are reduced to considering the inner diagram. By construction, it then suffices to show that the diagram of functors $\Hbb_1(Z) \rightarrow \Hbb_2(P)$
	\begin{equation*}
		\begin{tikzcd}[font=\small]
			p^* R_S (z_* A) \arrow{r}{\theta^{sm}} \arrow{d}{\eta} & R_P (p^* z_* A) \arrow{r}{\eta} & R_P ({z'}_* {z'}^* p^* z_* A) \arrow{r}{\bar{\theta}^{cl}} & {z'}_* R_{P_Z} ({z'}^* p^* z_* A) \isoarrow{d} \\
			p^* z_* z^* R_S (z_* A) \isoarrow{d} &&& {z'}_* R_{P_Z} ({z'}^* {z'}_* {p'}^* A) \arrow{d}{\epsilon} \\
			{z'}_* {p'}^* z^* R_S (z_* A) \arrow{r}{\bar{\theta}^{cl}} & {z'}_* {p'}^* z^* z_* R_Z(A) \arrow{r}{\epsilon}  & {z'}_* {p'}^* R_Z(A) \arrow{r}{\theta^{sm}} & {z'}_* R_{P_Z} ({p'}^* A)
		\end{tikzcd}
	\end{equation*}
    is commutative. To this end, it suffices to decompose it as
    \begin{equation*}
    	\begin{tikzcd}[font=\small]
    		& \bullet \arrow{r}{\theta^{sm}} \arrow[bend right]{dl}{\eta} \arrow{dr}{\bar{\theta}^{cl}} & \bullet \arrow{rr}{\eta} \arrow{dr}{\sim} && \bullet \arrow{r}{\bar{\theta}^{cl}} \isoarrow{d} & \bullet \isoarrow{d} \\
    		\bullet \arrow{r}{\bar{\theta}^{cl}} \isoarrow{d} & p^* z_* z^* z_* R_Z(A) \arrow{r}{\epsilon} \isoarrow{d}  & p^* z_* R_Z(A) \arrow{dr}{\sim} & R_P ({z'}_* {p'}_* A) \arrow{r}{\eta} \arrow{dr}{\bar{\theta}^{cl}} & R_P({z'}_* {z'}^* {z'}_* {p'}^* A) \arrow{r}{\bar{\theta}^{cl}} & \bullet \arrow[bend left]{dl}{\epsilon} \\
    		\bullet \arrow{r}{\bar{\theta}^{cl}} & \bullet \arrow{rr}{\epsilon}  && \bullet \arrow{r}{\theta^{sm}} & \bullet
    	\end{tikzcd}
    \end{equation*}
    Here, the central piece is commutative by condition (C'-$\Scal$-fib) while all the other inner pieces are commutative by naturality and by construction. This concludes the proof.
\end{proof}

\section{External tensor skeleta}\label{sect_ETSkel}

In this section we describe a variant of the notion of $\Scal$-skeleton from \Cref{sect_mor-skel} in the setting of external tensor structures (both on single $\Scal$-fibered categories and on morphisms of such). See the last part of \Cref{sect:adj-fib-cat} for the basic notions about external tensor structures.

To this end, we have to keep \Cref{hyp_skel} (but not \Cref{hyp_core}) and we also have to assume that the base category $\Scal$ admits binary products. 

\subsection{External tensor skeleta on fibered categories}

In the first part of this section, we work with a fixed $\Scal$-fibered category $\Hbb$. 

\begin{defn}\label{defn:ETSkel}
	An \textit{external tensor skeleton} $(\boxtimes,m^{sm},m^{cl})$ on $\Hbb$ is the datum of
	\begin{itemize}
		\item an external tensor structure $(\boxtimes^{sm},m^{sm})$ on $\Hbb^{sm}$,
		\item an external tensor structure $(\boxtimes^{cl},m^{cl})$ on $\Hbb^{cl}$
	\end{itemize}
	satisfying the following two conditions:
	\begin{enumerate}
		\item[(ex-ETS-0)] For every $S_1, S_2 \in \Scal$, the two functors
		\begin{equation*}
			- \boxtimes^{sm} - = - \boxtimes^{sm}_{S_1,S_2} -, \; - \boxtimes^{cl} - = - \boxtimes^{cl}_{S_1,S_2} -: \Hbb(S_1) \times \Hbb(S_2) \rightarrow \Hbb(S_1 \times S_2)
		\end{equation*}
		coincide, so that we can write both of them with the same symbol $-\boxtimes- = - \boxtimes_{S_1,S_2} -$.
		\item[(ex-ETS-1)] For every pair of commutative squares in $\Scal$
		\begin{equation*}
			\begin{tikzcd}
				Q_i \arrow{r}{h_i} \arrow{d}{q_i} & P_i \arrow{d}{p_i} \\
				Z_i \arrow{r}{z_i} & S_i
			\end{tikzcd}
			\qquad (i = 1,2)
		\end{equation*}
		with $p_i,q_i$ smooth morphisms and $z_i,h_i$ closed immersions, the diagram of functors $\Hbb(S_1) \times \Hbb(S_2) \rightarrow \Hbb(Q_1 \times Q_2)$
		\begin{equation*}
			\begin{tikzcd}
				q_1^* z_1^* A_1 \boxtimes q_2^* z_2^* A_2 \arrow{r}{m^{sm}} \arrow[equal]{d} & (q_1 \times q_2)^*(z_1^* A_1 \boxtimes z_2^* A_2) \arrow{r}{m^{cl}} & (q_1 \times q_2)^* (z_1 \times z_2)^*(A_1 \boxtimes A_2) \arrow[equal]{d} \\
				h_1^* p_1^* A_1 \boxtimes h_2^* p_2^* A_2 \arrow{r}{m^{cl}} & (h_1 \times h_2)^*(p_1^* A_1 \boxtimes p_2^* A_2) \arrow{r}{m^{sm}} & (h_1 \times h_2)^* (p_1 \times p_2)^*(A_1 \boxtimes A_2)
			\end{tikzcd}
		\end{equation*}
		is commutative.
	\end{enumerate}
	In case $\Hbb$ is smooth-adjointable (in the sense of \Cref{defn:geom}(1)), we say that a given external tensor skeleton on $\Hbb$ is \textit{smooth-adjointable} if the external tensor structure $(\boxtimes^{sm},m^{sm})$ is left-adjointable in the sense of \Cref{defn:Sfib-adj}(3).
\end{defn}

\begin{rem}\label{rem_skel}
	For $i = 1,2$ let $f_i: T_i \rightarrow S_i$ be a morphism in $\Scal$ which is both a closed immersion and a smooth morphism. Recalling the convention $\id_{S_i}^* = \id_{\Hbb(S_i)}$, and taking into account \cite[Rem.~2.7]{Ter24Fib}, we see that condition (ex-ETS) applied to the commutative squares
	\begin{equation*}
		\begin{tikzcd}
			T_i \arrow{r}{f_i} \arrow{d}{f_i} & S_i \arrow{d}{\id_{S_i}} \\
			S_i \arrow{r}{\id_{S_i}} & S_i
		\end{tikzcd}
		\qquad (i = 1,2)
	\end{equation*}
	asserts that the two natural isomorphisms of functors $\Hbb(S_1) \times \Hbb(S_2) \rightarrow \Hbb(T_1 \times T_2)$
	\begin{equation*}
		m^{sm} = m^{sm}_{f_1,f_2}, \; m^{cl} = m^{cl}_{f_1,f_2}: f_1^* A_1 \boxtimes f_2^* A_2 \xrightarrow{\sim} (f_1 \times f_2)^*(A_1 \boxtimes A_2)
	\end{equation*}
	obtained by regarding $f_1, f_2$ as smooth morphisms or as closed immersions, respectively, must coincide.
\end{rem}

Here is the first main result of this section:

\begin{prop}\label{thm_ETSskel}
	Every external tensor structure on $\Hbb$ determines by restriction an external tensor skeleton on it. Conversely, every external tensor skeleton on $\Hbb$ uniquely extends to an external tensor structure on it.
	
	Moreover, in case $\Hbb$ is smooth-adjointable, then a given external tensor structure on $\Hbb$ is smooth-adjointable if and only if the underlying external tensor skeleton is smooth-adjointable.
\end{prop}
\begin{proof}
	It suffices to adapt the proof of \Cref{prop_mor-skel} in the obvious way, with condition (ex-ETS-1) playing here the same role as condition (ex-$\Scal$-fib) there. 
	In particular, for every choice of arrows $f_i: T_i \rightarrow S_i$ in $\Scal$, $i = 1,2$, the external monoidality isomorphism of functors $\Hbb(S_1) \times \Hbb(S_2) \rightarrow \Hbb(T_1 \times T_2)$
	\begin{equation*}
		m = m_{f_1,f_2}: f_1^* A_1 \boxtimes f_2^* A_2 \xrightarrow{\sim} (f_1 \times f_2)^* (A_1 \boxtimes A_2)
	\end{equation*} 
    is defined by taking the composite
	\begin{equation*}
		m_{f_1,f_2}^{(t_1,p_1;t_2,p_2)}: f_1^* A_1 \boxtimes f_2^* A_2 = t_1^* p_1^* A_1 \boxtimes t_2^* p_2^* A_2 \xrightarrow{m^{cl}} t_{12}^* (p_1^* A_1 \boxtimes p_2^* A_2) \xrightarrow{m^{sm}} t_{12}^* p_{12}^* (A_1 \boxtimes A_2) = f_{12}^* (A_1 \boxtimes A_2)
	\end{equation*}
	associated to any choice of factorizations $(P_i;t_i,p_i) \in \Fact_{\Scal}(f_i)$, $i = 1,2$. As in the proof of \Cref{prop_mor-skel}, one can check that the resulting definition of $m_{f_1,f_2}$ does not depend on the chosen factorizations and that the external monoidality isomorphisms satisfy condition ($m$ETS), thus defining an external tensor structure on $\Hbb$. We leave the details to the interested reader.
\end{proof}

Before proceeding further with the theory of external tensor skeleta, let us quickly note the following analogue of \Cref{lem_Scl-str}, which will be useful in the next section:
\begin{lem}\label{lem:ETSskel-obs}
	Let $(\boxtimes,m^{sm})$ and $(\boxtimes,m^{cl})$ be external tensor structure on the underlying $\Scal^{sm}$-fibered and $\Scal^{cl}$-fibered category of $\Hbb$ satisfying condition (ex-ETS-0). Then the triple $(\boxtimes,m^{sm},m^{cl})$ defines an external tensor skeleton on $\Hbb$ if and only if it satisfies the following two conditions:
	\begin{enumerate}
		\item[(C-ETS-1)] For every pair of Cartesian squares in $\Scal$
		\begin{equation*}
			\begin{tikzcd}
				P_{i,Z_i} \arrow{r}{z'_i} \arrow{d}{p'_i} & P_i \arrow{d}{p_i} \\
				Z_i \arrow{r}{z_i} & S_i
			\end{tikzcd}
		    \qquad (i = 1,2)
		\end{equation*}
		with $p_1$, $p_2$ smooth morphisms and $z_1$, $z_2$ closed immersions, the diagram of functors $\Hbb(S_1) \times \Hbb(S_2) \rightarrow \Hbb(P_{1, Z_1} \times P_{2, Z_2})$
		\begin{equation*}
			\begin{tikzcd}
				{p'_1}^* z_1^* A_1 \boxtimes {p'_2}^* z_2^* A_2 \arrow{r}{m^{sm}} \arrow[equal]{d} & (p'_1 \times p'_2)^* (z_1^* A_1 \boxtimes z_2^* A_2) \arrow{r}{m^{cl}} & (p'_1 \times p'_2)^* (z_1 \times z_2)^* (A_1 \boxtimes A_2) \arrow[equal]{d} \\
				{z'_1}^* p_1^* A_1 \boxtimes {z'_2}^* p_2^* A_2 \arrow{r}{m^{cl}} & (z'_1 \times z'_2)^* (p_1^* A_1 \boxtimes p_2^* A_2) \arrow{r}{m^{sm}} & (z'_1 \times z'_2)^* (p_1 \times p_2)^* (A_1 \boxtimes A_2)
			\end{tikzcd}
		\end{equation*}
		is commutative.
		\item[(T-ETS-1)] For every pair of commutative triangles in $\Scal$
		\begin{equation*}
			\begin{tikzcd}
				Q_i \arrow{r}{h_i} \arrow{dr}{q_i} & P_i \arrow{d}{p_i} \\
				& S_i
			\end{tikzcd}
		    \qquad (i = 1,2)
		\end{equation*}
		with $p_1$, $p_2$, $q_1$, $q_2$ smooth morphisms and $h_1$, $h_2$ closed immersions, the diagram of functors $\Hbb(S_1 \times S_2) \rightarrow \Hbb(Q_1 \times Q_2)$
		\begin{equation*}
			\begin{tikzcd}
				q_1^* A_1 \boxtimes q_2^* A_2 \arrow{rr}{m^{sm}} \arrow[equal]{d} && (q_1 \times q_2)^* (A_1 \times A_2) \arrow[equal]{d} \\
				h_1^* p_1^* A_1 \boxtimes h_2^* p_2^* A_2 \arrow{r}{m^{cl}} & (h_1 \times h_2)^* (p_1^* A_1 \boxtimes p_2^* A_2) \arrow{r}{m^{sm}} & (h_1 \times h_2)^* (p_1 \times p_2)^* (A_1 \times A_2)
			\end{tikzcd}
		\end{equation*}
		is commutative.
	\end{enumerate}
\end{lem}
\begin{proof}
	This can be proves along the lines of \Cref{lem_Scl-str}. We leave the details to the interested reader.
\end{proof}

\subsection{Compatibility with associativity and commutativity constraints}

In practice, one often wants to consider monoidal structures carrying (at least some among) compatible associativity, commutativity, and unit constraints. As shown in \cite[\S\S~3--6]{Ter24Fib}, all these notions can be formulated in the setting of external tensor structures. We now want to show that the various axioms that these constraints introduce in the picture can even be checked on the level of external tensor skeleta. 
This is based on the following notions:

\begin{defn}
	Let $(\boxtimes,m^{sm},m^{cl})$ be an external tensor skeleton on $\Hbb$.
	\begin{enumerate}
		\item An \textit{external associativity constraint} $a$ on $(\boxtimes,m)$ is the datum of
		\begin{itemize}
			\item an associativity constraint $a^{sm}$ on $(\boxtimes^{sm},m^{sm})$,
			\item an associativity constraint $a^{cl}$ on $(\boxtimes^{cl},m^{cl})$
		\end{itemize}
		such that, for every choice of $S_1, S_2, S_3 \in \Scal$, the two natural isomorphisms of functors $\Hbb(S_1) \times \Hbb(S_2) \times \Hbb(S_3) \rightarrow \Hbb(S_1 \times S_2 \times S_3)$
		\begin{equation*}
			a^{sm}, \; a^{cl}: (A_1 \boxtimes A_2) \boxtimes A_3 \xrightarrow{\sim} A_1 \boxtimes (A_2 \boxtimes A_3)
		\end{equation*}
		coincide, so that we can write both of them with the same symbol $a$.
		\item An \textit{external commutativity constraint} $c$ on $(\boxtimes,m^{sm},m^{cl})$ is the datum of
		\begin{itemize}
			\item an external commutativity constraint $c^{sm}$ on $(\boxtimes^{sm},m^{sm})$,
			\item an external commutativity constraint $c^{cl}$ on $(\boxtimes^{cl},m^{cl})$
		\end{itemize}
		such that, for every choice of $S_1, S_2 \in \Scal$, the two natural isomorphisms of functors $\Hbb(S_1) \times \Hbb(S_2) \rightarrow \Hbb(S_1 \times S_2)$
		\begin{equation*}
			c^{sm}, \; c^{cl}: A_1 \boxtimes A_2 \xrightarrow{\sim} \tau^*(A_2 \boxtimes A_1)
		\end{equation*}
		coincide, so that we can write both of them with the same symbol $c$.
	\end{enumerate}
\end{defn}

We note the following result:

\begin{lem}\label{lem_constr-skel}
	Let $(\boxtimes,m^{sm},m^{cl})$ be an external tensor skeleton on $\Hbb$. Then:
	\begin{enumerate}
		\item Every associativity constraint $a$ on $(\boxtimes,m^{sm},m^{cl})$ defines an associativity constraint on the corresponding external tensor structure $(\boxtimes,m)$.
		\item Every commutativity constraint $c$ on $(\boxtimes,m^{sm},m^{cl})$ defines a commutativity constraint on the corresponding external tensor structure $(\boxtimes,m)$.
	\end{enumerate}
\end{lem}
\begin{proof}
	In the proof of each statements, we consider a finite number of morphisms $f_i: T_i \rightarrow S_i$ in $\Scal$ and, for each of them, we choose a factorization $(P_i;t_i,p_i) \in \Fact_{\Scal}(f_i)$.
	\begin{enumerate}
		\item Since condition ($a$ETS-1) is trivially satisfied, we only need to check that condition ($a$ETS-2) is satisfied as well: given three morphisms $f_i: T_i \rightarrow S_i$ in $\Scal$, $i = 1,2,3$, we have to show that the diagram of functors $\Hbb(S_1) \times \Hbb(S_2) \times \Hbb(S_3) \rightarrow \Hbb(T_1 \times T_2 \times T_3)$
		\begin{equation*}
			\begin{tikzcd}
				(f_1^* A_1 \boxtimes f_2^* A_2) \boxtimes f_3^* A_3 \arrow{r}{m} \arrow{d}{a} & (f_1 \times f_2)^*(A_1 \boxtimes A_2) \boxtimes f_3^* A_3 \arrow{r}{m} & (f_1 \times f_2 \times f_3)^*((A_1 \boxtimes A_2) \boxtimes A_3)   \arrow{d}{a} \\
				f_1^* A_1 \boxtimes (f_2^* A_2 \boxtimes f_3^* A_3) \arrow{r}{m} & f_1^* A_1 \boxtimes (f_2 \times f_3)^*(A_2 \boxtimes A_3) \arrow{r}{m} & (f_1 \times f_2 \times f_3)^*(A_1 \boxtimes (A_2 \boxtimes A_3))
			\end{tikzcd}
		\end{equation*}
		is commutative. To this end, for $i = 1,2,3$ choose a factorization $(P_i;t_i,p_i) \in \Fact_{\Scal}(f_i)$. Then we can decompose the diagram in question as
		\begin{equation*}
			\begin{tikzcd}[font=\small]
				\bullet \arrow[equal,bend right]{ddr} \arrow{ddddd}{a} \arrow{rr}{m} && \bullet \arrow[equal]{d} \arrow{rr}{m} && \bullet \arrow[equal,bend left]{ddl} \arrow{ddddd}{a} \\
				& t_{12}^* (p_1^* A_1 \boxtimes p_2^* A_2) \boxtimes t_3^* p_3^* A_3 \arrow{r}{m^{sm}} \arrow{dr}{m^{cl}} &  (t_{12}^* p_{12}^* (A_1 \boxtimes A_2) \boxtimes t_3^* p_3^* A_3) \arrow{r}{m^{cl}} & t_{123}^* (p_{12}^* (A_1 \boxtimes A_2) \boxtimes p_3^* A_3) \arrow{d}{m^{sm}} \\
				& (t_1^* p_1^* A_1 \boxtimes t_2^* p_2^* A_2) \boxtimes t_3^* p_3^* A_3 \arrow{u}{m^{cl}} \arrow{d}{a} & t_{123}^*((p_1^* A_1 \boxtimes p_2^* A_2) \boxtimes p_3^* A_3) \arrow{ur}{m^{sm}} \arrow{d}{a} & t_{123}^* p_{123}^* ((A_1 \boxtimes A_2) \boxtimes A_3) \arrow{d}{a} \\
				& t_1^* p_1^* A_1 \boxtimes (t_2^* p_2^* A_2 \boxtimes t_3^* p_3^* A_3) \arrow{d}{m^{cl}} & t_{123}^* (p_1^* A_1 \boxtimes (p_2^* A_2 \boxtimes p_3^* A_3)) \arrow{dr}{m^{sm}} & t_{123}^* p_{123}^* (A_1 \boxtimes (A_2 \boxtimes A_3)) \\
				& t_1^* p_1^* A_1 \boxtimes t_{23}^* (p_2^* A_2 \boxtimes p_3^* A_3) \arrow{r}{m^{sm}} \arrow{ur}{m^{cl}} & t_1^* p_1^* A_1 \boxtimes t_{23}^* p_{23}^* (A_2 \boxtimes A_3) \arrow{r}{m^{cl}} & t_{123}^* (p_1^* A_1 \boxtimes p_{23}^* (A_2 \boxtimes A_3)) \arrow{u}{m^{sm}} \\
				\bullet \arrow[equal,bend left]{uur} \arrow{rr}{m} && \bullet \arrow[equal]{u} \arrow{rr}{m} && \bullet \arrow[equal,bend right]{uul}
			\end{tikzcd}
		\end{equation*}
	    Here, the left-most and right-most central pieces are commutative by axioms ($a$ETS-2) over $\Scal^{cl}$ and $\Scal^{sm}$, respectively, while the remaining pieces are commutative by naturality and by construction.
		\item Since condition ($c$ETS-1) is trivially satisfied, we only need to check that condition ($c$ETS-2) is satisfied as well: given morphisms $f_i: T_i \rightarrow S_i$ in $\Scal$, $i = 1,2$, we have to show that the diagram
		\begin{equation*}
			\begin{tikzcd}
				f_1^* A_1 \boxtimes f_2^* A_2 \arrow{rr}{c} \arrow{d}{m} && \tau^* (f_2^* A_2 \boxtimes f_1^* A_1) \arrow{d}{m} \\
				(f_1 \times f_2)^* (A_1 \boxtimes A_2)\arrow{r}{c} & (f_1 \times f_2)^* \tau^*(A_2 \boxtimes A_1) \arrow[equal]{r} & \tau^* (f_2 \times f_1)^* (A_2 \boxtimes A_1)
			\end{tikzcd}
		\end{equation*}
		is commutative. To this end, for $i = 1,2$ choose a factorization $(P_i;t_i,p_i) \in \Fact_{\Scal}(f_i)$. Then we can decompose the above diagram as
		\begin{equation*}
			\begin{tikzcd}[font=\small]
				\bullet \arrow{rrrrr}{c} \arrow{dddd}{m} \arrow[equal]{dr} &&&&& \bullet \arrow{dddd}{m} \arrow[equal]{dl} \\
				& t_1^* p_1^* A_1 \boxtimes t_2^* p_2^* A_2 \arrow{rrr}{c} \arrow{d}{m^{cl}} &&& \tau^* (t_2^* p_2^* A_2 \boxtimes t_1^* p_1^* A_1) \arrow{d}{m^{cl}} \\
				& t_{12}^* (p_1^* A_1 \boxtimes p_2^* A_2) \arrow{rr}{c} \arrow{d}{m^{sm}} && t_{12}^* \tau^* (p_2^* A_2 \boxtimes p_2^* A_2) \arrow[equal]{r} \arrow{d}{m^{sm}} & \tau^* t_{21}^* (p_2^* A_2 \boxtimes p_1^* A_1) \arrow{d}{m^{sm}} \\
				& t_{12}^* p_{12}^* (A_1 \boxtimes A_2) \arrow{r}{c} & t_{12}^* p_{12}^* \tau^* (A_2 \boxtimes A_1) \arrow[equal]{r} & t_{12}^* \tau^* p_{21}^* (A_2 \boxtimes A_1) \arrow[equal]{r} & \tau^* t_{21}^* p_{21}^* (A_2 \boxtimes A_1) \\
				\bullet \arrow{rrr}{c} \arrow[equal]{ur} &&& \bullet \arrow[equal]{u} \arrow[equal]{rr} && \bullet \arrow[equal]{ul}
			\end{tikzcd}
		\end{equation*}
	    Here, the upper and lower-left central pieces are commutative by axiom ($c$ETS-2) over $\Scal^{cl}$ and $\Scal^{sm}$, respectively, while the remaining pieces are commutative by naturality and by construction.
	\end{enumerate}
\end{proof}

\begin{rem}
	The natural compatibility conditions between external associativity, commutativity and unit constraints listed in \cite[\S 6]{Ter24Fib} do not involve inverse image functors except those along permutation isomorphisms. Since the subcategories $\Scal^{sm}$ and $\Scal^{cl}$ contain all isomorphisms of $\Scal$ by \Cref{hyp_skel}(i), it is clear that the compatibility conditions make sense also in the setting of external tensor skeleta.
\end{rem}

\subsection{External tensor skeleta on morphisms}

In the final part of this section, we extend the notion of external tensor skeleta to external tensor structure on morphisms. 

For $j = 1,2$ let $\Hbb_i$ be an $\Scal$-fibered category and let $(\boxtimes_j,m_j)$ be an external tensor structure on it; moreover, let $R = (\Rcal,\theta): \Hbb_1 \rightarrow \Hbb_2$ be a fixed morphism of $\Scal$-fibered categories.
In virtue of \Cref{prop_mor-skel}, the $\Scal$-structure $\theta$ on the family $\Rcal = \left\{R_S: \Hbb_1(S) \rightarrow \Hbb_2(S) \right\}_{S \in \Scal}$ is completely determined by the underlying $\Scal$-skeleton $(\theta^{sm},\theta^{cl})$.

We are going to describe an analogous results for external tensor structures on $R$ (with respect to $(\boxtimes_1,m_1)$ and $(\boxtimes_2,m_2)$).
This is based on the following definition:

\begin{defn}\label{defn:ETSskel-mor}
	An \textit{external tensor skeleton} $(\rho^{sm},\rho^{cl})$ on $R$ (with respect to $(\boxtimes_1,m_1^{sm},m_1^{cl})$ and $(\boxtimes_2,m_2^{sm},m_2^{cl})$) is the datum of
	\begin{itemize}
		\item an external tensor structure $\rho^{sm}$ on $R^{sm}: \Hbb_1^{sm} \rightarrow \Hbb_2^{sm}$ (with respect to $(\boxtimes_1,m_1^{sm})$ and $(\boxtimes_2,m_2^{sm})$),
		\item an external tensor structure $\rho^{cl}$ on $R^{cl}: \Hbb_1^{cl} \rightarrow \Hbb_2^{cl}$ (with respect to $(\boxtimes_1,m_1^{cl})$ and $(\boxtimes_2,m_2^{cl})$)
	\end{itemize}
	such that, for every $S_1, S_2 \in \Scal$, the two natural isomorphisms of functors $\Hbb_1(S_1) \times \Hbb_1(S_2) \rightarrow \Hbb_2(S_1 \times S_2)$
	\begin{equation*}
		\rho^{sm} = \rho^{sm}_{S_1,S_2}, \; \rho^{cl} = \rho^{cl}_{S_1,S_2}: R_{S_1}(A_1) \boxtimes_2 R_{S_2}(A_2) \xrightarrow{\sim} R_{S_2 \times S_2}(A_1 \boxtimes_1 A_2)
	\end{equation*}
	coincide, so that we can write both of them with the same symbol $\rho = \rho_{S_1,S_2}$.
\end{defn}

\begin{prop}\label{prop_mor_ETSkel}
	An external tensor skeleton on $R$ (with respect to $(\boxtimes_1,m_1^{sm},m_1^{cl})$ and $(\boxtimes_2,m_2^{sm},m_2^{cl})$) is the same as an external tensor structure on $R$ (with respect to $(\boxtimes_1,m_1)$ and $(\boxtimes_2,m_2)$).
\end{prop}
\begin{proof}
	Clearly, any external tensor structure $\rho$ on $R$ determines by restriction an external tensor skeleton on it.
	Conversely, suppose that we are given an external tensor skeleton $(\rho^{sm},\rho^{cl})$ on $R$. In order to prove that it defines an external tensor structure, we need to check that it satisfies condition (mor-ETS): given morphisms $f_i: T_i \rightarrow S_i$ in $\Scal$, $i = 1,2$, we have to show that the diagram of functors $\Hbb(S_1) \times \Hbb(S_2) \rightarrow \Hbb(T_1 \times T_2)$
	\begin{equation*}
		\begin{tikzcd}
			f_1^* R_{S_1}(A_1) \boxtimes f_2^* R_{S_2}(A_2) \arrow{r}{m_2} \arrow{d}{\theta} & (f_1 \times f_2)^* (R_{S_1}(A_1) \boxtimes R_{S_2}(A_2)) \arrow{r}{\rho} & (f_1 \times f_2)^* R_{S_1 \times S_2}(A_1 \boxtimes A_2) \arrow{d}{\theta} \\
			R_{T_1}(f_1^* A_1) \boxtimes R_{T_2}(f_2^* A_2) \arrow{r}{\rho} & R_{T_1 \times T_2}(f_1^* A_1 \boxtimes f_2^* A_2) \arrow{r}{m_1} & R_{T_1 \times T_2}((f_1 \times f_2)^*(A_1 \boxtimes A_2))
		\end{tikzcd}
	\end{equation*}
	is commutative. To this end, for $i = 1,2$ choose a factorization $(P_i;t_i,p_i) \in \Fact_{\Scal}(f_i)$. Then we can decompose the diagram is question as
	\begin{equation*}
		\begin{tikzcd}[font=\small]
			\bullet \arrow{ddddd}{\theta} \arrow[equal,bend right]{ddr} \arrow{rr}{m_2} && \bullet \arrow[equal]{d} \arrow{rr}{\rho} && \bullet \arrow{ddddd}{\theta} \arrow[equal]{dl} \\
			& t_{12}^* (p_1^* R_{S_1}(A_1) \boxtimes_2 p_2^* R_{S_2}(A_2)) \arrow{r}{m_2^{sm}} \arrow{dr}{\theta^{sm}} & t_{12}^* p_{12}^* (R_{S_1}(A_1) \boxtimes_2 R_{S_2}(A_2)) \arrow{r}{\rho} & t_{12}^* p_{12}^* R_{S_1 \times S_2}(A_1 \boxtimes_1 A_2) \arrow{d}{\theta^{sm}} \\
			& t_1^* p_1^* R_{S_1}(A_1) \boxtimes_2 t_2^* p_2^* R_{S_2}(A_2) \arrow{u}{m_2^{cl}} \arrow{d}{\theta^{sm}} & t_{12}^* (R_{P_1}(p_1^* A_1) \boxtimes_2 R_{P_2}(p_2^* A_2)) \arrow{d}{\rho} & t_{12}^* R_{P_1 \times P_2}(p_{12}^* (A_1 \boxtimes_1 A_2)) \arrow{d}{\theta^{cl}} \\
			& t_1^* R_{P_1}(p_1^* A_1) \boxtimes_2 t_2^* R_{P_2}(p_2^* A_2) \arrow{d}{\theta^{cl}} \arrow{ur}{m_2^{cl}} & t_{12}^* R_{P_1 \times P_2}(p_1^* A_1 \boxtimes_1 p_2^* A_2) \arrow{ur}{m_1^{sm}} \arrow{dr}{\theta^{cl}} & R_{T_1 \times T_2}(t_{12}^* p_{12}^* (A_1 \boxtimes_1 A_2)) \\
			& R_{T_1}(t_1^* p_1^* A_1) \boxtimes_2 R_{T_2}(t_2^* p_2^* A_2) \arrow{r}{\rho} & R_{T_1 \times T_2}(t_1^* p_1^* A_1 \boxtimes_1 t_2^* p_2^* A_2) \arrow{r}{m_1^{cl}} & R_{T_1 \times T_2}(t_{12}^* (p_1^* A_1 \boxtimes_1 p_2^* A_2)) \arrow{u}{m_1^{sm}} \\
			\bullet \arrow[equal]{ur} \arrow{rr}{\rho} && \bullet \arrow[equal]{u} \arrow{rr}{m_1} && \bullet \arrow[equal,bend right]{uul}
		\end{tikzcd}
	\end{equation*}
	Here, the upper-right and lower-left central pieces are commutative by axiom (mor-ETS) for $\Scal^{sm}$ and for $\Scal^{cl}$, respectively, while the remaining pieces are commutative by naturality and by construction. 
\end{proof}

\begin{rem}\label{rem:ETSkel-ass-comm}
	The associativity and commutativity conditions for external tensor structures on morphisms listed in \cite[\S~8]{Ter24Fib} do not involve inverse image functors except those along permutation isomorphisms. Since the subcategories $\Scal^{sm}$ and $\Scal^{cl}$ contain all isomorphisms of $\Scal$ bu \Cref{hyp_skel}(i), it is clear that these conditions make sense also in the setting of external tensor skeleta.
\end{rem}

\section{External tensor cores}\label{sect_ETScores}

In this final section, we combine the approaches of \Cref{sect_mor-core} and \Cref{sect_ETSkel} in order to obtain an alternative description of external tensor structures which is more adapted to the setting of perverse sheaves and perverse Nori motives. 

As in the previous section, we assume that the base category $\Scal$ admits binary products. Moreover, throughout the present section, we work under \Cref{hyp_skel} and \Cref{hyp_core}.

\subsection{External tensor cores on fibered categories}

In the first part of this section, we fix one triangulated $\Scal$-fibered category $\Hbb$ that we assume to be localic in the sense of \Cref{defn_Hlocal}. 

\begin{defn}
	Let $\Hbb$ be a localic $\Scal$-fibered category. We say that a triangulated external tensor structure $(\boxtimes,m)$ on $\Hbb$ is \textit{localic} if it is smooth-adjointable in the sense of \Cref{defn:geom}(3).
\end{defn}

The natural analogue of \Cref{defn:ETSkel} is the following:

\begin{defn}\label{defn:ETCore}
	An \textit{external tensor core} $(\boxtimes,m^{sm},\bar{m}^{cl})$ on $\Hbb$ is the datum of
	\begin{itemize}
		\item a triangulated external tensor structure $(\boxtimes^{sm},m^{sm})$ on the $\Scal^{sm}$-fibered category $\Hbb^{sm}$,
		\item a triangulated external tensor structure $(\boxtimes^{cl},\bar{m}^{cl})$ on the $\Scal^{cl,op}$-fibered category $\Hbb^{cl}$
	\end{itemize}
	satisfying the following conditions:
	\begin{enumerate}
		\item[(ETC-0)] For every $S_1, S_2 \in \Scal$, the two functors
		\begin{equation*}
			- \boxtimes^{sm} - = - \boxtimes^{sm}_{S_1,S_2} -, \; - \boxtimes^{cl} - = - \boxtimes^{cl}_{S_1,S_2} -: \Hbb(S_1) \times \Hbb(S_2) \rightarrow \Hbb(S_1 \times S_2)
		\end{equation*}
		coincide, so that we can indicate both of them with the same symbol $- \boxtimes - = - \boxtimes_{S_1,S_2} -$.
		\item[(C'-ETC-1)] For every pair Cartesian squares in $\Scal$
		\begin{equation*}
			\begin{tikzcd}
				P_{i,Z_i} \arrow{r}{z'_i} \arrow{d}{p'_i} & P_i \arrow{d}{p_i} \\
				Z_i \arrow{r}{z_i} & S_i
			\end{tikzcd}
			\qquad (i = 1,2)
		\end{equation*}
		with $p_i$ a smooth morphism and $z_i$ a closed immersion, the diagram of functors $\Hbb(Z_1) \times \Hbb(Z_2) \rightarrow \Hbb(P_1 \times P_2)$
		\begin{equation*}
			\begin{tikzcd}
				p_1^* z_{1,*} A_1 \boxtimes p_2^* z_{2,*} A_2 \arrow{r}{m^{sm}} \arrow[equal]{d} & (p_1 \times p_2)^*(z_{1,*} A_1 \boxtimes z_{2,*} A_2) \arrow{r}{\bar{m}^{cl}} & (p_1 \times p_2)^* (z_1 \times z_2)_* (A_1 \boxtimes A_2) \arrow[equal]{d} \\
				z'_{1,*}* {p'_1}^* A_1 \boxtimes z'_{2,*} {p'_2}^* A_2 \arrow{r}{\bar{m}^{cl}} & (z'_1 \times z'_2)_* ({p'_1}^* A_1 \boxtimes {p'_2}^* A_2) \arrow{r}{m^{sm}} & (z'_1 \times z'_2)_* (p'_1 \times p'_2)^* (A_1 \boxtimes A_2)
			\end{tikzcd}
		\end{equation*}
		is commutative.
		\item[(T-ETS-1)] For every pair of commutative triangles in $\Scal$
		\begin{equation*}
			\begin{tikzcd}
				Q_i \arrow{r}{h_i} \arrow{dr}{q_i} & P_i \arrow{d}{p_i} \\
				& S_i 
			\end{tikzcd}
			\qquad (i = 1,2)
		\end{equation*}
		with $p_i$, $q_i$ smooth morphisms and $h_i$ a closed immersion, the diagram of functors $\Hbb_1(S_1) \times \Hbb(S_2) \rightarrow \Hbb(Q_1 \times Q_2)$
		\begin{equation*}
			\begin{tikzcd}
				q_1^* A_1 \boxtimes q_2^* A_2 \arrow{rr}{m^{sm}} \arrow[equal]{d} && (q_1 \times q_2)^* (A_1 \boxtimes A_2) \arrow[equal]{d} \\
				h_1^* p_1^* A_1 \boxtimes h_2^* p_2^* A_2  \arrow{r}{m^{sm}} & (h_1 \times h_2)^* (p_1^* A_1 \boxtimes p_2^* A_2) \arrow{r}{m^{cl}} & (h_1 \times h_2)^* (p_1 \times p_2)^* (A_1 \boxtimes A_2)
			\end{tikzcd}
		\end{equation*}
		is commutative.
	\end{enumerate}
\end{defn}

We can directly state the first main result of this section:

\begin{thm}\label{thm_ETScore}
	Let $\Hbb$ be a localic $\Scal$-fibered category.
	Then every localic external tensor structure on $\Hbb$ determines by restriction and adjunction an external tensor core; conversely, every external tensor core on $\Hbb$ uniquely extends to a localic external tensor structure.
\end{thm}
\begin{proof}
	This can be deduced from \Cref{prop_mor_ETSkel} in the same way as \Cref{prop_mor-core} can be deduced from \Cref{prop_mor-skel}. We leave the details to the interested reader.
\end{proof}

\subsection{Compatibility with associativity and commutativity constraints}

As done for external tensor skeleta in the previous section, we show that the axioms of external associativity and commutativity constraints can be checked on the level of external tensor cores.

\begin{defn}\label{defn:ETCore-asso-comm}
	Let $(\boxtimes,m^{sm},\bar{m}^{cl})$ be an external tensor core on $\Hbb$.
	\begin{enumerate}
		\item An \textit{external associativity constraint} $a$ on $(\boxtimes,m^{sm},\bar{m}^{cl})$ is the datum of
		\begin{itemize}
			\item a triangulated external associativity constraint $a^{sm}$ on $(\boxtimes,m^{sm})$,
			\item a triangulated external associativity constraint $a^{cl}$ on $(\boxtimes,\bar{m}^{cl})$
		\end{itemize}
		such that, for every $S_1, S_2, S_3 \in \Scal$, the two natural isomorphisms of functors $\Hbb(S_1) \times \Hbb(S_2) \times \Hbb(S_3) \rightarrow \Hbb(S_1 \times S_2 \times S_3)$
		\begin{equation*}
			a^{sm} = a^{sm}_{S_1,S_2,S_3}, \; a^{cl} = a^{cl}_{S_1,S_2,S_3}: (A_1 \boxtimes A_2) \boxtimes A_3 \xrightarrow{\sim} A_1 \boxtimes (A_2 \boxtimes A_3)
		\end{equation*}
		coincide, so that we can indicate both of them with the same symbol $a = a_{S_1,S_2,S_3}$.
		\item An \textit{external commutativity constraint} $c$ on $(\boxtimes,m^{sm},\bar{m}^{cl})$ is the datum of
		\begin{itemize}
			\item a triangulated external commutativity constraint $c^{sm}$ on $(\boxtimes,m^{sm})$
			\item a triangulated external commutativity constraint $c^{cl}$ on $(\boxtimes,\bar{m}^{cl})$
		\end{itemize}
		such that, for every $S_1, S_2 \in \Scal$, the two natural isomorphisms of functors $\Hbb(S_1) \times \Hbb(S_2) \rightarrow \Hbb(S_1 \times S_2)$
		\begin{equation*}
			c^{sm} = c^{sm}_{S_1,S_2}, \; c^{cl} = c^{cl}_{S_1,S_2}: A_1 \boxtimes A_2 \xrightarrow{\sim} \tau^*(A_2 \boxtimes A_1)
		\end{equation*}
		coincide, so that we can indicate both of them with the same symbol $c = c_{S_1,S_2}$.
	\end{enumerate}
\end{defn}

\begin{lem}
	Let $(\boxtimes,m^{sm},\bar{m}^{cl})$ be an external tensor core on $\Hbb$. Then:
	\begin{enumerate}
		\item Giving an external associativity constraint on $(\boxtimes,m^{sm},\bar{m}^{cl})$ is equivalent to giving a triangulated external associativity constraint on the corresponding triangulated external tensor structure $(\boxtimes,m)$.
		\item Giving an external commutativity constraint on $(\boxtimes,m^{sm},\bar{m}^{cl})$ is equivalent to giving a triangulated external commutativity constraint on the corresponding triangulated external tensor structure $(\boxtimes,m)$.
	\end{enumerate}
\end{lem}
\begin{proof}
	\begin{enumerate}
		\item In view of \Cref{lem_constr-skel}(1), it suffices to show that giving an external associativity constraint on $(\boxtimes,m^{cl})$ is equivalent to giving an external associativity constraint on $(\boxtimes,\bar{m}^{cl})$. This follows from \Cref{lem:adj-asso-comm}(1).
		\item In view of \Cref{lem_constr-skel}(2), it suffices to show that giving an external commutativity constraint on $(\boxtimes,m^{cl})$ is equivalent to giving an external commutativity constraint on $(\boxtimes,\bar{m}^{cl})$. This follows from \Cref{lem:adj-asso-comm}(2).
	\end{enumerate}
\end{proof}

\subsection{External tensor cores on morphisms}

We conclude by discussing the notion of external tensor core on morphisms of $\Scal$-fibered categories. 

We fix two localizing triangulated $\Scal$-fibered categories as well as a morphism of $\Scal$-fibered categories $R: \Hbb_1 \rightarrow \Hbb_2$, that we assume to be smooth-adjointable in the sense of \Cref{defn:geom}(2); moreover, we suppose that $\Hbb_1$ and $\Hbb_2$ are endowed with external tensor structures $(\boxtimes_1,m_1)$ and $(\boxtimes_2,m_2)$, respectively, that we assume to be smooth-adjointable in the sense of \Cref{defn:ETSkel}.

\begin{defn}\label{defn:ETCore-rho}
	An \textit{external tensor core} $(\rho^{sm},\rho^{cl})$ on $R$ (with respect to $(\boxtimes_1,m_1)$ and $(\boxtimes_2,m_2)$) is the datum of
	\begin{itemize}
		\item a triangulated external tensor structure $\rho^{sm}$ on $R: \Hbb_1^{sm} \rightarrow \Hbb_2^{sm}$ (with respect to $(\boxtimes_1,m_1^{sm})$ and $(\boxtimes_2,m_2^{sm})$),
		\item a triangulated external tensor structure $\rho^{cl}$ on $R: \Hbb_1^{cl} \rightarrow \Hbb_2^{cl}$ (with respect to $(\boxtimes,\bar{m}_1^{cl})$ and $(\boxtimes,\bar{m}_2^{cl})$)
	\end{itemize}
	such that, for every $S_1, S_2 \in \Scal$, the two natural isomorphisms of functors $\Hbb_1(S_1) \times \Hbb_1(S_2) \rightarrow \Hbb_2(S_1 \times S_2)$
	\begin{equation*}
		\rho^{sm} = \rho^{sm}_{S_1,S_2}, \; \rho^{cl} = \rho^{cl}_{S_1,S_2}: R_{S_1}(A_1) \boxtimes_2 R_{S_2}(A_2) \xrightarrow{\sim} R_{S_1 \times S_2}(A_1 \boxtimes_1 A_2)
	\end{equation*}
	coincide, so that we can write both of them with the same symbol $\rho = \rho_{S_1,S_2}$.
\end{defn}

We conclude by stating the second main result of this section:

\begin{thm}\label{thm:ETScore-mor}
	Giving an external tensor core on the morphism $R$ (with respect to $(\boxtimes_1,m_1)$ and $(\boxtimes_2,m_2)$) is the same as giving a  triangulated external tensor structure on $R$ (with respect to $(\boxtimes_1,m_1)$ and $(\boxtimes_2,m_2)$).
\end{thm}
\begin{proof}
	In view of \Cref{prop_mor_ETSkel}, it suffices to show that giving an external tensor core on $R$ is the same as a geometric triangulated external tensor skeleton on $R$. To this end, it suffices to show that the notion of external tensor structure on $R: \Hbb_1^{cl} \rightarrow \Hbb_2^{cl}$ is the same regardless of whether it is regarded as a morphism of $\Scal$-fibered categories or as a morphism of $\Scal^{op}$-fibered categories. This follows easily from \Cref{lem:rho-adj}.
\end{proof}

\begin{rem}\label{rem:ETCore-ass-comm}
	The analogue of \Cref{rem:ETSkel-ass-comm} holds: the associativity and commutativity conditions for external tensor structures on morphisms from \cite[\S~8]{Ter24Fib} admit obvious equivalent formulations in the setting of external tensor cores.
\end{rem}

\end{document}